\documentclass[11pt]{article}
\usepackage[T1]{fontenc}
\usepackage{lmodern}
\usepackage[utf8]{inputenc}
\usepackage{amsmath,amssymb,amsthm,mathtools}
\usepackage{graphicx,enumerate,cite,authblk}
\usepackage{microtype,needspace}
\usepackage{xcolor}
\usepackage[colorlinks=true,linkcolor=blue!45!black,citecolor=blue!45!black,urlcolor=blue!45!black]{hyperref}
\numberwithin{equation}{section}
\usepackage{amsfonts,mathrsfs,etoolbox}
\input{ursfs.fd}
\DeclareFontShape{U}{rsfs}{m}{n}{<-6.5> rsfs5 <6.5-8> rsfs7 <8-> rsfs10}{}

\theoremstyle{plain}
\newtheorem{thm}{Theorem}[section]
\newtheorem{pro}{Proposition}[section]

\newtheorem{lem}[thm]{Lemma}
\theoremstyle{definition}

\theoremstyle{remark}
\newtheorem{Remark}{Remark}[section]
\AtBeginEnvironment{Remark}{\pushQED{\qed}}
\AtEndEnvironment{Remark}{\popQED}

\newcommand{\U}{\mathcal{U}}

\numberwithin{figure}{section}
\allowdisplaybreaks[4]
\hypersetup{
  bookmarksdepth=subsection,
  pdftitle={Global dynamics of even solutions to the generalized Boussinesq equation},
  pdfauthor={Xiaoguang Li, Jun Wang, and Jian Zhang},
  pdfsubject={Global dynamics for the generalized Boussinesq equation},
  pdfkeywords={Generalized Boussinesq equation, ground-state energy, standing waves, one-pass theorem, global classification}}

\makeatletter
\newcommand{\statementspace}{\if@nobreak\else\Needspace{3\baselineskip}\fi}
\makeatother
\BeforeBeginEnvironment{thm}{\statementspace}
\BeforeBeginEnvironment{pro}{\statementspace}
\BeforeBeginEnvironment{lem}{\statementspace}

\begin{document}
\title{Global dynamics of even solutions\\to the generalized Boussinesq equation}
\author[1]{Xiaoguang Li}
\author[2]{Jun Wang\thanks{Corresponding author.}}
\author[3]{Jian Zhang}
\renewcommand{\Affilfont}{\small}
\affil[1]{School of Mathematical Sciences, Sichuan Normal University, Chengdu 610068, P.R. China\\\href{mailto:lixgmath@163.com}{lixgmath@163.com}}
\affil[2]{School of Mathematical Sciences, Fudan University, Shanghai 200433, P.R. China\\\href{mailto:wangjunmath1996@163.com}{wangjunmath1996@163.com}}
\affil[3]{School of Mathematical Sciences, University of Electronic Science and Technology of China, Chengdu 611731, P.R. China\\\href{mailto:zhangjiancdv@sina.com}{zhangjiancdv@sina.com}}
\date{}
\maketitle

\begin{abstract}
We study the generalized good Boussinesq equation with the focusing nonlinearity
$|u|^pu$ in the energy space. For $p>6-\varrho$,
where $\varrho>0$ is a fixed constant,
we classify even solutions with energy at most slightly above that of the ground state.
In each time direction, the solution either scatters in the energy space,
blows up in finite time, or is eventually trapped near a standing wave. All nine ordered pairs of
forward and backward behaviors occur for initial data in every neighborhood of the standing wave.
In the trapped case the solution
enters the corresponding local center-stable or center-unstable
manifold. For solutions that are not eventually
trapped, it remains to distinguish scattering from finite-time blow-up.
The proof of this dichotomy relies on two main ingredients.
First, we classify even solutions with energy below the ground-state
energy by identifying two invariant regions corresponding to scattering
and blow-up. Second, we establish a one-pass theorem in one dimension.
Once a solution starting sufficiently close to a standing wave leaves
a suitably chosen neighborhood of the standing waves, it cannot return.
Its variational sign remains fixed after the first exit.
A further compactness and perturbation argument yields scattering in the
positive-sign channel, whereas a regularized concavity argument yields
finite-time blow-up in the negative-sign channel. Together, these results give a global classification of even solutions in this energy range.


\end{abstract}

\clearpage
\tableofcontents
\section{Introduction}
\subsection{Model and setting}
\indent
\par
The classical Boussinesq equation was introduced to describe the
propagation of small-amplitude surface waves \cite{ref-Bo1872}.
Related models also arise in the study of nonlinear waves in elastic
media \cite{ref-Wh1974}, including longitudinal waves in elastic rods
\cite{ref-GKS2019}. In these models, a balance between dispersion and nonlinearity
can give rise to the phenomenon of solitary waves. To study their role in the long-time dynamics,
we consider the generalized Boussinesq equation
with a focusing power nonlinearity
\begingroup\renewcommand*{\theHequation}{model-GBQ}
\begin{equation}\tag{GBQ}\label{eq01}
\partial_t^2u-\partial_x^2u+\partial_x^4u+\partial_x^2(|u|^pu)=0,
\qquad (t,x)\in I\times\mathbb R,
\end{equation}
\endgroup
where $p>0$ and $u$ is real-valued with maximal lifespan $I=(T^-_{\max},T^+_{\max})$.

With $v=\partial_x^{-1}\partial_tu\in L^2(\mathbb R)$, equation
\eqref{eq01} takes the equivalent first-order form
\begingroup\renewcommand*{\theHequation}{model-HGBQ}
\begin{equation}\tag{HGBQ}\label{eq02}
\begin{cases}
\partial_tu=\partial_xv,\\
\partial_tv=\partial_x(-\partial_x^2u+u-|u|^pu).
\end{cases}
\end{equation}
\endgroup
System \eqref{eq02} conserves the energy and momentum, given respectively by
\begin{equation}
\label{eq03}
E(u(t),v(t))=\frac12\int_{\mathbb R}\bigl(v^2+u^2+(\partial_xu)^2\bigr)\,dx-\frac{1}{p+2}\int_{\mathbb R}|u|^{p+2}\,dx=E(u_0,v_0),
\end{equation}
\begin{equation}
\label{eq04}
\mathcal P(u(t),v(t))=\int_{\mathbb R}uv\,dx=\mathcal P(u_0,v_0).
\end{equation}
The energy space for $(u,v)$ is $X^1=H^1(\mathbb R)\times L^2(\mathbb R)$,
which corresponds to $H^1(\mathbb R)\times\dot H^{-1}(\mathbb R)$
for the scalar variables $(u,\partial_tu)$.
If $u$ is even in $x$, then $\partial_x v=\partial_t u$ is even, so the condition $v(t,\cdot)\in L^2(\mathbb R)$ forces $v$ to be odd. We call such a pair $(u,v)$ an even--odd solution of \eqref{eq02} and write
$X^1_{\mathrm{eo}}=H^1_e(\mathbb R)\times L^2_o(\mathbb R)$ for the
corresponding subspace of $X^1$.

Writing $\vec u=(u,v)^T$, we express \eqref{eq02} in Hamiltonian form as
\begin{equation*}
\partial_t\vec u=J\mathcal A\vec u+J\mathcal N(\vec u),\qquad
J\mathcal N(\vec u)=\bigl(0,-(p+1)|u|^p\partial_xu\bigr)^T,
\end{equation*}
with
\begin{equation*}
J=\left(\begin{matrix}
                 0 & \partial_{x} \\
                 \partial_{x} & 0
               \end{matrix}\right),\;
\mathcal{A}=\left(\begin{matrix}
                 -\partial^{2}_{x}+1 & 0 \\
                 0 & 1
               \end{matrix}\right),\;
\mathcal{N}(\vec{u})=\left(\begin{matrix}
                 -|u|^{p}u \\
                 0
               \end{matrix}\right).
\end{equation*}
The operator $J\mathcal{A}$ generates the strongly continuous unitary group $\U(t)=e^{tJ\mathcal A}$ on $X^{1}$, given explicitly by the Fourier multiplier
\begin{equation*}
\U(t)\vec{u}_{0}=\mathcal{F}^{-1}\left(\begin{matrix}
                 \cos(t\xi\sqrt{1+\xi^{2}}) & \frac{i}{\sqrt{1+\xi^{2}}}\sin(t\xi\sqrt{1+\xi^{2}})\\
                 i\sqrt{1+\xi^{2}}\sin(t\xi\sqrt{1+\xi^{2}}) & \cos(t\xi\sqrt{1+\xi^{2}})
               \end{matrix}\right)\mathcal{F}\binom{u_{0}}{v_{0}},
\end{equation*}
and satisfies the identity
\begin{equation*}
\langle\U(t)\vec{u},\vec{v}\rangle_{X^{1}}=\langle\vec{u},\U(-t)\vec{v}\rangle_{X^{1}},\;\forall \vec{u},\vec{v}\in X^{1},\; \forall t\in\mathbb{R}.
\end{equation*}

For every $p>0$, the Cauchy problem for \eqref{eq02} is locally
well-posed in $X^1$ \cite[Theorem~1.1]{ref-Ch2021}. Each
$\vec u_0\in X^1$ generates a unique maximal solution
$\vec u\in C(I;X^1)$ with $\vec u(0)=\vec u_0$, which depends
continuously on the initial data on compact subintervals of $I$.
The flow conserves $E$ and $\mathcal P$, and the local existence time
can be chosen uniformly on bounded sets of initial data
\cite[Lemma~1.1 and Proposition~1.2]{ref-Liu1993}.
Thus, at any finite endpoint $T$ of $I$,
\[
\lim_{\substack{t\to T\\ t\in I}}\|\vec u(t)\|_{X^1}=\infty.
\]
In scalar variables, Linares \cite[Theorem~5.2]{ref-Li1993} proved global existence for small initial data. For sufficiently high powers, small, spatially localized
data also yield scattering \cite{ref-Li1995,ref-Liu1997}. Below the ground-state
energy, Liu \cite{ref-Liu1995} used invariant sets to give criteria for
global existence and finite-time blow-up, with the additional assumption
$\partial_x^{-1}u_0\in L^2(\mathbb R)$ in the blow-up case. Chen, Guo
and Shao \cite{ref-Ch2021} extended these criteria to higher dimensions
for energy-subcritical powers. For radial data in dimensions $d\geq3$ and powers
$4/d\leq p<4/(d-2)$, Chen and Guo \cite{ref-Ch2023} proved a dichotomy
between scattering and finite-time blow-up below the ground-state energy.

System (\ref{eq02}) admits a family of solitary waves (traveling waves)
\begin{equation*}\vec{Q}_{c}(x-ct)=(Q_{c}(x-ct), -cQ_{c}(x-ct))^{T},\end{equation*}
where
\begin{equation*}Q_{c}=[\frac{(p+2)(1-c^{2})}{2}]^{\frac{1}{p}}\operatorname{sech}^{\frac{2}{p}}(\frac{p\sqrt{1-c^{2}}}{2}x)\end{equation*}
denotes the unique positive even $H^1(\mathbb R)$ solution of the elliptic equation
\begin{equation}
\label{eq06}
-\partial^{2}_{x} Q_{c}+(1-c^{2})Q_{c}- Q_{c}^{p+1}=0, \quad c^{2}<1.
\end{equation}
 Solitary waves play an important role in the study of dynamics and have attracted considerable attention. Using the Grillakis--Shatah--Strauss framework \cite{ref-Gr1987, ref-Gr1990}, Bona--Sachs \cite{ref-Bo1988} proved orbital stability for $0<p<4$ and $\frac{p}{4}<c^{2}<1$. In contrast, Liu \cite{ref-Liu1993} proved that all traveling waves are orbitally unstable when either $0<p<4$ with $c^{2}<\frac{p}{4}$, or $p\geq4$. At the critical speed $c^2=p/4$ for $0<p<4$, Li, Ohta, Wu, and Xue
\cite{ref-Wu2020} proved orbital instability. Liu \cite[Theorem~4.3]{ref-Liu1995} proved that the ground-state standing
wave $\vec Q=\vec Q_c\vert_{c=0}$ is strongly unstable by blow-up.
Liu, Ohta, and Todorova \cite[Theorem~1.6]{ref-Liu2007} later established
the same conclusion for traveling waves with $0<c^2<p/(2p+4)$. For $p\geq1$, Maul\'{e}n
\cite[Theorems~1 and~2]{ref-Ma2023} showed that the even--odd initial data whose forward
orbits stay near the standing wave $\vec Q$ form a local center-stable
manifold $\mathcal M$. For each such solution, the first component $u(t)$
converges uniformly to $Q$ on compact spatial intervals as $t\to+\infty$.

\subsection{Main results}
We are interested in the global dynamics of solutions that start near a
standing wave. More generally, can one classify all solutions with energy
at most slightly above the ground-state energy?

For the nonlinear Klein--Gordon equation (NLKG) and the nonlinear Schr\"odinger equation (NLS), this question has been answered in several settings.
In three dimensions, Nakanishi and Schlag treated the focusing cubic NLKG,
first for radial data \cite{ref-NSradial2011} and then for general data
\cite{ref-NSnonradial2012}, as well as the focusing cubic NLS for radial
data \cite{ref-Na2012}. Krieger, Nakanishi and Schlag \cite{ref-KNS2012}
obtained the corresponding result for even solutions of the one-dimensional
focusing NLKG. In each time direction, a solution scatters, blows up in
finite time, or is eventually trapped near a solitary wave.
The proofs combine sub-threshold scattering in the positive-sign region
with a one-pass theorem~\cite{ref-Nak2011}. With the localized virial errors
controlled, unstable-mode estimates and variational
bounds rule out a return after ejection from a neighborhood of the solitary
waves. The variational sign then stays fixed and distinguishes scattering
from blow-up. For NLKG, the localized virial argument uses finite propagation
speed to bound the energy outside a light cone. For radial NLS in three
dimensions, uniform pointwise decay at spatial infinity controls the
nonlinear part of the virial error.

Motivated by these results, we develop the global dynamics of even solutions
of \eqref{eq01} in this energy range. The main difficulty here is that the Boussinesq equation
has no finite speed of propagation and that even symmetry on the line
does not give the required pointwise decay at spatial infinity. Our proof therefore focuses on overcoming this
loss of spatial control in the one-pass argument and in the rigidity step
for scattering below the ground-state energy. To this end, we prove that
the scattering threshold for general data in the positive-sign region
lies strictly above half the ground-state energy. As a consequence, even
at the ground-state energy, a pair of reflected profiles escaping in
opposite directions cannot obstruct scattering, and we can use a virial
estimate centered at the origin in the positive-sign one-pass argument
(see Section~\ref{sec:proof-outline}).
Our main theorem also shows that in each time direction the solution either
scatters, blows up in finite time, or is eventually trapped near a standing wave.

\begin{thm}[Global classification of even solutions]
\label{thm002}
There is a constant $\varrho>0$ such that the following holds
for every real $p>6-\varrho$. There exist $\varepsilon_E,\delta>0$ such that, for every
$\vec u_0\in X^1_{\mathrm{eo}}$ satisfying
\[
E(\vec u_0)<E(\vec Q)+\varepsilon_E^2,
\]
exactly one of the following alternatives holds forward in time:
\begin{enumerate}
\renewcommand{\labelenumi}{\textup{(\arabic{enumi})}}
\item The solution is global forward and scatters: for some
$\vec u_+\in X^1_{\mathrm{eo}}$,
\[
\|\vec u(t)-\U(t)\vec u_+\|_{X^1}\longrightarrow0
\qquad(t\to+\infty).
\]
\item The solution blows up in finite positive time:
$T^+_{\max}<\infty$ and
\[
\|\vec u(t)\|_{X^1}\longrightarrow\infty
\qquad(t\uparrow T^+_{\max}).
\]
\item The solution is global forward and eventually trapped near one
standing wave. More precisely, there are $t_0\geq0$ and
$\varsigma\in\{1,-1\}$ such that
\[
\varsigma\vec u(t_0)\in\mathcal M,\qquad
\sup_{t\geq t_0}\|\vec u(t)-\varsigma\vec Q\|_{X^1}\leq\delta.
\]
For every compact interval $I\subset\mathbb R$,
\[
\|u(t)-\varsigma Q\|_{L^2(I)}
+\|u(t)-\varsigma Q\|_{L^\infty(I)}\longrightarrow0
\qquad(t\to+\infty).
\]
\end{enumerate}
Applying the forward statement to the time-reversed solution $\mathcal T\vec u(-t)$,
where $\mathcal T(u,v)=(u,-v)$, gives the same alternatives backward in time
with $\mathcal M$ replaced by $\mathcal T\mathcal M$.
\end{thm}

In the trapped case, the local convergence of $u(t)$ follows directly
from Maul\'{e}n's asymptotic stability theorem \cite[Theorem~1]{ref-Ma2023}.
Maul\'{e}n \cite[Theorem~2]{ref-Ma2023} characterizes the stable dynamics
on $\mathcal M$ near $\vec Q$. Theorem~\ref{thm002} complements this
result by classifying the long-time behavior of unstable solutions
that start sufficiently close to $\vec Q$ but off $\mathcal M$.
They leave a fixed neighborhood of the standing waves, and the
one-pass theorem rules out their return. For solutions that leave this neighborhood or never enter it,
the variational sign determines whether they scatter or blow up in finite
time. Time reversal gives the classification in both time directions
throughout the energy region $E(\vec u_0)<E(\vec Q)+\varepsilon_E^2$.

\begin{Remark}[Realization of the nine behaviors]\label{rem:nine}
Let \(S,B,T\) denote scattering, finite-time blow-up and trapping,
respectively. For each \((\sigma_+,\sigma_-)\in\{S,B,T\}^2\), every
neighborhood of \(\vec Q\) in the energy space contains initial data for
a nonstationary even solution with forward behavior \(\sigma_+\) and
backward behavior \(\sigma_-\). Section~\ref{sec:nine} constructs all nine
types using the local center-stable manifold and its time reverse.
At energies strictly below \(E(\vec Q)\), only \((S,S)\) and \((B,B)\)
occur, by Theorem~\ref{thm003}.
\end{Remark}
\begin{Remark}
\label{Remaa001}
Without even--odd symmetry, slow traveling waves can start arbitrarily
close to $\vec Q$ with energy slightly above $E(\vec Q)$. They retain
their shape as their centers move away from the origin, so the trapped
alternative in Theorem~\ref{thm002} would have to allow moving solitary
waves. Strictly below $E(\vec Q)$, a single profile can carry the full
critical energy without a reflected partner. Ruling out a critical element
in the positive well would then require control of its spatial center,
which may drift beyond the reach of the virial estimate at the origin.
\end{Remark}
To prove Theorem~\ref{thm002}, we first classify even--odd solutions
below the ground-state energy $E(\vec Q)$.
For $\vec u=(u,v)\in X^1$, define
\[
K_1(\vec u):=\left.\frac{d}{d\lambda}E(\lambda\vec u)\right|_{\lambda=1}
=\|u\|_{H^1}^2+\|v\|_{L^2}^2-\|u\|_{L^{p+2}}^{p+2}.
\]
\begin{thm}[Sub-threshold dynamics]
\label{thm003}
Let $p>6-\varrho$, with $\varrho$ as in Theorem~\ref{thm002}. Define
\begin{equation*}
\mathcal E_{<Q}=\{\vec{u}\in X^{1}:\vec{u}\;\text{is even-odd},\; E(\vec{u})< E(\vec{Q})\}
\end{equation*}
and the sets
\begin{equation*}
\mathcal{K}^{+}=\{\vec{u}\in \mathcal E_{<Q}:K_{1}(\vec{u})>0\},\; \mathcal{K}^{-}=\{\vec{u}\in \mathcal E_{<Q}:K_{1}(\vec{u})<0\}.
\end{equation*}
Then the sets $\mathcal K^+$ and $\mathcal K^-$ are invariant under the flow
of \eqref{eq02}, and the following dichotomy holds.
\begin{enumerate}
\renewcommand{\labelenumi}{\textup{(\arabic{enumi})}}
\item If $\vec{u}_{0}\in\mathcal{K}^{+}$, then $\vec{u}(t)$ exists globally in time and scatters as $t\rightarrow\pm\infty$.
\item If $\vec{u}_{0}\in\mathcal{K}^{-}$, then $\vec{u}(t)$ blows up in finite time in both time directions.
\end{enumerate}
\end{thm}
The dichotomy between scattering and finite-time blow-up in
Theorem~\ref{thm003} is a central theme in the study of focusing
dispersive equations. Kenig--Merle \cite{ref-Ke2006,ref-Ke2008}
introduced the concentration--compactness and rigidity framework for
energy-critical NLS and wave equations, obtaining
sharp dynamical thresholds below the ground state. This approach was
subsequently developed for higher-dimensional NLS by Killip--Visan
\cite{ref-Ki2010}, for mass-supercritical, energy-subcritical NLS by
Holmer--Roudenko \cite{ref-Ho2008}, and for NLS with combined
nonlinearities by Miao--Xu--Zhao \cite{ref-Mi2013}. The generalized
Boussinesq equation shares the potential-well structure behind these
results. In one dimension, however, even-odd symmetry allows pairs of
reflected profiles to escape to spatial infinity, and this loss of
compactness must be addressed in the scattering argument.

Write $P(u,v)=u$ for the first component. In one dimension the linear
evolution $\U(t)$ satisfies the dispersive estimate
\cite{ref-Ch2021,ref-Gu2006}
\[
\begin{aligned}
\|P\U(t)\vec u\|_{\dot B^0_{q,2}}
\lesssim |t|^{1/q-1/2}\bigl(&
\|K(D)^{1/q-1/2}u\|_{\dot B^0_{q',2}}\\
&+\|K(D)^{1/q-1/2}\Lambda^{-1}v\|_{\dot B^0_{q',2}}\bigr),
\qquad 2\leq q\leq\infty,
\end{aligned}
\]
where $1/q+1/q'=1$, $D=|\partial_x|$,
$\Lambda=(1+D^2)^{1/2}$, and $K(D)=D\Lambda^{-1}$; see
Lemma~\ref{le301}. For $q>2$, the factor $K(D)^{1/q-1/2}$ is singular
at zero frequency. Thus the familiar decay of the propagator does not
immediately provide the estimates needed for arbitrary energy data.
We use homogeneous Sobolev embedding and a suitable choice of
spacetime exponents to compensate for this loss. The assumption
$p>6$ makes these estimates and the nonlinear interpolations compatible
on the full space $X^1$; see Remark~\ref{Re001}. For $p$ close to $6$,
Appendix~\ref{app:powers} uses a fixed pair of finite exponents.

The proof of Theorem~\ref{thm002} builds on the sub-threshold classification
in Theorem~\ref{thm003} and the one-pass theorem. Sections~2--5 give the proof for $p>6$.
Appendix~\ref{app:powers} obtains the endpoint estimate at $p=6$
and transfers its strict energy margin to nearby powers. This last
step gives the constant $\varrho$ in the two theorems.
A wave-packet example at the end of that appendix illustrates the
limitation of the direct estimate for $p<6$.
\subsection{Outline of the proofs}\label{sec:proof-outline}
\indent
\par
We first describe the concentration--compactness and rigidity argument
of Kenig and Merle \cite{ref-Ke2006,ref-Ke2008} as it applies to
sub-threshold scattering for \eqref{eq01}. We then explain how this
scattering result enters the one-pass argument.

\medskip\noindent\textit{Sub-threshold scattering.}
The concentration--compactness argument for Theorem~\ref{thm003}
requires particular care in one dimension. A profile decomposition
adapted to $\U(t)$ in $X^1_{\mathrm{eo}}$ may contain pairs of
reflected wave packets escaping to infinity. In contrast to the
higher-dimensional radial setting, where the Strauss estimate
\cite{ref-St1977} gives uniform pointwise decay at spatial infinity, evenness alone does not
confine these packets to a fixed center. Moreover, \eqref{eq01} has
neither Galilean nor Lorentz invariance. We therefore cannot pass to a frame in which a
general critical element has zero momentum, as is done for
NLS~\cite{ref-Du2008} and NLKG~\cite{ref-NSnonradial2012}.

Chen and Shao \cite[Theorem~3]{ref-CS2026} state that one-dimensional
even solutions in the positive well below the ground-state energy scatter
for $p>5$. For any nonzero even $f\in C_c^\infty(\mathbb R)$, the sequence
$(f(\cdot-n)+f(\cdot+n),0)$ has two profiles escaping in opposite
directions, so the spatial translations in their Lemma~8 cannot all be
set to zero as asserted in their Remark~3. This leaves a gap in their
proof of the one-dimensional scattering statement. Komada and Masaki \cite{ref-KM2025} encountered the same loss of compactness
for the one-dimensional focusing fourth-order NLS. They proved scattering
for even solutions in the positive well below the ground-state threshold,
assuming a quantitative relation between the thresholds for even and general data.

To account for such pairs, we distinguish the general-data and even--odd
scattering thresholds. For $p>6$, set
\[
(\tilde p,\tilde q)=\left(\frac{7p-6}{p},\frac{21p-18}{p-6}\right),
\qquad S(I)=L_t^{\tilde p}L_x^{\tilde q}(I\times\mathbb R).
\]
Define
\[
A_{\mathrm{gen}}(E)=\sup\{\|P\vec u\|_{S(\mathbb R)}:
\vec u_0\in X^1,\ E(\vec u_0)<E,\ K_1(\vec u_0)>0\},\]
\[A_{\mathrm{eo}}(E)=\sup\{\|P\vec u\|_{S(\mathbb R)}:
\vec u_0\in X^1_{\mathrm{eo}},\ E(\vec u_0)<E,\ K_1(\vec u_0)>0\},
\]
and put
\begin{equation}\label{eqe03}
E_c^{\mathrm{gen}}=\sup\{E>0:A_{\mathrm{gen}}(E)<\infty\},\qquad
E_c^{\mathrm{eo}}=\sup\{E>0:A_{\mathrm{eo}}(E)<\infty\}.
\end{equation}
These quantities measure uniform scattering bounds on each energy class. We assign norm $+\infty$ to a solution whose maximal lifespan is not all of $\mathbb R$.

To rule out escaping pairs, we prove the following strict bound
in Proposition~\ref{prop:strict-gap}. There is a constant $c_*\in(1/2,1)$, independent of $p>6$, such that
\begin{equation}\label{intro:gap}
E_c^{\mathrm{gen}}\geq c_*E(\vec Q)
>\frac12E(\vec Q).
\end{equation}
The key is a bounded dispersive functional whose derivative controls
the nonlinearity beyond half the ground-state energy. With
$D=|\partial_x|$, $\Lambda=(1+D^2)^{1/2}$, and $\mathcal H$ the
Hilbert transform, the real-linear isometry
\[
z=u-i\Lambda^{-1}\mathcal H v,\qquad
\|z\|_{H^1}^2=\|\vec u\|_{X^1}^2
\]
diagonalizes the system as
\[
i\partial_tz=D\Lambda z-D\Lambda^{-1}(|u|^pu).
\]
For $S_0(t)=e^{-itD\Lambda}$ and
$\mathscr D(z)=\|\partial_x|z|^2\|_{L^2}^2$, a change of variables
in the Fourier phase leads to the estimate
\[
\int_{\mathbb R}\mathscr D(S_0(s)z)\,ds
\leq\frac23\|\Lambda^{1/2}z\|_{L^2}^4.
\]
Thus the quartic functional and its normalization,
\[
\mathfrak M(z)=\frac12\int_{\mathbb R}\operatorname{sgn}(s)
\mathscr D(S_0(s)z)\,ds,\qquad
\mathfrak V(z)=\frac{\mathfrak M(z)}{\|\Lambda^{1/2}z\|_{L^2}^4},
\]
are well defined for nonzero $z$, and $\mathfrak V$ is bounded.
Decompose $D\Lambda^{-1/2}(|u|^pu)$ into its component parallel to
$\Lambda^{1/2}z$ in complex $L^2$ and its orthogonal complement.
The parallel contribution cancels exactly against differentiation of
the denominator, leaving
\[
\partial_t\mathfrak V(z(t))=
\frac{-\mathscr D(z(t))+\mathcal R_\perp(t)}
{\|\Lambda^{1/2}z(t)\|_{L^2}^4}.
\]
Complex orthogonality eliminates the cross term in the squared norm
of the variation of the two-variable product. The density and
fractional-chain estimates then yield
\[
|\mathcal R_\perp|\leq a_p\|z\|_{H^1}^{p}\mathscr D(z).
\]
At half the ground-state energy the potential-well bound leaves a
strict absorption margin, uniform for $p\geq6$. A continuous upper
bound for the energy-space norm preserves this margin on a slightly
larger energy interval. There we have
\[
|\mathcal R_\perp|\leq\theta_*\mathscr D(z),\qquad 0<\theta_*<1,
\]
with $\theta_*$ independent of $p>6$ by \eqref{gap:margin}. Integrating the resulting monotonicity inequality bounds $u$ in
$L_t^6L_x^\infty$. Combining this spacetime estimate with the energy bound, we obtain
$J\mathcal N(\vec u)\in L_t^1X^1$ and scattering in both time
directions. The quantitative estimates appear in Proposition~\ref{prop:strict-gap} and Appendix~\ref{app:gap}.

We then complete the even-odd argument through profile decomposition, reduction to a single critical element, and virial rigidity. If \(E_c^{\mathrm{eo}}<E(\vec Q)\), each constituent of an escaping pair has energy at most \(E_c^{\mathrm{eo}}/2\), hence lies strictly below the general scattering threshold. These constituents therefore scatter without any parity assumption, while centered profiles below \(E_c^{\mathrm{eo}}\) scatter by definition. Long-time perturbation then leaves a single centered even-odd critical element; its orbit is precompact in $X^1$ without a moving spatial center. The positive
potential well and compactness provide a uniform positive lower bound for the scaling functional $K_2$ of Section~\ref{sec:variational} on this orbit, which leads to a contradiction through a localized virial identity. This proves scattering for even-odd solutions in the invariant set $\mathcal{K}^{+}$ (see Proposition \ref{proo001}). A regularized concavity argument yields finite-time blow-up for solutions in $\mathcal{K}^{-}$ (see Proposition \ref{proo002}).

For general data in the positive well, Proposition~\ref{prop:strict-gap} gives a uniform scattering bound for energies below $c_*E(\vec Q)$.
Scattering throughout the nonsymmetric positive well below
$E(\vec Q)$ would require additional control of critical elements
with a moving spatial center and is left for future work.

\medskip\noindent\textit{Global dynamics and the one-pass argument.}
We first combine the local analysis near the standing waves with
the one-pass argument in Section~4. In contrast to NLKG, the GBQ
propagator has no finite speed of propagation, so localized virial
errors cannot be controlled by a light-cone argument. We treat the
two variational signs separately, using compactness in the positive
case and a regularized antiderivative in the negative case.

Writing $\vec u=\vec Q+\vec\eta$, the perturbed equation is
\[
\partial_t\vec\eta=J\mathcal L\vec\eta
+J\widetilde{\mathcal N}(\vec\eta),\qquad
\mathcal L=E''(\vec Q),
\]
where $\widetilde{\mathcal N}$ is the nonlinear remainder.
By the spectral decomposition in Proposition~\ref{le13}, we may write
\[
\vec\eta=\lambda_+\vec\phi^++\lambda_-\vec\phi^-+\vec\gamma,
\qquad \langle\mathcal L\vec\phi^\pm,\vec\gamma\rangle=0,
\]
with $J\mathcal L\vec\phi^\pm=\pm\mu\vec\phi^\pm$ and
\[
\langle\mathcal L\vec\gamma,\vec\gamma\rangle
\sim\|\vec\gamma\|_{X^1}^2.
\]
For $\lambda_1=(\lambda_++\lambda_-)/2$ and
$\lambda_2=(\lambda_+-\lambda_-)/2$, the energy expansion reads
\[
E(\vec u)-E(\vec Q)+2\lambda_2^2
=\lambda_1^2+\lambda_2^2+\frac12\langle\mathcal L\vec\gamma,\vec\gamma\rangle
+O(\|\vec\eta\|_{X^1}^3)\sim\|\vec\eta\|_{X^1}^2.
\]
The square root of the energy expression on the left defines a continuous distance $d_Q$ near both standing waves,
extended by a smooth cutoff so that
\[
d_Q(\vec u)\sim d_0(\vec u)
:=\min_{\sigma=\pm1}\|\vec u-\sigma\vec Q\|_{X^1}.
\]
The full construction is given in
\eqref{distance:coordinates}--\eqref{distance:equivalence}; in the
ejection region the cutoff is one and
$d_Q^2=E(\vec u)-E(\vec Q)+2\lambda_2^2$.
Near $\pm\vec Q$, the amplitude and spatial-scaling derivatives $K_1$ and $K_2$ satisfy
\[
E(\vec u)-E(\vec Q)+CK_j(\vec u)^2\gtrsim d_Q(\vec u)^2,
\qquad j=1,2.
\]
An amplitude-scaling argument shows that $K_1$ and $K_2$ have the
same sign whenever
\[
|E(\vec u)-E(\vec Q)|\lesssim\varepsilon_{\mathrm{op}}^2\ll R^2
\lesssim d_Q(\vec u)^2.
\]
The hyperbolic structure also forces a solution leaving the
$R$-neighborhood to reach the larger $\delta^*$-neighborhood at an
exponential rate. Proposition~\ref{pro105} proves this ejection in
both time directions, with stable and unstable modes interchanged.
For data sufficiently close to either standing wave, the first exit
from the $R$-neighborhood of $\{\pm\vec Q\}$ is final. The common sign
of $K_1$ and $K_2$ at exit therefore persists for the rest of the
lifespan.

In the proof of the one-pass theorem (see Proposition \ref{pro001}), we use two virial functionals for the cases $K_{j}<0$ and $K_{j}>0$, respectively. Suppose that $t_{0}$ and $t_{1}$ denote the times at which the solution $\vec{u}$ first leaves and subsequently returns to the $R$-neighborhood of $\pm\vec{Q}$, respectively.

For the case $K_{1}\big(\vec{u}(t_{0})\big)<0$, we consider
\[
\mathcal{T}_{r}(u,v)(t)=\langle\partial_{x}^{-1}[u(t)-u_{r}(t_{0})],\, v(t)-v(t_{1})\rangle,
\]
where $u_r(t_0)=\Pi_r u(t_0)$ and $\Pi_r$ projects onto the frequencies $|\xi|\leq r$. The cutoff error vanishes when the
first component is either stationary profile. Near the standing waves
this cancellation bounds the error by $C\sqrt R\,d_Q(\vec u(t))$ for
$r=R$. Farther away, $-K_1$ is bounded below and absorbs the error.
Thus $\partial_t\mathcal T_R\geq -K_1/2$ throughout the return
interval. Integration over the initial ejection segment gives
$\delta^*\lesssim|\mathcal T_R(t_0)|$, whereas the endpoint data give
$|\mathcal T_R(t_0)|\lesssim\sqrt R\ll\delta^*$.

For the case $K_{1}\big(\vec{u}(t_{0})\big)>0$, we consider
\[
\mathcal{I}_{\varphi}(u,v)(t)=\int_{\mathbb{R}}\varphi uv\,dx,
\]
where $\varphi(x)=\phi(Rx)/R$ and $\phi$ is the cutoff in
\eqref{eeqqq02}. The localized virial identity has leading term
$-K_2(\vec u)$. After estimating the cutoff errors, we obtain
\[
\partial_t\mathcal I_\varphi(\vec u(t))\leq
\begin{cases}
-c\,d_Q(\vec u(t)),&R<d_Q(\vec u(t))\leq\delta^*,\\
0,&d_Q(\vec u(t))>\delta^*,
\end{cases}
\]
on a hypothetical first-return interval, for suitable
$0<\varepsilon_{\mathrm{op}}\ll R\ll\delta^*$.
The nonpositivity bound away from the small solitary-wave neighborhood follows from concentration--compactness on the
first-return interval: $K_1>0$ supplies the energy-space bound, while the
final segment near $\pm\vec Q$ has length tending to infinity and
forces the scattering norm to diverge. The compactness consequence
is recorded in Lemma~\ref{lem:threshold-compactness}.
In the sub-threshold argument each escaping constituent has energy
strictly below $E(\vec Q)/2$. Here the limiting total energy is
$E(\vec Q)$, so a constituent can reach $E(\vec Q)/2$. The strict
inequality in \eqref{intro:gap} puts this endpoint inside the
general scattering range as well.
These bounds, together with the exponential growth rate of $d_{Q}(\vec{u}(t))$, lead to
\[
\delta^{*}\lesssim|\mathcal{I}_{\varphi}(u,v)(t_{1})-\mathcal{I}_{\varphi}(u,v)(t_{0})|\lesssim R,
\]
again a contradiction.

Now consider a solution with initial data as in Theorem~\ref{thm002}.
If it enters a sufficiently small
neighborhood of $\pm\vec Q$, we restart the local classification
there. If it never enters, the variational sign is fixed.
Proposition~\ref{prop:exterior-scattering} gives scattering in the positive region, while
regularized concavity gives finite-time blow-up in the negative
region. This proves the forward trichotomy. The transformation
$\vec u(t)\mapsto\mathcal T\vec u(-t)$ gives the backward statement
without a second dynamical argument.
\subsubsection*{Notation}
We collect here some standard notation used throughout the paper.
\begin{itemize}
\itemsep0em
\item For $\vec{u}=(u_1,u_2)^T\in X^0$, we define the component projections
\[
P\vec{u}=u_1,\qquad P^c\vec{u}=u_2.
\]
\item We write $\U(t)=e^{tJ\mathcal A}$ for the linear evolution of
the system. The scalar group used in the diagonalized equation is
$S_0(t)=e^{-it\omega(D)}$, with $\omega(D)=D\Lambda$.
\item For a function of $x$, $L^q$ denotes $L^q(\mathbb R)$. We write
$L_t^aL_x^b(I)$ for $L^a(I;L^b(\mathbb R))$ and omit $I$ when $I=\mathbb R$.
\item The reflection and translation operators are defined by
\[
\mathscr R(f,g)(x)=\big(f(-x), - g(-x)\big)^{T},\qquad T_{a}(f,g)(x)= \big(f(x - a),g(x - a)\big)^{T}.
\]
\item Let $X^0:=L^2(\mathbb R)\times L^2(\mathbb R)$ be endowed with the norm
\[
\|\vec{u}\|_{X^0}:=\bigl(\|u_1\|_{L^2}^2+\|u_2\|_{L^2}^2\bigr)^{1/2},\qquad \vec{u}=(u_1,u_2)^T\in X^0.
\]
\item Let $\langle\cdot,\cdot\rangle$ denote the real inner product on $X^0$:
\[
\langle \vec{u},\vec{v}\rangle:=\int_{\mathbb R}\bigl(u_1v_1+u_2v_2\bigr)\,dx,
\qquad \vec{u}=(u_1,u_2)^T,\ \vec{v}=(v_1,v_2)^T\in X^0.
\]
\item Let $X^1:=H^1(\mathbb R)\times L^2(\mathbb R)$, equipped with the norm
\[
\|\vec{u}\|_{X^1}:=\bigl(\|u_1\|_{H^1}^2+\|u_2\|_{L^2}^2\bigr)^{1/2},\qquad \vec{u}=(u_1,u_2)^T\in X^1.
\]
\item The fixed points of $\mathscr R$ in $X^1$ form the subspace $X^1_{\mathrm{eo}}$.
\item Let $\langle\cdot,\cdot\rangle_{X^1}$ denote the inner product on $X^1$ induced by $\mathcal A$:
\[
\langle \vec u,\vec v\rangle_{X^1}
:=\int_{\mathbb R}(u_1v_1+\partial_xu_1\partial_xv_1+u_2v_2)\,dx.
\]
We also use $\langle\cdot,\cdot\rangle$ for the dual pairing between
$(X^1)^*=H^{-1}\times L^2$ and $X^1$, so that
\[
\langle\vec u,\vec v\rangle_{X^1}=\langle\mathcal A\vec u,\vec v\rangle,
\qquad \vec u,\vec v\in X^1.
\]
\item For $\alpha>0$ and $\vec{w}\in X^1$, define the $\alpha$-neighborhood of $\vec{w}$ as
\[
U_\alpha(\vec{w}):=\{\vec{u}\in X^1:\|\vec{u}-\vec{w}\|_{X^1}<\alpha\}.
\]
We write $U_\alpha(\pm\vec Q)=U_\alpha(\vec Q)\cup U_\alpha(-\vec Q)$.
\item Let $\mathcal{F}$ and $\mathcal{F}^{-1}$ denote the Fourier transform and its inverse, respectively:
\[
\begin{gathered}
\mathcal{F}(f)(\xi)=\frac{1}{\sqrt{2\pi}}\int_{\mathbb{R}}e^{-i\xi x}f(x)\,dx,\qquad
\mathcal{F}^{-1}(f)(x)=\frac{1}{\sqrt{2\pi}}\int_{\mathbb{R}}e^{i\xi x}f(\xi)\,d\xi,\\
f\in L^1(\mathbb R)\cap L^2(\mathbb R).
\end{gathered}
\]
The transforms extend to all of $L^2$ by Plancherel's theorem.
The same normalization is used for the partial transforms
$\mathcal{F}_x$, $\mathcal{F}_t$, and $\mathcal{F}_s$.
\item For $a,b>0$, we write $a\lesssim b$ if there exists $C>0$ such that $a\le Cb$; and $a\sim b$ if $a\lesssim b\lesssim a$.
\item Throughout the paper, $C>0$ denotes a generic constant, whose value may change from line to line.
\end{itemize}
\section{Preliminaries}
\subsection{Spectral coordinates and the local center-stable manifold}\label{sec:spectral-local}
\indent
\par
We begin with the self-adjoint energy operator
\[
\mathcal{L}=\begin{pmatrix}
-\partial_x^2+1-(p+1)Q^p & 0 \\
0 & 1
\end{pmatrix} : X^{1}\to (X^{1})^{*}.
\]
For the spectral analysis, we use its self-adjoint realization on $X^0=L^2\times L^2$, with domain $H^2\times L^2$. The scalar part is determined by the ground state $Q$ of \eqref{eq06}. We then apply the Hamiltonian index theorem to the linearized operator $J\mathcal L$.

\begin{lem}[Spectrum of the energy Hessian]
\label{le11}
For any $0<p<\infty$, the following hold:
\begin{enumerate}
\item[\textup{(i)}] $\sigma_{\mathrm{ess}}(\mathcal L)=[1,\infty)$;
\item[\textup{(ii)}] $\mathcal{L}$ has exactly one simple negative eigenvalue;
\item[\textup{(iii)}] $\ker \mathcal{L}=\operatorname{span}\{\partial_x \vec{Q}\}$.
\end{enumerate}
\end{lem}
For the scalar essential spectrum, see \cite[Lemma~35]{ref-Ma2023}.
The kernel and negative eigenvalue are described in
\cite[Lemmas~3.1 and~3.2]{ref-Wu2020}.

\begin{pro}
\label{le13}
For any $0<p<\infty$, $J\mathcal{L}$ has exactly one simple negative eigenvalue $-\mu<0$ and exactly one simple positive eigenvalue $\mu>0$, corresponding to the even-odd eigenfunctions $\vec{\phi}^{-}$ and $\vec{\phi}^{+}$, respectively. Moreover, each $\vec u\in X^1_{\mathrm{eo}}$ has a unique decomposition
\[
\vec{u}=\lambda_{+}\vec{\phi}^{+}+\lambda_{-}\vec{\phi}^{-}+\vec{\gamma}
\]
such that
\[
\langle \mathcal{L}\vec{\phi}^{+},\vec{\gamma}\rangle=\langle \mathcal{L}\vec{\phi}^{-},\vec{\gamma}\rangle=0
\]
and
\[
\langle \mathcal{L}\vec{\gamma},\vec{\gamma}\rangle\gtrsim\|\vec{\gamma}\|^{2}_{X^{1}}.
\]
\end{pro}

\begin{proof}
Let $L=-\partial_x^2+1-(p+1)Q^p$ be the scalar part of $\mathcal L$.
By \cite[Theorem~40 and Corollary~41]{ref-Ma2023}, $-\partial_x^2 L$ has a simple negative
eigenvalue $-\nu_0^2$, with a real even eigenfunction $\phi_0$ normalized by
\[
-\partial_x^2 L\phi_0=-\nu_0^2\phi_0,\qquad
\|\partial_x^{-1}\phi_0\|_{L^2}=1,\qquad \nu_0>0.
\]
The eigenfunction and its first two antiderivatives decay exponentially \cite[Lemma~38]{ref-Ma2023}.
Integrating the eigenvalue equation gives
$\partial_x L\phi_0=\nu_0^2\partial_x^{-1}\phi_0$. Hence
\[
\vec Y_\pm=\binom{\phi_0}{\pm\nu_0\partial_x^{-1}\phi_0},
\qquad J\mathcal L\vec Y_\pm=\pm\nu_0\vec Y_\pm.
\]
By the Hamiltonian index theorem of Lin and Zeng \cite{ref-Lin2022},
the sum of the algebraic multiplicities of the positive real
eigenvalues is at most $n^-(\mathcal L)=1$ by Lemma~\ref{le11}.
Thus $\mu=\nu_0$ is the only positive real eigenvalue and is
algebraically simple. Since $\mathcal T J\mathcal L=-J\mathcal L\mathcal T$,
the same holds for $-\mu$ on the negative real axis.

For the energy coordinates we take
$\vec\phi^+=\vec Y_+/(\sqrt2\,\nu_0)$ and
$\vec\phi^-=-\vec Y_-/(\sqrt2\,\nu_0)$. If $\vec\phi^+=(f,g)^T$,
then $\vec\phi^-=(-f,g)^T$. Since $J$ is skew-adjoint,
$\langle\mathcal L\vec\phi^\pm,\vec\phi^\pm\rangle=0$. Thus
$\langle\mathcal L\vec\phi^+,\vec\phi^-\rangle=2\|g\|_{L^2}^2=1$
by the normalization of $\partial_x^{-1}\phi_0$. Consequently the
only possible remainder with the required orthogonality is
\[
\vec{\gamma}=\vec{u}-\langle \mathcal{L}\vec{\phi}^{-},\vec{u}\rangle\vec{\phi}^{+}
            -\langle \mathcal{L}\vec{\phi}^{+},\vec{u}\rangle\vec{\phi}^{-}.
\]
By construction,
\[
\langle \mathcal{L}\vec{\phi}^{+},\vec{\gamma}\rangle=\langle \mathcal{L}\vec{\phi}^{-},\vec{\gamma}\rangle=0.
\]
Now define
\[
\vec{v}=\langle \vec{\phi}^{+}, \vec{\zeta}^{-}\rangle\vec{\gamma}
      -\langle \vec{\gamma}, \vec{\zeta}^{-}\rangle\vec{\phi}^{+},
\]
where $\vec{\zeta}^{-}$ is the eigenfunction corresponding to the unique negative eigenvalue of $\mathcal{L}$. Since $\vec{u}$ is even-odd, both $\vec{\gamma}$ and $\vec{v}$ are even-odd, and hence
\[
\langle\vec{v},\vec{\zeta}^{-}\rangle=\langle\vec{v},\partial_x\vec{Q}\rangle=0.
\]
Spectral coercivity on the orthogonal complement of the negative and zero
modes then implies
\begin{equation}
\label{eq2013}
\langle \vec{\phi}^{+}, \vec{\zeta}^{-}\rangle^{2}\langle \mathcal{L}\vec{\gamma},\vec{\gamma}\rangle
= \langle \mathcal{L}\vec{v},\vec{v}\rangle
\gtrsim \|\vec{v}\|^{2}_{X^{1}}.
\end{equation}
Set $a=\langle\vec\phi^+,\vec\zeta^-\rangle$. If $a=0$, spectral
coercivity in the even--odd subspace would contradict
$\langle\mathcal L\vec\phi^+,\vec\phi^+\rangle=0$. Moreover, pairing
the definition of $\vec v$ with $\mathcal L\vec\phi^-$ gives
$\langle\vec\gamma,\vec\zeta^-\rangle=-\langle\mathcal L\vec\phi^-,\vec v\rangle$.
Thus $a\vec\gamma=\vec v-\langle\mathcal L\vec\phi^-,\vec v\rangle\vec\phi^+$,
so $\|\vec\gamma\|_{X^1}\lesssim\|\vec v\|_{X^1}$.
The desired coercivity follows from \eqref{eq2013}.
\end{proof}

\medskip\noindent\textit{The local center-stable manifold.}
We record Maul\'{e}n's graph in the coordinates used below.
Set
\[
\vec Z_\pm=\binom{-\partial_x^{-2}\phi_0}{\pm\nu_0^{-1}\partial_x^{-1}\phi_0}.
\]
Since $\langle\phi_0,\partial_x^{-2}\phi_0\rangle
=-\|\partial_x^{-1}\phi_0\|_{L^2}^2=-1$, we have
$\langle\vec Y_\pm,\vec Z_\pm\rangle=2$ and
$\langle\vec Y_\pm,\vec Z_\mp\rangle=0$.
Thus the coefficient of $\vec Y_\pm$ in a perturbation $\vec\eta$ is
$b_\pm=\tfrac12\langle\vec\eta,\vec Z_\pm\rangle$.
Our $\vec Z_\pm$ are the negatives of the vectors denoted by $\vec Z_\mp$
in \cite[Eq.~(10)]{ref-Ma2023}. The coefficients $b_\pm$ agree with
those in the graph construction of \cite[Section~7.2]{ref-Ma2023},
where the radiation component and $b_-$ are prescribed and $b_+$ is
chosen to keep the solution near $\vec Q$.

For $p\geq 1$ in our notation (the power in \cite{ref-Ma2023} is $p+1$),
Maul\'{e}n's graph theorem \cite[Theorem~2]{ref-Ma2023} gives constants
$\delta_0,C>0$ and a Lipschitz function $h:A_0\to\mathbb R$, where
\[
 A_0=\{\vec\epsilon\in X^1_{\mathrm{eo}} :
 \|\vec\epsilon\|_{X^1}<\delta_0,\ \langle\vec\epsilon,\vec Z_+\rangle=0\}.
\]
The local center-stable manifold has the form
\[
\mathcal M=\{\vec Q+\vec\epsilon+h(\vec\epsilon)\vec Y_+ : \vec\epsilon\in A_0\},
\qquad |h(\vec\epsilon)|\leq C\|\vec\epsilon\|_{X^1}^{3/2}.
\]
Solutions with initial data on $\mathcal M$ are global forward and satisfy
\[
\sup_{t\geq0}\|\vec u(t)-\vec Q\|_{X^1}\leq C\|\vec u_0-\vec Q\|_{X^1}.
\]
Conversely, if a global forward even--odd solution satisfies
$\sup_{t\geq0}\|\vec u(t)-\vec Q\|_{X^1}\leq\delta_0/2$, then
$\vec u(t)\in\mathcal M$ for every $t\geq0$.

\subsection{Dispersive and Strichartz estimates}
\indent
\par
We first collect the linear estimates used below. To distinguish a
time-integrability exponent from the nonlinear power $p$, we write
$(a,q)\in\Omega(\dot H^s)$ when
\[
\frac2a+\frac1q=\frac12-s,\qquad 2\leq a,q<\infty,
\qquad s<\frac12.
\]
We use the Fourier multipliers
\[
D=|\partial_x|=\mathcal{F}^{-1}|\xi|\mathcal{F},\qquad
\Lambda=(1+D^2)^{1/2},\qquad
\omega(D)=D\Lambda,\qquad K(D)=D\Lambda^{-1}.
\]
We also use the Hilbert transform $\mathcal H$, whose symbol is
$-i\operatorname{sgn}\xi$. Fractional powers are defined by their
Fourier symbols. In particular, $\mathcal H\partial_x=D$ and
$D\mathcal H=-\partial_x$.

\begin{lem}[\cite{ref-Ch2021,ref-Gu2006}]\label{le301}
The group $S_0(t)=e^{-it\omega(D)}$ satisfies the following estimates.

\noindent \textup{(i)} For $2\leq q\leq\infty$ and $1/q+1/q'=1$,
\begin{equation}\label{ee201}
\|S_0(t)f\|_{\dot B^0_{q,2}}
\lesssim |t|^{1/q-1/2}
\|K(D)^{1/q-1/2}f\|_{\dot B^0_{q',2}}.
\end{equation}
\noindent \textup{(ii)} For $(a,q)\in\Omega(L^2)$,
\begin{equation}\label{ee202}
\|S_0(t)f\|_{L_t^a\dot B^0_{q,2}}
\lesssim\|K(D)^{-1/a}f\|_{L^{2}}.
\end{equation}
\noindent \textup{(iii)} For $(a_1,q_1),(a_2,q_2)\in\Omega(L^2)$,
\begin{equation}\label{ee203}
\left\|\int_{\mathbb R}S_0(t-s)F(s)\,ds\right\|_{L_t^{a_1}\dot B^0_{q_1,2}}
\lesssim\|K(D)^{-1/a_1-1/a_2}F\|_{L_t^{a_2'}\dot B^0_{q_2',2}}.
\end{equation}
\end{lem}

\begin{lem}\label{le302}
The following estimates hold.

\noindent \textup{(i)} If $0\leq s<1/2$ and $(a,q)\in\Omega(\dot H^s)$, then
\begin{equation}\label{e202}
\|S_0(t)f\|_{L_t^aL_x^q}
\lesssim\|K(D)^{-1/a}D^sf\|_{L^{2}}.
\end{equation}
\noindent \textup{(ii)} For the same $(a,q)$ and $b=a/(1+as)$,
\begin{equation}\label{e203}
\left\|\int_{\mathbb R}S_0(t-\tau)F(\tau)\,d\tau\right\|_{L_t^aL_x^q}
\lesssim\|K(D)^{1/q-1/2}F\|_{L_t^{b'}L_x^{q'}}.
\end{equation}
Here $(b,q)\in\Omega(\dot H^{-s})$ and
$1/b'=1-1/a-s$. The same bound holds for integration over
$\tau<t$ or $\tau>t$, and for sources restricted to any time interval.

\noindent \textup{(iii)} For $(a,q)\in\Omega(L^2)$,
\begin{equation}\label{e204}
\left\|\int_{\mathbb R}S_0(t)F(t)\,dt\right\|_{L^{2}}
\lesssim\|K(D)^{-1/a}F\|_{L_t^{a'}L_x^{q'}}.
\end{equation}
\end{lem}
\begin{proof}
For (i), choose $r$ by $1/r=1/q+s$. Then
$(a,r)\in\Omega(L^2)$ and $2\leq r<\infty$. Apply
Lemma~\ref{le301}(ii) to $D^sf$, and use
$\dot B^0_{r,2}\hookrightarrow L^r$ and
$\dot W^{s,r}\hookrightarrow L^q$.

For (ii), apply the Besov embeddings to both sides of
\eqref{ee201} to obtain
\[
\|S_0(t-\tau)F(\tau)\|_{L^{q}}
\lesssim |t-\tau|^{1/q-1/2}
\|K(D)^{1/q-1/2}F(\tau)\|_{L^{q'}}.
\]
Since $1/2-1/q=2/a+s\in(0,1)$, Minkowski's inequality and the
Hardy--Littlewood--Sobolev inequality in time imply
\begin{equation}\label{e205}
\begin{aligned}
\left\|\int_{\mathbb R}S_0(t-\tau)F(\tau)\,d\tau\right\|_{L_t^aL_x^q}
&\lesssim\left\|\int_{\mathbb R}|t-\tau|^{-2/a-s}
\|K(D)^{1/q-1/2}F(\tau)\|_{L^{q'}}\,d\tau\right\|_{L_t^a}\\
&\lesssim\|K(D)^{1/q-1/2}F\|_{L_t^{b'}L_x^{q'}}.
\end{aligned}
\end{equation}
Restricting the nonnegative scalar integral to $\tau<t$ or $\tau>t$
proves the asserted one-sided bounds. The interval versions follow by
extending the source by zero.

For (iii), use the complex $L^2$ inner product, linear in the first
argument. By unitarity,
\begin{equation}\label{e206}
\left(\int_{\mathbb R}S_0(t)F(t)\,dt,g\right)_{L^2}
=\int_{\mathbb R}\bigl(F(t),S_0(-t)g\bigr)_{L^2}\,dt.
\end{equation}
Move $K(D)^{1/a}$ to the second factor and apply H\"older's inequality
and Lemma~\ref{le301}(ii). We obtain
\begin{equation}\label{e207}
\begin{aligned}
\left\|\int_{\mathbb R}S_0(t)F(t)\,dt\right\|_{L^{2}}
&\leq\sup_{\|g\|_{L^{2}}=1}
\|K(D)^{-1/a}F\|_{L_t^{a'}L_x^{q'}}
\|S_0(-t)K(D)^{1/a}g\|_{L_t^aL_x^q}\\
&\lesssim\|K(D)^{-1/a}F\|_{L_t^{a'}L_x^{q'}}.
\end{aligned}
\end{equation}
These calculations are first justified for smooth functions with frequency
support away from zero; the estimates then extend by density.
\end{proof}

The linear system has the component representation
\[
\begin{cases}
P\U(t)\vec u=\cos(t\omega(D))u
-\mathcal H\Lambda^{-1}\sin(t\omega(D))v,\\
P^c\U(t)\vec u=-\mathcal H\Lambda\sin(t\omega(D))u
+\cos(t\omega(D))v.
\end{cases}
\]
Thus the corresponding estimates for $\U(t)$ follow from
Lemma~\ref{le302} and the boundedness of the Hilbert transform on $L^q$ for $1<q<\infty$.
For the endpoint Besov estimates in Section~1.2, its symbol is smooth on each dyadic annulus, so the localized multipliers have uniformly bounded $L^1$ kernels.

We now turn to the scattering estimates. For $p>6$, set $s=\frac16$ and choose the $\dot H^{1/6}$-admissible pair
\[
(\tilde p,\tilde q)=\left(\frac{7p-6}{p},\frac{21p-18}{p-6}\right).
\]
For any time interval $I$, write
\[
S(I)=L_t^{\tilde p}L_x^{\tilde q}(I\times\mathbb R),\qquad
N(I)=L_t^{\tilde r'}L_x^{\tilde q'}(I\times\mathbb R),\qquad
\frac1{\tilde r'}=\frac56-\frac1{\tilde p}.
\]
Here $S(I)$ is a scalar function space; thus
$\|P\vec u\|_{S(I)}=\|u\|_{S(I)}$. We reserve $S=S(\mathbb R)$
and $N=N(\mathbb R)$ for the whole time axis. We also write
$I_+=[0,\infty)$ and $I_-=(-\infty,0]$.
A finite scattering norm, together with a uniform energy-space bound,
implies scattering on either time half-line.
\begin{lem}[Scattering criterion]
\label{le303}
Let $\sigma\in\{+,-\}$, and let $\vec u$ solve \eqref{eq02} on
$I_\sigma$. If
\[
\sup_{t\in I_\sigma}\|\vec u(t)\|_{X^1}<\infty,
\qquad \|P\vec u\|_{S(I_\sigma)}<\infty,
\]
then $\vec u$ scatters in that time direction: there is
$\vec u_\sigma\in X^1$ such that
\[
\lim_{t\to\sigma\infty}\|\vec u(t)-\U(t)\vec u_\sigma\|_{X^1}=0.
\]
\end{lem}

\begin{proof}
We prove forward scattering; the backward argument is identical.
Write $f(u)=|u|^pu$ and $A=\sup_{t\in I_+}\|\vec u(t)\|_{X^1}$.
Duhamel's formula is
\begin{equation}\label{e2008}
\vec u(t)=\U(t)\vec u_0+
\int_0^t \U(t-s)J\mathcal N(\vec u(s))\,ds.
\end{equation}
We first show that the Duhamel integral in the interaction representation
is Cauchy in $X^1$. For $t_2>t_1\geq0$, the dual Strichartz estimate of
Lemma~\ref{le302}, with the $L^2$-admissible pair $(6,6)$, gives
\[
\begin{aligned}
\left\|\int_{t_1}^{t_2}\U(-s)J\mathcal N(\vec u(s))\,ds\right\|_{X^1}
&\lesssim\|K(D)^{-1/6}\partial_xf(u)\|_{L_t^{6/5}L_x^{6/5}(I)}\\
&\quad+\|\partial_x\Lambda^{-1}K(D)^{-1/6}\partial_xf(u)\|_{L_t^{6/5}L_x^{6/5}(I)}\\
&\lesssim\|K(D)^{-1/6}\partial_xf(u)\|_{L_t^{6/5}L_x^{6/5}(I)},
\end{aligned}
\]
where $I=(t_1,t_2)$ and $\partial_x\Lambda^{-1}$ is a bounded Mihlin
multiplier. To control the low-frequency factor, write
\[
K(D)^{-1/6}=m(D)(1+D^{-1/6}),\qquad
m(\xi)=\frac{(1+\xi^2)^{1/12}}{1+|\xi|^{1/6}}.
\]
The symbol $m$ satisfies the Mihlin bounds. Since
$D^{-1/6}\partial_x$ is, up to a Hilbert transform, $D^{5/6}$,
the embedding $W^{1,6/5}\hookrightarrow\dot W^{5/6,6/5}$ and
H\"older's inequality imply
\begin{equation}\label{e211}
\begin{aligned}
\|K(D)^{-1/6}\partial_xf(u)\|_{L^{6/5}}
&\lesssim\|f(u)\|_{L^{6/5}}+\|\partial_xf(u)\|_{L^{6/5}}\\
&\lesssim(\|u\|_{L^{2}}+\|\partial_xu\|_{L^{2}})\|u\|_{L^{3p}}^{p}
\lesssim\|u\|_{H^1}\|u\|_{L^{3p}}^{p}.
\end{aligned}
\end{equation}
Set
\[
\theta=\frac{35p-30}{6p^2},\qquad
m_1=\frac{3(6p^2-35p+30)}{p+30}.
\]
For $p>6$, we have $0<\theta<1$ and $m_1>2$. These exponents satisfy
\[
\frac5{6p}=\frac\theta{\tilde p},\qquad
\frac1{3p}=\frac\theta{\tilde q}+\frac{1-\theta}{m_1}.
\]
Interpolation and $H^1(\mathbb R)\hookrightarrow L^{m_1}(\mathbb R)$
therefore yield
\begin{equation}\label{e212}
\|u\|_{L_t^{6p/5}L_x^{3p}(I)}^p
\leq\|u\|_{S(I)}^{p\theta}
\|u\|_{L_t^\infty L_x^{m_1}(I)}^{p(1-\theta)}
\lesssim A^{p(1-\theta)}\|u\|_{S(I)}^{p\theta}.
\end{equation}
Combining the preceding bounds, we obtain
\begin{equation}\label{e213}
\left\|\int_{t_1}^{t_2}\U(-s)J\mathcal N(\vec u(s))\,ds\right\|_{X^1}
\lesssim A^{1+p(1-\theta)}\|u\|_{S((t_1,t_2))}^{p\theta}.
\end{equation}
The right-hand side tends to zero as $t_1\to\infty$, uniformly in
$t_2>t_1$, because the forward $S$ norm is finite. Hence
\[
\vec u_+=\vec u_0+\int_0^\infty
\U(-s)J\mathcal N(\vec u(s))\,ds
\]
is well defined in $X^1$. Subtract the free evolution of this scattering state from
\eqref{e2008}. Unitarity and \eqref{e213} then imply
\[
\|\vec u(t)-\U(t)\vec u_+\|_{X^1}
\lesssim A^{1+p(1-\theta)}\|u\|_{S((t,\infty))}^{p\theta}
\longrightarrow0.
\]
This proves the assertion.
\end{proof}
\begin{Remark}
\label{Re001}
The restriction $p>6$ is sufficient for the energy-space estimates used
here. For a $\dot H^{1/6}$-admissible pair $(a,b)$, the homogeneous
estimate contains $K(D)^{-1/a}D^{1/6}$. Its low-frequency singularity
is controlled on the full energy space when $1/6-1/a\geq0$.
Our pair satisfies $\tilde p\in(6,7)$, $\tilde q\in(21,\infty)$,
and the interpolation exponents in \eqref{e212} are admissible.
Moreover, $\tilde p(1-2/\tilde q)>6$, as required for the strict
threshold estimate of Proposition~\ref{prop:strict-gap}. We make no
claim here that this range is optimal.
\end{Remark}
\begin{lem}[Small-data scattering]
\label{le304}
There exists a sufficiently small constant $\delta>0$ such that if $\|\vec{u}_{0}\|_{X^{1}}<\delta$, then the corresponding solution $\vec{u}(t)$ scatters as $t\to\pm\infty$.
\end{lem}
\begin{proof}
Let $a_0=\|\vec u_0\|_{X^1}$. Throughout the lifespan, conservation of energy and the Sobolev
embedding imply
\[
\|\vec u(t)\|_{X^1}^2
=2E(\vec u_0)+\frac2{p+2}\|u(t)\|_{L^{p+2}}^{p+2}
\leq a_0^2+C\|\vec u(t)\|_{X^1}^{p+2}.
\]
If $a_0$ is sufficiently small, continuity improves the bootstrap bound
$\|\vec u(t)\|_{X^1}\leq2a_0$ to a strictly smaller bound.
Thus $\sup_t\|\vec u(t)\|_{X^1}\leq2a_0$, and the continuation
criterion ensures global existence in both time directions.

We next bound the scattering norm uniformly in time. On $I=[0,T]$,
Duhamel's formula and Lemma~\ref{le302} imply
\[
\begin{aligned}
\|u\|_{S(I)}\lesssim{}&
\|K(D)^{-1/\tilde p}D^{1/6}u_0\|_{L^{2}}
+\|K(D)^{-1/\tilde p}D^{1/6}\Lambda^{-1}v_0\|_{L^{2}}\\
&+\|K(D)^{1/\tilde q-1/2}\partial_x\Lambda^{-1}f(u)\|_{N(I)}.
\end{aligned}
\]
The symbol $K(\xi)^{-1/\tilde p}|\xi|^{1/6}/\sqrt{1+\xi^2}$ is
bounded since $\tilde p>6$, so the two homogeneous terms are bounded
by $Ca_0$. The inhomogeneous multiplier has size
$K(\xi)^{1/2+1/\tilde q}$ and satisfies the Mihlin bounds; it is
therefore bounded on $L^{\tilde q'}$. Consequently
\begin{equation}\label{e217}
\|u\|_{S(I)}\lesssim a_0+\|f(u)\|_{N(I)}
=a_0+\|u\|_{L_t^{(p+1)\tilde r'}L_x^{(p+1)\tilde q'}(I)}^{p+1}.
\end{equation}
Put
\[
\alpha=\frac{29p-30}{6p}>1,\qquad
\theta=\frac\alpha{p+1},\qquad
\frac1{\tilde q'}=\frac\alpha{\tilde q}+\frac{p+1-\alpha}{m}.
\]
Here $0<\theta<1$ and $2<m<\infty$. The time-exponent identity is
$1/\tilde r'=\alpha/\tilde p$. Interpolation with
$L_t^\infty L_x^m$, together with the energy bound, implies
\[
\|f(u)\|_{N(I)}
\lesssim\|u\|_{S(I)}^\alpha
\|u\|_{L_t^\infty H_x^1(I)}^{p+1-\alpha}
\lesssim a_0^{p+1-\alpha}\|u\|_{S(I)}^\alpha.
\]
Thus, with constants independent of $T$, the quantity
$Y(T):=\|u\|_{S([0,T])}$ satisfies
\[
Y(T)\leq C_0a_0+C_1a_0^{p+1-\alpha}Y(T)^\alpha.
\]
Since $Y(0)=0$ and $Y$ is continuous on finite intervals, choosing
$\delta$ so that $C_1(2C_0)^\alpha\delta^p<C_0$ closes the bootstrap
$Y(T)\leq2C_0a_0$. Therefore
\begin{equation}\label{e218}
\|u\|_{S([0,\infty))}\lesssim\|\vec u_0\|_{X^1}.
\end{equation}
The same argument on negative time intervals yields the backward bound.
Lemma~\ref{le303} now proves scattering in both directions.
\end{proof}

\begin{Remark}
\label{lRe302}
The small-data argument also applies when the free scattering norm
is small relative to a fixed energy-space bound. More precisely, for
each $A>0$ there is $\delta=\delta(A,p)>0$ such that, if
\[
\sup_{t\in I_\sigma}\|\vec u(t)\|_{X^1}\leq A,
\qquad \eta:=\|P\U(t)\vec u_0\|_{S(I_\sigma)}\leq\delta,
\]
then $\|P\vec u\|_{S(I_\sigma)}\leq2\eta$, and the solution scatters
in that direction. Indeed, \eqref{e217} and Duhamel's formula imply
$Y\leq\eta+C_A Y^\alpha$ on each compact subinterval, where
$\alpha=(29p-30)/(6p)>1$. Choosing
$C_A(2\delta)^{\alpha-1}\leq1/4$ gives $Y\leq3\eta/2$ whenever $Y\leq2\eta$, closing the continuity argument.
The same estimate, without a scattering conclusion beyond the
interval, holds on every finite interval with one endpoint at the
initial time, in either time direction. Its constants are independent
of the interval length.

Conversely, suppose that $\vec u$ scatters forward to $\vec u_+\in X^1$.
Writing $I_T=[T,\infty)$, the homogeneous estimate gives
\[
\|P\U(t-T)\vec u(T)\|_{S(I_T)}
\leq \|P\U(t)\vec u_+\|_{S(I_T)}
     +C\|\vec u(T)-\U(T)\vec u_+\|_{X^1}\longrightarrow0.
\]
The first term tends to zero because $P\U(t)\vec u_+\in S(\mathbb R)$.
Scattering also bounds the energy norm on $I_T$, uniformly for large $T$.
Applying the estimate above with initial time $T$ yields
$\|P\vec u\|_{S(I_T)}\to0$. The backward case follows by time reversal.
\end{Remark}

\subsection{Variational structure}\label{sec:variational}
\indent
\par
In addition to $K_1$, introduced before Theorem~\ref{thm003}, we use two further functionals on $X^1$, given by
\[
K_{2}(u,v):=\frac{d}{d\lambda}E\big(\lambda\vec{u}(\lambda x)\big)|_{\lambda=1}=\frac{3}{2}\|\partial_{x}u\|^{2}_{L^{2}}+\frac{1}{2}\|u\|^{2}_{L^{2}}
+\frac{1}{2}\|v\|^{2}_{L^{2}}-\frac{p+1}{p+2}\|u\|^{p+2}_{L^{p+2}},\]
\[G(u,v):=E(u,v)-\frac{1}{p+2}K_{1}(u,v)=\frac{p}{2p+4}\|u\|^{2}_{H^{1}}
+\frac{p}{2p+4}\|v\|^{2}_{L^{2}}.\]

We begin by recalling the variational characterization of the ground state
in the energy space.
\begin{lem}
\label{le101}
Let $p>2$ and let $Q\in H^1(\mathbb R)$ be the unique positive even solution of
\[
 -\partial_x^2 Q+Q-Q^{p+1}=0.
\]
Then
\begin{equation*}
\begin{aligned}
 E(\pm Q,0)&=\inf\{E(u,v):K_{1}(u,v)=0,(u,v)\in X^{1}\setminus\{\mathbf{0}\}\}
 \\&=\inf\{E(u,v):K_{2}(u,v)=0,(u,v)\in X^{1}\setminus\{\mathbf{0}\}\}\\&=
 \inf\{G(u,v):K_{1}(u,v)\leq0,(u,v)\in X^{1}\setminus\{\mathbf{0}\}\}
 \end{aligned}
\end{equation*}
and
\begin{equation*}
\begin{aligned}
&\{(u,v)\in X^{1}:E(u,v)=E(Q,0), K_{1}(u,v)=0\}
 \\&=\{(u,v)\in X^{1}:E(u,v)=E(Q,0), K_{2}(u,v)=0\}
 \\&=\{(\pm Q(x+x_{0}),0):x_{0}\in \mathbb{R}\}.
 \end{aligned}
\end{equation*}
In particular,
\begin{equation*}
\begin{aligned}
&\{(u,v)\in H^{1}_{e}(\mathbb{R})\times L^{2}_{o}(\mathbb{R}):E(u,v)=E(Q,0), K_{1}(u,v)=0\}
 \\&=\{(u,v)\in H^{1}_{e}(\mathbb{R})\times L^{2}_{o}(\mathbb{R}):E(u,v)=E(Q,0), K_{2}(u,v)=0\}
 \\&=\{(\pm Q,0)\}.
 \end{aligned}
\end{equation*}
For $j=1,2$, every minimizing sequence on the nonzero constraint
$K_j=0$ has a strongly convergent subsequence in $X^1$ after spatial
translations. In $X^1_{\mathrm{eo}}$ the subsequence itself converges.
\end{lem}
\begin{proof}
On the scalar Nehari constraint $h\ne0$, $K_1(h,0)=0$, the minimum of
$\|h\|_{H^1}^2$ is $S_Q=\|Q\|_{H^1}^2$, attained precisely at
$\pm Q(\cdot-a)$ by the ground-state characterization \cite{ref-Ca2003}.
Scaling any $g\ne0$ onto this constraint by
$s^p=\|g\|_{H^1}^2/\|g\|_{L^{p+2}}^{p+2}$ yields
$\|g\|_{L^{p+2}}^{p+2}\le S_Q^{-p/2}\|g\|_{H^1}^{p+2}$.

Let $h_n$ be a scalar minimizing sequence on this constraint. Then
$\|h_n\|_{H^1}^2=\|h_n\|_{L^{p+2}}^{p+2}\to S_Q$.
Since $\|h_n\|_{L^{p+2}}^{p+2}\le\|h_n\|_{L^\infty}^p\|h_n\|_{L^2}^2$,
we may translate so that $|h_n(0)|\ge 1/2$. After taking a subsequence,
$h_n\rightharpoonup h$ in $H^1$ and $h_n\to h$ locally uniformly, with $h\ne0$.
Put $a=\|h\|_{H^1}^2/S_Q\in(0,1]$. The Hilbert norm splitting and the
Br\'ezis--Lieb lemma \cite{ref-BL1983}, followed by the sharp inequality
above for $h$ and $h_n-h$, give
$1\le a^{1+p/2}+(1-a)^{1+p/2}$. Hence $a=1$ and the convergence is strong.

We now reduce the second constraint to the first.
For $h\ne0$ with $K_2(h,0)=0$, put $h_\lambda(x)=\lambda h(\lambda x)$
and write $A=\|\partial_x h\|_{L^2}^2$, $B=\|h\|_{L^2}^2$,
$N=\|h\|_{L^{p+2}}^{p+2}$. The function
\[
 E(h_\lambda,0)=\tfrac12(\lambda^3 A+\lambda B)
                   -\tfrac1{p+2}\lambda^{p+1}N
\]
has its unique maximum at $\lambda=1$. Indeed its derivative vanishes
there and, after division by $\lambda^p$, is strictly decreasing for
$p>2$. Choose $\lambda_*>0$ with
$K_1(h_{\lambda_*},0)=0$, equivalently
$\lambda_*^p N=\lambda_*^2 A+B$. Such a root exists by continuity.
Then $E(h,0)\geq E(h_{\lambda_*},0)\geq E(\vec Q)$, and equality
forces $\lambda_*=1$. Testing with $Q$ identifies the minimum.

For a minimizing sequence on $K_2=0$, the identity
$N=(p+2)(3A+B)/(2(p+1))$ gives $N\sim A+B$. In the root equation,
$\lambda_*\geq1$ implies $\lambda_*^{p-2}\leq(A+B)/N$, and
$\lambda_*\leq1$ implies the reverse inequality. Thus the scaling
parameters stay in a compact subinterval of $(0,\infty)$. The rescaled
sequence minimizes on $K_1=0$. Passing to a convergent subsequence of
parameters and undoing the scaling transfers compactness modulo
translation to the original sequence.
To pass from scalar functions to $X^1$, introduce the positive quadratic forms
\[
H_1=E-\frac{K_1}{p+2}=G,\qquad
H_2=E-\frac{K_2}{p+1}
=\frac{p-2}{2(p+1)}\|\partial_xu\|_{L^{2}}^2+
\frac{p}{2(p+1)}(\|u\|_{L^{2}}^2+\|v\|_{L^{2}}^2).
\]
If $K_j(u,v)=0$ and $(u,v)\ne0$, then $u\ne0$ and $K_j(u,0)\leq0$.
There is a unique $s\in(0,1]$ with $K_j(su,0)=0$, and
\[
E(\vec Q)\leq E(su,0)=H_j(su,0)\leq H_j(u,v)=E(u,v).
\]
Equality forces $s=1$, $v=0$, and $u=\pm Q(\cdot-a)$.
The same comparison for a minimizing sequence forces $v_n\to0$ and
$s_n\to1$. Thus scalar compactness extends to compactness in $X^1$ modulo
translation. If the sequence is even-odd, an unbounded sequence of translations would
produce two disjoint reflected copies of the nonzero limit, contradicting
strong convergence to a single copy. The translations are therefore bounded;
evenness of a translate of $Q$ then places its center at zero.

Finally, if $K_1(\vec u)\leq0$ and $\vec u\ne0$, choose
$s\in(0,1]$ by amplitude scaling so that $K_1(s\vec u)=0$. Hence
$G(\vec u)\geq G(s\vec u)=E(s\vec u)\geq E(\vec Q)$.
Testing with $\pm\vec Q$ proves the infimum characterization involving the
nonpositive variational constraint.
\end{proof}
We next prove coercivity of $\mathcal L$ under a variational orthogonality condition. For $\vec u=(u,v)^T\in X^1_{\mathrm{eo}}$, the quadratic form splits as
\begin{equation*}
 \langle\mathcal{L}\vec{u},\vec{u}\rangle
 =\langle Lu,u\rangle+\|v\|^{2}_{L^{2}}.
\end{equation*}
It therefore suffices to consider $L$ on $H^1_e(\mathbb R)$. The restriction has exactly one negative eigenvalue, with an eigenfunction
identified by the calculation
\begin{equation*}
 LQ^{m}=(1-m^{2})Q^{m},
\end{equation*}
where $m>1$ satisfies $m(2m+p)=(p+1)(p+2)$. Moreover, $L$ is non-degenerate on $H^{1}_{e}(\mathbb{R})$.

Introduce the two auxiliary functionals
\begin{equation*}
\widetilde{K}_{1}(u)=\|\partial_{x}u\|^{2}_{L^{2}}+\|u\|^{2}_{L^{2}}-\|u\|^{p+2}_{L^{p+2}},
\end{equation*}
\begin{equation*}
\widetilde{K}_{2}(u)=\frac{3}{2}\|\partial_{x}u\|^{2}_{L^{2}}+\frac{1}{2}\|u\|^{2}_{L^{2}}
-\frac{p+1}{p+2}\|u\|^{p+2}_{L^{p+2}}.
\end{equation*}
Their variational derivatives at $Q$ are
\begin{equation*}
\widetilde{K}'_{1}(Q)=-2\partial_x^2Q+2Q-(p+2
)Q^{p+1}=-pQ^{p+1},
\end{equation*}
\begin{equation*}
\widetilde{K}'_{2}(Q)=-3\partial_x^2Q+Q-(p+1)Q^{p+1}=-2Q-(p-2)Q^{p+1}.
\end{equation*}

\begin{lem}
\label{le102}
Assume $p>2$. There exists a constant $C>0$, independent of $u$, such that if $u\in H^{1}_{e}(\mathbb{R})$ satisfies
\begin{equation*}
 \langle \widetilde{K}'_{1}(Q),u\rangle=0 \quad\text{or}\quad \langle \widetilde{K}'_{2}(Q),u\rangle=0,
\end{equation*}
then
\begin{equation*}
 \langle Lu,u\rangle\geq C\|u\|^{2}_{H^{1}}.
\end{equation*}
\end{lem}

\begin{proof}
Take $\varphi_1=Q$, so that $L\varphi_1=\widetilde K_1'(Q)$. Since
\begin{equation*}
 \langle L\varphi_{1},\varphi_{1}\rangle=-p\|Q\|^{p+2}_{L^{p+2}}<0
\end{equation*}
and $L$ has only one negative eigenvalue and is non-degenerate on $H^{1}_{e}(\mathbb{R})$, we conclude that
\begin{equation*}
 \langle Lu,u\rangle\gtrsim\|u\|^{2}_{H^{1}}\quad\text{whenever}\quad\langle\widetilde{K}'_{1}(Q),u\rangle=0.
\end{equation*}
Similarly, the function $\varphi_{2}\in H^{1}_{e}(\mathbb{R})$ satisfying
$L\varphi_{2}=\widetilde{K}'_{2}(Q)$ is $\varphi_{2}=Q+x\partial_{x}Q$. Since
\begin{equation*}
 \langle L\varphi_{2},\varphi_{2}\rangle=-\frac{2p^{2}-p}{p+4}\|Q\|^{2}_{L^{2}}<0,
\end{equation*}
the same coercivity argument proves
\begin{equation*}
 \langle Lu,u\rangle\gtrsim\|u\|^{2}_{H^{1}}\quad\text{for}\quad\langle\widetilde{K}'_{2}(Q),u\rangle=0.
\end{equation*}
This proves the lemma.
\end{proof}

For $p>2$ and $\vec u\in X^1_{\mathrm{eo}}$, we now define local
coordinates and a distance to the two standing waves for use in both the
variational estimates and the dynamical analysis. Set
\[
d_0(\vec u):=\min_{\varsigma=\pm1}
\|\vec u-\varsigma\vec Q\|_{X^1}.
\]
If $d_0(\vec u)$ is sufficiently small, the minimizing sign
$\varsigma$ is unique. Write $\vec\eta=\varsigma\vec u-\vec Q$,
so that $\|\vec\eta\|_{X^1}=d_0(\vec u)$. With the normalization
of Proposition~\ref{le13}, define
\begin{equation}\label{distance:coordinates}
\begin{gathered}
\lambda_+=\langle\mathcal L\vec\phi^-,\vec\eta\rangle,
\qquad
\lambda_-=\langle\mathcal L\vec\phi^+,\vec\eta\rangle,\\
\vec\eta=\lambda_+\vec\phi^++\lambda_-\vec\phi^-+\vec\gamma,
\qquad
\lambda_1=\frac{\lambda_++\lambda_-}{2},\quad
\lambda_2=\frac{\lambda_+-\lambda_-}{2}.
\end{gathered}
\end{equation}
Then $\langle\mathcal L\vec\phi^\pm,\vec\gamma\rangle=0$, and
spectral coercivity implies
\[
\|\vec\eta\|_{X^1}^2\sim
\lambda_1^2+\lambda_2^2+\langle\mathcal L\vec\gamma,\vec\gamma\rangle.
\]
Since $E'(\vec Q)=0$ and $E''(\vec Q)=\mathcal L$, Taylor expansion
and $E(\varsigma\vec u)=E(\vec u)$ yield
\begin{equation}\label{distance:energy-expansion}
E(\vec u)-E(\vec Q) =\frac12\langle\mathcal L\vec\eta,\vec\eta\rangle +O(\|\vec\eta\|_{X^1}^3) =\lambda_1^2-\lambda_2^2 +\frac12\langle\mathcal L\vec\gamma,\vec\gamma\rangle +O(\|\vec\eta\|_{X^1}^3).
\end{equation}
Here the quadratic expression follows from
$\langle\mathcal L\vec\phi^\pm,\vec\phi^\pm\rangle=0$ and
$\langle\mathcal L\vec\phi^+,\vec\phi^-\rangle=1$.
Consequently, for $d_0$ small,
\[
E(\vec u)-E(\vec Q)+2\lambda_2^2
=\lambda_1^2+\lambda_2^2
+\frac12\langle\mathcal L\vec\gamma,\vec\gamma\rangle
+O(d_0^3)\sim d_0^2.
\]
Choose a fixed $\delta_{\mathrm{loc}}>0$ small enough that this
equivalence holds for $d_0<2\delta_{\mathrm{loc}}$. Let
$\mathcal X$ be a smooth function with $0\leq\mathcal X\leq1$,
$\mathcal X=1$ on $[0,1]$, and $\mathcal X=0$ on $[2,\infty)$.
For $d_0<2\delta_{\mathrm{loc}}$, define
\begin{equation}\label{distance:definition}
d_Q(\vec u)=
\mathcal X\!\left(\frac{d_0(\vec u)}{\delta_{\mathrm{loc}}}\right)
\sqrt{E(\vec u)-E(\vec Q)+2\lambda_2^2}
+\left[1-\mathcal X\!\left(\frac{d_0(\vec u)}{\delta_{\mathrm{loc}}}\right)\right]
d_0(\vec u),
\end{equation}
and set $d_Q=d_0$ outside this neighborhood. Thus the square root
is used only where its argument is nonnegative. The resulting
continuous distance satisfies
\begin{equation}\label{distance:equivalence}
d_Q(\vec u)\sim d_0(\vec u),\qquad
d_Q(\vec u)=0\ \Longleftrightarrow\ \vec u=\pm\vec Q.
\end{equation}
In particular, when $d_Q$ is sufficiently small, the cutoff is one
and $d_Q^2=E(\vec u)-E(\vec Q)+2\lambda_2^2$.

We now use Lemma~\ref{le102} to control the distance locally in terms of
the energy and the variational functionals.
\begin{lem}
\label{le103}
For $p>2$, there exist constants $C>0$ and $0<\delta\ll1$ such that, for every $\vec{u}=(u,v)^{T}\in H^{1}_{e}(\mathbb{R})\times L^{2}_{o}(\mathbb{R})$ with $d_{Q}(\vec{u})<\delta$, the lower bound
\begin{equation*}
 E(\vec{u})-E(\vec{Q})+CK^{2}_{j}(\vec{u})\gtrsim d^{2}_{Q}(\vec{u}),\quad j=1,2,
\end{equation*}
holds.
\end{lem}

\begin{proof}
By sign symmetry, we may assume $\|\vec u-\vec Q\|_{X^1}\leq\|\vec u+\vec Q\|_{X^1}$. Expanding the energy at the ground state, we obtain
\begin{equation}
\label{equ005}
E(u,v)-E(Q,0)=\frac{1}{2}\langle L(u-Q),u-Q\rangle+\frac{1}{2}\|v\|^{2}_{L^{2}}+O(\|\vec{u}-\vec{Q}\|^{3}_{X^{1}}).
\end{equation}
Writing $c_1=1$ and $c_2=1/2$, we likewise expand $K_j$:
\begin{equation}
\label{equ006}
\begin{aligned}
K^{2}_{j}(u,v)&=\big(\widetilde{K}_{j}(u)+c_j\|v\|^{2}_{L^{2}}\big)^{2}\\&=\big(\langle \widetilde{K}'_{j}(Q),u-Q\rangle+c_j\|v\|^{2}_{L^{2}}+O(\|u-Q\|^{2}_{H^{1}})\big)^{2}
\\&=\langle \widetilde{K}'_{j}(Q),u-Q\rangle^{2}+O(\|\vec{u}-\vec{Q}\|^{3}_{X^{1}}).
\end{aligned}
\end{equation}
Lemma \ref{le102} states that if $\langle \widetilde{K}'_{j}(Q),u-Q\rangle=0$, then
\begin{equation}
\label{equ007}
\langle L(u-Q),u-Q\rangle\gtrsim\|(u-Q)\|^{2}_{H^{1}}.
\end{equation}
Consequently, there is a constant $C>0$, independent of $u$, such that every $u\in H^{1}_{e}(\mathbb R)$ satisfies
\begin{equation}
\label{equ008}
\langle L(u-Q),u-Q\rangle+C\langle \widetilde{K}'_{j}(Q),u-Q\rangle^{2}\gtrsim\|u-Q\|^{2}_{H^{1}}.
\end{equation}
Combining (\ref{equ005}), (\ref{equ006}) and (\ref{equ008}) completes the proof.
\end{proof}

\begin{lem}\label{le104}
Let $p>2$ and $\delta>0$. There exist constants $c_\delta,C_\delta>0$
such that every $\vec u\in X^1_{\mathrm{eo}}$ with
$\|\vec u\|_{X^1}\geq\delta$ satisfies
\[
E(\vec u)-E(\vec Q)+C_\delta K_j(\vec u)^2
\geq c_\delta d_Q(\vec u)^2,\qquad j=1,2.
\]
\end{lem}
\begin{proof}
Fix $j\in\{1,2\}$. Let $q_1=p+2$, $q_2=p+1$, and write
$H_j=E-K_j/q_j$ as in the proof of Lemma~\ref{le101}.
For $p>2$, these positive quadratic forms satisfy
$H_j(\vec u)\geq h_p\|\vec u\|_{X^1}^2$ for some $h_p>0$.
Completing the square gives
\[
E(\vec u)-E(\vec Q)+K_j(\vec u)^2 =H_j(\vec u)-E(\vec Q)+K_j(\vec u)/q_j+K_j(\vec u)^2 \geq h_p\|\vec u\|_{X^1}^2-E(\vec Q)-\frac1{4q_j^2}.
\]
Since $d_Q\sim d_0\leq\|\vec u\|_{X^1}+\|\vec Q\|_{X^1}$,
this proves the required inequality whenever
$\|\vec u\|_{X^1}\geq M$, for a fixed sufficiently large $M>\delta$.
When $d_Q(\vec u)<r_0$, for a fixed sufficiently small $r_0>0$,
the result follows from Lemma~\ref{le103}.

It remains to treat
\[
\mathcal C_{\delta,M,r_0}
=\{\vec u\in X^1_{\mathrm{eo}}:
\delta\leq\|\vec u\|_{X^1}\leq M,\ d_Q(\vec u)\geq r_0\}.
\]
The distance and the absolute value of the energy difference are
bounded on this set, say $d_Q\leq D_M$ and
$|E-E(\vec Q)|\leq B_M$. We claim that there exist
$\kappa,\eta>0$ such that
\begin{equation}\label{le104:annular-gap}
\vec u\in\mathcal C_{\delta,M,r_0},\quad |K_j(\vec u)|\leq\kappa
\quad\Longrightarrow\quad E(\vec u)-E(\vec Q)\geq\eta.
\end{equation}
Otherwise there is a sequence in this set with
$K_j(\vec u_n)\to0$ and
$\limsup_n E(\vec u_n)\leq E(\vec Q)$.
Write $K_j(\vec u_n)=A_{j,n}-b_jB_n$, where
\[
\begin{aligned}
A_{1,n}&=\|\vec u_n\|_{X^1}^2,& b_1&=1,\\
A_{2,n}&=\tfrac32\|\partial_xP\vec u_n\|_{L^{2}}^2
+\tfrac12\bigl(\|P\vec u_n\|_{L^{2}}^2+\|P^c\vec u_n\|_{L^{2}}^2\bigr),
&b_2&=\frac{p+1}{p+2},
\end{aligned}
\qquad B_n=\|P\vec u_n\|_{L^{p+2}}^{p+2}.
\]
The lower bound on $\|\vec u_n\|_{X^1}$ implies
$A_{j,n}\geq\delta^2/2$. Therefore $b_jB_n/A_{j,n}\to1$, and
\[
s_n=\left(\frac{A_{j,n}}{b_jB_n}\right)^{1/p}\longrightarrow1,
\qquad K_j(s_n\vec u_n)=0.
\]
The energy is uniformly continuous on bounded subsets of $X^1$, so
$E(s_n\vec u_n)-E(\vec u_n)\to0$. Lemma~\ref{le101} now shows
$E(s_n\vec u_n)\to E(\vec Q)$. By the compactness of minimizing sequences
established in the proof of that lemma, we may pass to a subsequence so that
$s_n\vec u_n\to\varsigma\vec Q$ in $X^1$ for some
$\varsigma\in\{+1,-1\}$; even-odd symmetry fixes the spatial center.
It follows that $d_Q(\vec u_n)\to0$, contradicting $d_Q\geq r_0$.
This proves \eqref{le104:annular-gap}.

Choose $c>0$ with $cD_M^2\leq\eta$. For $|K_j|\leq\kappa$,
\eqref{le104:annular-gap} implies the desired bound on the remaining
set. For $|K_j|>\kappa$, choose $C$ so large that
$C\kappa^2\geq B_M+cD_M^2$; then
$E-E(\vec Q)+CK_j^2\geq c d_Q^2$ there as well.
Taking the minimum of the lower constants and the maximum of the
coefficients over the three regions, and then over $j=1,2$, proves the lemma.
\end{proof}

We conclude by proving that the two variational functionals have the same
sign away from the solitary waves, with quantitative bounds in each case.
\begin{pro}[Variational sign alternatives away from solitary waves]\label{pro106}
Let $p>2$. There are constants $R_0,c_0,c_1>0$ such that, whenever
$0<R<R_0$, $0<\alpha\leq c_0R$, and $\vec u\in X^1_{\mathrm{eo}}$ satisfies
\[
E(\vec u)<E(\vec Q)+\alpha^2,\qquad d_Q(\vec u)>R,
\]
one of the following two alternatives holds simultaneously for $j=1,2$:
\[
K_j(\vec u)\geq c_1\min\{\|\vec u\|_{X^1}^2,d_Q(\vec u)\}
\quad\hbox{or}\quad
K_j(\vec u)\leq-c_1d_Q(\vec u).
\]
The constants may depend on $p$ and the fixed definition of $d_Q$, but
are independent of $R$, $\alpha$, and $\vec u$. The zero vector belongs
to the nonnegative alternative; every other vector has $K_1$ and $K_2$ of the
same strict sign.
\end{pro}
\begin{proof}
We begin by estimating the functionals in the small-norm region.
For $\vec v=(f,g)^T$, Sobolev embedding gives
\[
\|f\|_{L^{p+2}}^{p+2}\leq C_p\|\vec v\|_{X^1}^{p+2}.
\]
The quadratic parts of $K_1$ and $K_2$ are bounded below by
$\frac12\|\vec v\|_{X^1}^2$. Hence we may fix $\rho>0$, depending
only on $p$, so small that
\begin{equation}\label{equ0016}
K_j(\vec v)\geq\tfrac14\|\vec v\|_{X^1}^2,
\qquad \|\vec v\|_{X^1}\leq\rho,\quad j=1,2.
\end{equation}
In particular, every nonzero vector with $K_j=0$ has norm greater
than $\rho$. Lemma~\ref{le104}, applied with this fixed lower bound,
provides constants $c_\rho,C_\rho>0$ such that
\begin{equation}\label{prop42:coercivity}
c_\rho d_Q(\vec v)^2
\leq E(\vec v)-E(\vec Q)+C_\rho K_j(\vec v)^2,
\qquad \|\vec v\|_{X^1}\geq\rho,\quad j=1,2.
\end{equation}
If a nonzero $\vec u$ satisfying the hypotheses had $K_j(\vec u)=0$,
the coercive estimate above would imply
$c_\rho d_Q(\vec u)^2<\alpha^2\leq c_0^2R^2$, contrary to
$d_Q(\vec u)>R$ once $c_0^2<c_\rho$. Thus neither functional vanishes on a nonzero vector satisfying the
hypotheses.

We next prove that their signs agree. Suppose otherwise, and put
\[
A=\|\partial_xu\|_{L^{2}}^2,\qquad B=\|u\|_{L^{2}}^2+\|v\|_{L^{2}}^2,
\qquad N_u=\|u\|_{L^{p+2}}^{p+2}.
\]
Here $N_u>0$, since $u=0$ would make both functionals positive.
Along the amplitude path $\theta\vec u$, $\theta>0$, we have
\[
\begin{aligned}
K_1(\theta\vec u)&=\theta^2(A+B-\theta^pN_u),\\
K_2(\theta\vec u)&=\theta^2\left(\frac{3A+B}{2}
-\frac{p+1}{p+2}\theta^pN_u\right),\\
\frac{d}{d\theta}E(\theta\vec u)&=\theta^{-1}K_1(\theta\vec u).
\end{aligned}
\]
There is therefore a unique $s>0$ such that $\vec w=s\vec u$ satisfies
$K_2(\vec w)=0$. Explicitly,
\[
s^p=\frac{(p+2)(3A+B)}{2(p+1)N_u},\qquad K_2(\vec w)=0.
\]
To determine the sign of $K_1$ along the amplitude path between the original
and rescaled vectors, we combine the mixed-sign assumption with the identities
\[
\begin{aligned}
\theta^{-2}K_1(\theta\vec u)
&=K_1(\vec u)-(\theta^p-1)N_u,\\
0=s^{-2}K_2(s\vec u)
&=K_2(\vec u)-\frac{p+1}{p+2}(s^p-1)N_u.
\end{aligned}
\]
If $K_2(\vec u)>0>K_1(\vec u)$, the second identity implies $s>1$.
The first then shows that for every $\theta\in[1,s]$,
\[
\theta^{-2}K_1(\theta\vec u)\leq K_1(\vec u)<0.
\]
Consequently,
\[
E(\vec w)-E(\vec u)
=\int_1^s \theta^{-1}K_1(\theta\vec u)\,d\theta\leq0.
\]
If $K_2(\vec u)<0<K_1(\vec u)$, then $s<1$, and for every
$\theta\in[s,1]$,
\[
\theta^{-2}K_1(\theta\vec u)\geq K_1(\vec u)>0.
\]
It follows that
\[
E(\vec u)-E(\vec w)
=\int_s^1 \theta^{-1}K_1(\theta\vec u)\,d\theta\geq0.
\]
In either case, Lemma~\ref{le101} now implies
\[
E(\vec Q)\leq E(\vec w)\leq E(\vec u)<E(\vec Q)+\alpha^2.
\]
Since $\vec w\ne0$ and $K_2(\vec w)=0$, its norm exceeds $\rho$.
Applying \eqref{prop42:coercivity} to $\vec w$ and using
$d_Q\sim d_0$, we obtain a sign $\varsigma\in\{+1,-1\}$ such that
\begin{equation}\label{prop42:scaled-near-Q}
d_Q(\vec w)\leq C\alpha,
\qquad \|\vec w-\varsigma\vec Q\|_{X^1}\leq C\alpha.
\end{equation}

We must also control the scaling factor $s$. Write
$M_w=\|\vec w\|_{X^1}^2$ and
$N_w=\|P\vec w\|_{L^{p+2}}^{p+2}$. By \eqref{prop42:scaled-near-Q},
Sobolev embedding, and $\|\vec Q\|_{X^1}^2=\|Q\|_{L^{p+2}}^{p+2}$,
\[
\big|M_w-\|\vec Q\|_{X^1}^2\big|
+\big|N_w-\|Q\|_{L^{p+2}}^{p+2}\big|\leq C\alpha.
\]
For all sufficiently small $\alpha$, it follows that
\[
M_w\geq\tfrac12\|\vec Q\|_{X^1}^2,
\qquad \left|\frac{N_w}{M_w}-1\right|\leq C\alpha.
\]
On the other hand,
\[
K_1(\vec u)=s^{-2}\bigl(M_w-s^{-p}N_w\bigr).
\]
When $K_2(\vec u)>0>K_1(\vec u)$, this formula implies
$1<s^p<N_w/M_w$; when $K_2(\vec u)<0<K_1(\vec u)$, it implies
$N_w/M_w<s^p<1$. In both cases $|s^p-1|\leq C\alpha$.
After reducing the upper bound on $\alpha$, the numbers $s$ and
$s^{-1}$ are uniformly bounded, and the mean value theorem yields
$|s-1|\leq C\alpha$. Consequently,
\[
d_Q(\vec u)\lesssim\|s^{-1}\vec w-\varsigma\vec Q\|_{X^1}
\leq s^{-1}\|\vec w-\varsigma\vec Q\|_{X^1}
+|s^{-1}-1|\|\vec Q\|_{X^1}\leq C\alpha.
\]
The constant $C$ is independent of $R$, $\alpha$, and $\vec u$.
Choosing $c_0$ so that $Cc_0<1$ contradicts $d_Q(\vec u)>R$.
This proves the agreement of signs.

It remains to obtain the stated quantitative bounds. If
$\|\vec u\|_{X^1}\geq\rho$, the energy hypothesis and
\eqref{prop42:coercivity} give
\[
C_\rho K_j(\vec u)^2
\geq c_\rho d_Q(\vec u)^2-\alpha^2
\geq\tfrac12c_\rho d_Q(\vec u)^2,
\]
provided $c_0^2\leq c_\rho/2$. Therefore
\begin{equation}\label{equ0017}
|K_j(\vec u)|\geq c\,d_Q(\vec u),\qquad j=1,2.
\end{equation}
Since the two functionals have the same sign, this estimate proves the
corresponding alternative in the large-norm region. If
$\|\vec u\|_{X^1}\leq\rho$, the positive alternative follows directly
from \eqref{equ0016}, including $\vec u=0$. Taking the smaller of
these fixed lower constants defines $c_1$.
Finally, fix $c_0$ to satisfy the preceding restrictions and choose
$R_0>0$ so small that $c_0R_0$ meets all the required upper bounds
on $\alpha$ and $R_0<d_Q(0)$. All constants are then independent
of the particular $R$, $\alpha$, and $\vec u$, as claimed.
\end{proof}

\section{Dynamics below the ground-state energy}
\indent
\par
Throughout this section, $p>6$. We prove Theorem~\ref{thm003}. The variational characterization in Lemma~\ref{le101} establishes the invariance of $\mathcal K^+$ and $\mathcal K^-$ below $E(\vec Q)$. We then treat scattering in the positive well and blow-up in the negative well.
\subsection{Scattering theory and wave operators}

\begin{pro}[Long-time perturbation theory]\label{pr001}
For every $A,B>0$, there exist $\epsilon_0=\epsilon_0(A,B)>0$ and
$C=C(A,B)>0$ with the following property. Let $I$ be a time interval
containing $0$, and suppose that $\vec w$ is an approximate solution of
\eqref{eq02} with scalar error $e$:
\[
\partial_t\vec w=J\mathcal A\vec w+J\mathcal N(\vec w)+J(e,0)^T,
\qquad
\sup_{t\in I}\|\vec w(t)\|_{X^1}<A,\quad
\|P\vec w\|_{S(I)}<B,\quad \|e\|_{N(I)}<\epsilon.
\]
If $0<\epsilon<\epsilon_0$ and a solution $\vec u$ satisfies
\[
\sup_{t\in I}\|\vec u(t)\|_{X^1}<A,\qquad
\|P\U(t)(\vec u(0)-\vec w(0))\|_{S(I)}<\epsilon,
\]
then
\[
\|P(\vec u-\vec w)\|_{S(I)}\leq C\epsilon.
\]
The spaces $S(I)$ and $N(I)$ are those defined in the subsection on dispersive and Strichartz estimates, and
$\U(t)=e^{tJ\mathcal A}$.
\end{pro}
\begin{proof}
Write $w=P\vec w$, $u=P\vec u$, $\vec d=\vec w-\vec u$,
and $d=P\vec d$. We first establish the nonlinear estimate needed for the iteration. Set
\[
\alpha=\frac{\tilde p}{\tilde r'}=\frac{29p-30}{6p}>1,
\qquad \beta=p+1-\alpha>0,\qquad
\frac1{\tilde q'}=\frac\alpha{\tilde q}+\frac\beta m.
\]
For $p>6$, this determines a finite $m\geq2$. Since
$|f(w)-f(u)|\lesssim |d||w|^p+|d|^{p+1}$, H\"older's inequality
on a subinterval $J\subset I$ yields
\[
\begin{aligned}
\|\,|d||w|^p\|_{N(J)}
&\leq\|d\|_{S(J)}\|w\|_{S(J)}^{\alpha-1}
\|w\|_{L_t^\infty L_x^m(J\times\mathbb R)}^\beta,\\
\||d|^{p+1}\|_{N(J)}
&\leq\|d\|_{S(J)}^\alpha
\|d\|_{L_t^\infty L_x^m(J\times\mathbb R)}^\beta.
\end{aligned}
\]
The energy-space bounds and the embedding $H^1(\mathbb R)\hookrightarrow L^m(\mathbb R)$ control the final factors in both estimates. Thus
\begin{equation}\label{eee006}
\|f(w)-f(u)\|_{N(J)}
\leq C_A\bigl(\|w\|_{S(J)}^{\alpha-1}
+\|d\|_{S(J)}^{\alpha-1}\bigr)\|d\|_{S(J)}.
\end{equation}

We first consider $I\cap[0,\infty)$; the negative-time interval is treated by reversing time. Partition the positive-time interval into consecutive intervals
$I_j=[t_j,t_{j+1}]$, $1\leq j\leq N_0$, with $t_1=0$ and
\[
\|w\|_{S(I_j)}\leq\delta,\qquad
N_0\leq1+(B/\delta)^{\tilde p}.
\]
The final endpoint may be infinite. We choose $\delta>0$ so that
$C_A\delta^{\alpha-1}\leq1/4$. On $I_j$, Duhamel's formula reads
\[
\vec d(t)=\U(t-t_j)\vec d(t_j)
+\int_{t_j}^t\U(t-s)J
\bigl(\mathcal N(\vec w(s))-\mathcal N(\vec u(s))+(e(s),0)^T\bigr)\,ds.
\]
Put
\[
D_j=\|d\|_{S(I_j)},\qquad
L_j=\|P\U(t-t_j)\vec d(t_j)\|_{S(I)}.
\]
The inhomogeneous estimate of Lemma~\ref{le302} and
\eqref{eee006} imply
\begin{equation}\label{stab:interval}
D_j\leq L_j+C\epsilon+\tfrac14D_j+C_A D_j^\alpha.
\end{equation}
Choose $\eta>0$ so that $C_A(4\eta)^{\alpha-1}\leq1/4$.
If $L_j+C\epsilon\leq\eta$, apply \eqref{stab:interval} first on
$[t_j,t]\subset I_j$. The norm restricted to this interval starts at zero and is
continuous in $t$. Under the bootstrap bound $D_j\leq4(L_j+C\epsilon)$,
the estimate improves to
\begin{equation}\label{stab:local}
D_j\leq2(L_j+C\epsilon).
\end{equation}
A continuity argument therefore establishes this bound on all of $I_j$.

To propagate the estimate to the next interval, we use the group property:
\[
\begin{aligned}
\U(t-t_{j+1})\vec d(t_{j+1})
={}&\U(t-t_j)\vec d(t_j)\\
&+\int_{I_j}\U(t-s)J
\bigl(\mathcal N(\vec w)-\mathcal N(\vec u)+(e,0)^T\bigr)(s)\,ds.
\end{aligned}
\]
The full inhomogeneous estimate, with the source restricted to $I_j$,
is valid in $S(I)$. With $C_A$ enlarged in the preceding choices if necessary,
we obtain
\[
L_{j+1}\leq L_j+\tfrac14D_j+C_A D_j^\alpha+C\epsilon
\leq2(L_j+C\epsilon).
\]
Since $L_1<\epsilon$, induction yields
$L_j+C\epsilon\leq C_0 3^{j-1}\epsilon$ for a fixed $C_0$.
Choose $\epsilon_0$ so that $C_0 3^{N_0}\epsilon_0<\eta$.
Then \eqref{stab:local} applies on every interval and
\[
\|d\|_{S(I\cap[0,\infty))}
\leq\sum_{j=1}^{N_0}D_j\leq C(A,B)\epsilon.
\]
Repeating the same argument on $I\cap(-\infty,0]$ proves the result.
\end{proof}

\begin{pro}[Existence of wave operators]
\label{le305}
Suppose that $\vec{\psi}^{+}=(\psi^{+}_{1},\psi^{+}_{2})^{T}\in X^{1}$ satisfies
\[
0<\|\vec{\psi}^{+}\|^{2}_{X^{1}}<2E(\vec{Q}).
\]
Then there exists $\vec{v}_{0}\in X^{1}$ such that the corresponding solution $\vec{v}(t)$ of (\ref{eq02}) exists globally and
\[
\lim_{t\to+\infty}\big\|\vec{v}(t)-\U(t)\vec{\psi}^{+}\big\|_{X^{1}}=0.
\]
\end{pro}

\begin{proof}
Fix $M>\|\vec\psi^+\|_{X^1}$ and write $I_T=[T,\infty)$.
By the homogeneous estimate,
\[
\|P\U(t)\vec\psi^+\|_{S(\mathbb R)}\lesssim\|\vec\psi^+\|_{X^1},
\]
and the scattering norm of the free solution on $I_T$ tends to zero as $T\to\infty$.
We solve
\begin{equation}\label{e0222}
\vec v(t)=\U(t)\vec\psi^+
-\int_t^\infty\U(t-s)J\mathcal N(\vec v(s))\,ds
=:\Phi(\vec v)(t)
\end{equation}
in the closed ball
\[
\|\vec v\|_{L_t^\infty X^1([T,\infty))}\leq M,\qquad
\|P\vec v\|_{S([T,\infty))}\leq\sigma,
\]
with the metric induced by the sum of these two norms. Write
$\alpha=(29p-30)/(6p)>1$ and $\beta=(35p-30)/(6p)>1$.
Estimates \eqref{e213} and \eqref{e217} imply
\[
\|\Phi(\vec v)\|_{L_t^\infty X^1(I_T)}\leq\|\vec\psi^+\|_{X^1}+C_M\sigma^\beta,
\quad
\|P\Phi(\vec v)\|_{S(I_T)}\leq\|P\U(t)\vec\psi^+\|_{S(I_T)}+C_M\sigma^\alpha.
\]
The corresponding difference estimates are
\[
\begin{aligned}
&\|\Phi(\vec v)-\Phi(\vec w)\|_{L_t^\infty X^1(I_T)}
+\|P(\Phi(\vec v)-\Phi(\vec w))\|_{S(I_T)}\\
&\qquad\leq C_M(\sigma^{\alpha-1}+\sigma^{\beta-1})
\big(\|\vec v-\vec w\|_{L_t^\infty X^1(I_T)}+\|P(\vec v-\vec w)\|_{S(I_T)}\big).
\end{aligned}
\]
For completeness, the differentiated term is bounded using
\[
|\partial_x(f(v)-f(w))|
\lesssim (|v|^p+|w|^p)|\partial_x(v-w)|
+(|v|^{p-1}+|w|^{p-1})|v-w|\,|\partial_x w|.
\]
For the term containing $\partial_x(v-w)$, we use $L_t^\infty L_x^2$
and \eqref{e212}. To estimate the remaining term, set
\[
\frac1{m_*}=\frac{1/3-\beta/\tilde q}{p-\beta},\qquad
m_* =\frac{3(6p^2-35p+30)}{p+30}>2.
\]
Since $1/(6/5)=\beta/\tilde p$ and
$1/(6/5)=1/2+\beta/\tilde q+(p-\beta)/m_*$, H\"older's inequality then gives
\[
\|\,|v-w|\,|w|^{p-1}\partial_xw\|_{L_t^{6/5}L_x^{6/5}}
\lesssim\|v-w\|_{S(I_T)}\|w\|_{S(I_T)}^{\beta-1}
\|w\|_{L_t^\infty L_x^{m_*}}^{p-\beta}
\|\partial_xw\|_{L_t^\infty L_x^2}.
\]
The term with $|v|^{p-1}$ is estimated in the same way. The
undifferentiated term is estimated by replacing the final derivative factor
with $w$ (or $v$). The estimate in $S$ follows from \eqref{eee006}.
Choosing $\sigma$ small and then $T$ large makes $\Phi$ a contraction.
By \eqref{e213}, the resulting solution scatters forward.
If $\vec\psi^+\in X^1_{\mathrm{eo}}$, both the free evolution and the
Duhamel map preserve this subspace, so the fixed point is even--odd.

We also need the following decay property of the free flow:
\begin{equation}\label{free-decay}
\|P\U(t)\vec\psi\|_{L^q}\longrightarrow0
\quad (|t|\to\infty),\qquad \vec\psi\in X^1,\quad 2<q<\infty.
\end{equation}
Approximate $\vec\psi$ in $X^1$ by smooth data with Fourier support in
a compact annulus. For these data, decay follows from Lemma~\ref{le301};
unitarity and Sobolev embedding provide a uniform bound on the approximation error.
It follows that
\[
E(\vec v)=\tfrac12\|\vec\psi^+\|_{X^1}^2<E(\vec Q),\qquad
K_1(\vec v(t))\longrightarrow\|\vec\psi^+\|_{X^1}^2>0.
\]
The variational barrier in Lemma~\ref{le101} ensures that $K_1$ remains positive throughout the
lifespan. Thus
\[
\|\vec v(t)\|_{X^1}^2
=\frac{2(p+2)}p\left(E(\vec v)-\frac{K_1(\vec v(t))}{p+2}\right)
<\frac{p+2}{p}\|\vec\psi^+\|_{X^1}^2.
\]
The continuation criterion extends $\vec v$ to all $t\in\mathbb R$.
\end{proof}

For any $\vec\psi^-\in X^1$ with $0<\|\vec\psi^-\|_{X^1}^2<2E(\vec Q)$,
the same construction at negative infinity gives a global solution $\vec w$ satisfying
\[
\lim_{t\to-\infty}\|\vec{w}(t)-\U(t)\vec{\psi}^{-}\|_{X^{1}}=0.
\]
\subsection{Compactness analysis}
\indent
\par
The extraction and iteration below follow the concentration--compactness
scheme in \cite{ref-Ho2008,ref-Du2008}. We present the details needed for the
Boussinesq flow, with particular attention to reflection pairing in the even-odd space.

\begin{lem}[Quantitative extraction]\label{lee301}
Let $\{h_n\}$ be bounded in $H^{1/3}(\mathbb R)$, and suppose
\[
A=\limsup_{n\to\infty}\|h_n\|_{H^{1/3}}>0,\qquad
B=\limsup_{n\to\infty}\|h_n\|_{L^{3}}>0.
\]
After passing to a subsequence, there exist $x_n\in\mathbb R$ and
$h\in H^{1/3}$ such that
\[
h_n(\cdot+x_n)\rightharpoonup h,\qquad
\|h\|_{L^{2}}\geq cB^6A^{-5}>0,
\]
where $c>0$ is an absolute constant.
\end{lem}
\begin{proof}
Pass to a subsequence on which $\|h_n\|_{L^{3}}\to B$ and
$\|h_n\|_{H^{1/3}}\leq2A$. Let $\chi_r$ be a smooth Fourier cutoff
equal to one on $r^{-1}\leq|\xi|\leq r$, supported where
$(2r)^{-1}\leq|\xi|\leq2r$, and bounded by one. Sobolev embedding gives
\[
\|h_n-\chi_r(D)h_n\|_{L^{3}}\lesssim Ar^{-1/6}.
\]
Choose $r=C(A/B)^6\geq1$ so that the approximation error is at most $B/4$.
Writing $g_n=\chi_r(D)h_n$, the inequality
$\|g_n\|_{L^{3}}^3\leq\|g_n\|_{L^{2}}^2\|g_n\|_{L^{\infty}}$
allows us to choose $x_n$ so that $|g_n(x_n)|\gtrsim B^3/A^2$.
Pass to a weakly convergent subsequence of $h_n(\cdot+x_n)$.
Evaluation of $\chi_r(D)h$ at zero defines a continuous $L^2$ functional,
whose norm is $O(r^{1/2})$. Thus
$B^3/A^2\lesssim r^{1/2}\|h\|_{L^{2}}$, proving the claim.
\end{proof}

\begin{pro}[Linear profile decomposition]\label{proo1}
Let $\{\vec u_n\}$ be bounded in $X^1$.
After passing to a subsequence, we have, for every finite $L$,
\[
\vec u_n=\sum_{j=1}^L\U(-t_n^j)T_{x_n^j}\vec\psi^j+\vec w_n^L,
\qquad \vec\psi^j\in X^1.
\]
The parameters and the remainder satisfy
\[
\begin{aligned}
|t_n^j-t_n^k|+|x_n^j-x_n^k|&\to\infty &&(1\leq j<k\leq L),\\
T_{-x_n^j}\U(t_n^j)\vec w_n^L&\rightharpoonup0\ \text{in }X^1 &&(1\leq j\leq L).
\end{aligned}
\]
Moreover,
\begin{equation}\label{eee029}
\|\vec u_n\|_{X^1}^2=\sum_{j=1}^L\|\vec\psi^j\|_{X^1}^2
+\|\vec w_n^L\|_{X^1}^2+o_n(1),
\end{equation}
and, for $2<q<\infty$,
\begin{equation}\label{profile-small}
\lim_{L\to\infty}\limsup_{n\to\infty}
\left(\|P\U(t)\vec w_n^L\|_{L_t^\infty L_x^q}
+\|P\U(t)\vec w_n^L\|_{S(\mathbb R)}\right)=0.
\end{equation}
If $\vec u_n\in X^1_{\mathrm{eo}}$, the profiles can instead be grouped
into reflection-invariant packets:
\begin{equation}\label{paired-profiles}
\vec u_n=\sum_{j=1}^L\vec\Phi_n^j+\vec r_n^L,\qquad
\vec\Phi_n^j=
\begin{cases}
\U(-t_n^j)\vec\psi^j,&\vec\psi^j\in X^1_{\mathrm{eo}},\ x_n^j=0,\\
\U(-t_n^j)(T_{x_n^j}\vec\psi^j+\mathscr R T_{x_n^j}\vec\psi^j),
&|x_n^j|\to\infty.
\end{cases}
\end{equation}
Here $\vec r_n^L\in X^1_{\mathrm{eo}}$, any two distinct constituent profiles
are asymptotically orthogonal, and \eqref{profile-small} holds for
$\vec r_n^L$. With $m_j=1$ for centered packets and $m_j=2$ for escaping
packets, \eqref{eee029} becomes
\[
\|\vec u_n\|_{X^1}^2=\sum_{j=1}^L m_j\|\vec\psi^j\|_{X^1}^2
+\|\vec r_n^L\|_{X^1}^2+o_n(1).
\]
Bounded time parameters may be absorbed into the profiles, so each
$t_n^j$ can be taken to be identically zero or to tend to $+\infty$ or $-\infty$.
\end{pro}
\begin{proof}
We follow the usual iterative extraction argument
\cite{ref-Ho2008,ref-Du2008}, while retaining the reflection symmetry.
Let $A$ bound $\|\vec u_n\|_{X^1}$ and set $\vec w_n^0=\vec u_n$.
Suppose that $j$ profiles have been extracted. After taking a subsequence,
put
\[
B_j=\lim_{n\to\infty}
\|P\U(t)\vec w_n^j\|_{L_t^\infty L_x^3(\mathbb R^2)}.
\]
If $B_j=0$, the extraction stops, and all subsequent partial sums and remainders equal the final ones. Otherwise, choose $t_n^{j+1}$ such that
\[
\|P\U(t_n^{j+1})\vec w_n^j\|_{L^{3}}\geq B_j/2.
\]
The energy bound controls the first component in $H^{1/3}$.
Applying Lemma~\ref{lee301} and then using weak compactness of both components,
we obtain $x_n^{j+1}$ and a nonzero $\vec\psi^{j+1}\in X^1$ such that
\begin{equation}\label{profile:extraction}
T_{-x_n^{j+1}}\U(t_n^{j+1})\vec w_n^j
\rightharpoonup\vec\psi^{j+1},\qquad
\|\vec\psi^{j+1}\|_{X^1}\gtrsim B_j^6/A^5.
\end{equation}
Define the new remainder by subtracting this profile:
\[
\vec w_n^{j+1}=\vec w_n^j
-\U(-t_n^{j+1})T_{x_n^{j+1}}\vec\psi^{j+1}.
\]
The new remainder converges weakly to zero in this frame. Since the
translations and the free group are unitary on $X^1$, expanding the
squared norm yields
\[
\begin{aligned}
\|\vec w_n^j\|_{X^1}^2
={}&\|\vec\psi^{j+1}\|_{X^1}^2+\|\vec w_n^{j+1}\|_{X^1}^2\\
&+2\bigl\langle\vec\psi^{j+1},
T_{-x_n^{j+1}}\U(t_n^{j+1})\vec w_n^{j+1}\bigr\rangle_{X^1}\\
={}&\|\vec\psi^{j+1}\|_{X^1}^2+\|\vec w_n^{j+1}\|_{X^1}^2+o_n(1).
\end{aligned}
\]
In particular, the remainders remain uniformly bounded. Summing these identities proves \eqref{eee029}.

To prove parameter orthogonality, we use the following weak convergence
property of the free group:
\begin{equation}\label{profile:dislocation}
\U(t_n)T_{x_n}\vec\psi\rightharpoonup0\quad\hbox{in }X^1
\quad\hbox{if }|t_n|+|x_n|\to\infty.
\end{equation}
Every subsequence has a further subsequence on which $t_n\to t_*$ or
$|t_n|\to\infty$. In the first case $|x_n|\to\infty$, so the claim
follows from strong continuity of $\U$ and weak convergence of spatial
translates to zero. In the second case, take smooth, compactly supported data
$\vec\psi_m=(f_m,g_m)\to\vec\psi$ in $X^1$ and write
$(u_m(t),v_m(t))=\U(t)\vec\psi_m$. The identity
$v_m(t)=P\U(t)(g_m,(1-\partial_x^2)f_m)$ allows us to apply \eqref{free-decay}
to both components. Hence $\|u_m(t_n)\|_{L^4}+\|v_m(t_n)\|_{L^4}\to0$ for each fixed $m$.
For a smooth, compactly supported test pair $\vec\phi=(a,b)$,
integration by parts and H\"older's inequality give
\[
\big|\langle\U(t_n)T_{x_n}\vec\psi_m,\vec\phi\rangle_{X^1}\big|
\le \|u_m(t_n)\|_{L^4}\,\|a-\partial_x^2a\|_{L^{4/3}}
   +\|v_m(t_n)\|_{L^4}\,\|b\|_{L^{4/3}}\to0.
\]
By density and unitarity the same limit holds for every $\vec\phi\in X^1$.
The error in replacing $\vec\psi$ by $\vec\psi_m$ satisfies
\[
\big|\langle\U(t_n)T_{x_n}(\vec\psi-\vec\psi_m),\vec\phi\rangle_{X^1}\big|
\le \|\vec\psi-\vec\psi_m\|_{X^1}\,\|\vec\phi\|_{X^1}.
\]
This bound is uniform in $n$, so letting first $n\to\infty$ and then
$m\to\infty$ proves \eqref{profile:dislocation}.

Suppose now that the parameters of the new frame differ by a bounded amount from those of
an earlier frame $k\leq j$. After passing to limits in the relative parameters,
strong continuity would imply that the remainder, which converges weakly to zero in frame $k$,
also has zero weak limit in the new frame. This contradicts
\eqref{profile:extraction}. Thus the relative parameters diverge.
Conversely, \eqref{profile:dislocation} shows that subtracting the new
profile preserves weak convergence to zero in every earlier frame.
This proves both parameter orthogonality and the stated weak
convergence of the final remainder. A diagonal subsequence makes the
construction valid for every finite $L$.

By \eqref{eee029} and \eqref{profile:extraction},
\[
\sum_{j\geq0}B_j^{12}
\lesssim A^{10}\sum_{j\geq1}\|\vec\psi^j\|_{X^1}^2
\lesssim A^{12}.
\]
Hence $B_j\to0$. Interpolation with the uniform $L_t^\infty L_x^2$
bound proves the remainder estimate for $2<q<3$; interpolation with the
uniform $L_t^\infty L_x^\infty$ bound covers $3\leq q<\infty$.
For the scattering norm, use
\[
\|P\U(t)\vec w_n^L\|_{S(\mathbb R)}
\leq\|P\U(t)\vec w_n^L\|_{L_t^{p_0}L_x^{q_0}(\mathbb R^2)}^{\vartheta}
\|P\U(t)\vec w_n^L\|_{L_t^\infty L_x^3(\mathbb R^2)}^{1-\vartheta},
\]
where
\[
p_0=\frac{13p-6}{2p},\qquad q_0=\frac{39p-18}{p-6},\qquad
\vartheta=\frac{13p-6}{14p-12}\in(0,1).
\]
Set $s_0=1/2-2/p_0-1/q_0$. Since
$s_0-1/p_0=(p-6)/(6(13p-6))>0$, the symbol of
$K(D)^{-1/p_0}D^{s_0}$ is bounded near zero, and it has order
$s_0<1/2$ at high frequencies. Thus \eqref{e202} bounds the first norm
by $C\|\vec w_n^L\|_{X^1}$, whose limit superior in $n$ is uniform in $L$
by \eqref{eee029}. The second norm tends to zero in the iterated limit,
proving \eqref{profile-small}.

For $\vec u_n\in X^1_{\mathrm{eo}}$, we preserve the reflection symmetry
at each extraction step.
The identities
\[
\mathscr R\U(t)=\U(t)\mathscr R,\qquad
\mathscr R T_a=T_{-a}\mathscr R
\]
lead to the two alternatives in \eqref{paired-profiles}. If $x_n$ is
bounded, pass to a convergent subsequence and absorb the limiting translation into the profile.
The weak limit of $\U(t_n)\vec r_n$ is then even-odd, and it can be
subtracted in the centered frame $x_n=0$. If $|x_n|\to\infty$ and
$T_{-x_n}\U(t_n)\vec r_n\rightharpoonup\vec\psi$, symmetry implies
\[
T_{x_n}\U(t_n)\vec r_n
=\mathscr R T_{-x_n}\U(t_n)\vec r_n
\rightharpoonup\mathscr R\vec\psi.
\]
These two frames are orthogonal, since their spatial separation is
$2|x_n|\to\infty$. We therefore subtract them together:
\[
\vec r_n^{\,\mathrm{new}}=\vec r_n
-\U(-t_n)\bigl(T_{x_n}\vec\psi+T_{-x_n}\mathscr R\vec\psi\bigr).
\]
The remainder is even-odd and converges weakly to zero in both frames. Expanding its norm, we obtain
\[
\|\vec r_n\|_{X^1}^2
=2\|\vec\psi\|_{X^1}^2
+\|\vec r_n^{\,\mathrm{new}}\|_{X^1}^2+o_n(1).
\]
For two distinct escaping packets, the same argument shows that
$|t_n^j-t_n^k|+\min_{\sigma=\pm1}|x_n^j-\sigma x_n^k|\to\infty$;
the time-space parameters likewise diverge between a centered and an
escaping packet. Thus each centered profile is counted once and each
escaping pair twice. The preceding summability and interpolation
arguments apply without change and prove the paired decomposition.
If a time sequence has a finite limit, strong continuity allows that
limit to be absorbed into its profile, with a vanishing error added to
the remainder. After passage to further subsequences, all remaining time sequences tend to
$+\infty$ or $-\infty$. This completes the proof.
\end{proof}

\begin{lem}[Decoupling of the variational functionals]\label{lee302}
For the general decomposition in Proposition~\ref{proo1} and fixed $L$,
\[
G(\vec u_n)=\sum_{j=1}^LG(\vec\psi^j)+G(\vec w_n^L)+o_n(1),
\]
and, for $F=E,K_1$,
\[
F(\vec u_n)=\sum_{j=1}^LF(\U(-t_n^j)\vec\psi^j)
+F(\vec w_n^L)+o_n(1).
\]
For \eqref{paired-profiles}, the corresponding terms carry multiplicity
$m_j$, and $\vec w_n^L$ is replaced by $\vec r_n^L$.
\end{lem}
\begin{proof}
The decoupling of the quadratic terms follows from \eqref{eee029}. For the nonlinear term,
the first components $P\U(-t_n^j)\vec\psi^j$ tend to zero in $L^{p+2}$ whenever $|t_n^j|\to\infty$, by
\eqref{free-decay}. The remaining profiles have $t_n^j=0$ and mutually
divergent translations. In each such frame, local compactness of
$H^1(\mathbb R)$ implies almost-everywhere convergence to the corresponding profile.
The Br\'ezis--Lieb lemma \cite{ref-BL1983}, applied successively to the
finitely many profiles, yields
\begin{equation}\label{e2030}
\|P\vec u_n\|_{L^{p+2}}^{p+2}
=\sum_{j=1}^L\|P\U(-t_n^j)\vec\psi^j\|_{L^{p+2}}^{p+2}
+\|P\vec w_n^L\|_{L^{p+2}}^{p+2}+o_n(1).
\end{equation}
The nonlinear decoupling completes the proof of both identities. Applying the same argument to the constituent
profiles in each packet proves the even-odd version.
\end{proof}
\subsection{Proof of Theorem \ref{thm003}}
\indent
\par
We first establish a strict lower bound for the general scattering threshold and then prove the scattering and blow-up assertions of Theorem~\ref{thm003}. We begin with two virial functionals used in these arguments and in the later dynamical analysis.

For a solution $\vec{u}=(u,v)^T$ of \eqref{eq02} and a smooth bounded weight $\varphi$,
the first virial functional is
\[
\mathcal{I}_{\varphi}(u,v)=\int_{\mathbb{R}}\varphi uv\,dx.
\]
Here and throughout, we choose $\varphi_R(x)=\phi(Rx)/R$, abbreviated to
$\varphi$ when $R$ is fixed, where $\phi$ is smooth, bounded, and odd, with
\begin{equation}\label{eeqqq02}
\phi(x)=x\quad(|x|\leq1),\qquad
0<\phi'(x)\leq1,\qquad
|\phi''(x)|+|\phi'''(x)|\leq C\phi'(x).
\end{equation}
To construct such a function, choose an even, positive derivative
that equals one on $[-1,1]$ and has smooth exponential tails, then
integrate from zero. In particular,
$|\varphi''|\lesssim R\varphi'$ and
$|\varphi'''|\lesssim R^2\varphi'$, while $\varphi'=1$ on $|x|\leq1/R$.
Integration by parts gives
\begin{equation}\label{eqq005}
\partial_t\mathcal{I}_{\varphi}(u,v)(t)
=\int_{\mathbb{R}}
\left[-\frac{3}{2}\varphi'(\partial_x u)^2
-\frac{1}{2}(\varphi'-\varphi''')u^2
-\frac{1}{2}\varphi'v^2
+\frac{p+1}{p+2}\varphi'|u|^{p+2}\right]dx.
\end{equation}
For energy solutions this identity follows by smooth approximation on
compact time intervals. We record its consequence for a trajectory
whose closure in $X^1$ is compact. For $0<R\leq1$ the nonlinear
localization error has a favorable sign since $0<\varphi_R'\leq1$, so
\begin{equation}\label{virial:tail}
 \partial_t\mathcal I_{\varphi_R}(\vec u)
 \leq-K_2(\vec u)+C\int_{|x|>1/R}
 \bigl(|\partial_xu|^2+|u|^2+|v|^2\bigr)\,dx .
\end{equation}
Here $\varphi_R'=1$ and $\varphi_R'''=0$ on $|x|\leq1/R$, and
$|\varphi_R'''|\lesssim R^2$ everywhere. If the forward orbit is
precompact, its quadratic tails tend to zero uniformly as $R\downarrow0$.
Thus $\inf_{t\geq0}K_2(\vec u(t))=\kappa>0$ would give, for a fixed
small $R$, $\partial_t\mathcal I_{\varphi_R}\leq-\kappa/2$ throughout
the forward lifespan. Since
$|\mathcal I_{\varphi_R}|\lesssim R^{-1}\sup_t\|\vec u(t)\|_{X^1}^2$,
integration rules out such a forward-global orbit.

To define the second functional, for $r>0$ introduce the frequency projections
\begin{equation}\label{equu0015}
\Pi_r f:=\mathcal{F}^{-1}\bigl[\mathbf1_{[-r,r]}(\mathcal{F}f)\bigr],
\qquad \Pi_r^\perp f:=f-\Pi_r f.
\end{equation}
For a solution with initial first component $u_0$, we write
\begin{equation}\label{equu00015}
u_r(t):=\Pi_r u(t),\qquad u_{0,r}:=\Pi_r u_0=u_r(0).
\end{equation}
We use this notation for all the regularized virial functionals
in the subsequent arguments.

For a solution $\vec{u}=(u,v)^{T}$ of (\ref{eq02}), define
\[
h_r(t)=\frac{1}{2}\big\|\partial_x^{-1}\bigl(u(t)-u_r(0)\bigr)\big\|_{L^2}^2.
\]
Integrating the first equation in \eqref{eq02} gives
\begin{equation}
\label{eqqq02}
\partial_x^{-1}\bigl[u(t)-u_r(0)\bigr]
= \int_0^t v(s)\,ds + \partial_x^{-1}\Pi_r^\perp u_0.
\end{equation}
The integral is continuous in $L^2$ on compact time intervals, and
$\|\partial_x^{-1}\Pi_r^\perp u_0\|_{L^2}\leq r^{-1}\|u_0\|_{L^2}$.
Thus $h_r$ is finite throughout the lifespan.

We use the scattering thresholds defined in \eqref{eqe03}.
In the positive well,
$p\|\vec u_0\|_{X^1}^2\leq 2(p+2)E(\vec u_0)$.
Small-data scattering therefore gives
$E_c^{\mathrm{eo}}\geq E_c^{\mathrm{gen}}>0$.
\begin{pro}[A strict lower bound for the general scattering threshold]
\label{prop:strict-gap}
There is a $c_*\in(1/2,1)$ such that, for every $p>6$, every real initial datum
$\vec u_0\in X^1$ satisfying
\[
E(\vec u_0)<c_*E(\vec Q),\qquad K_1(\vec u_0)>0
\]
generates a global solution which scatters as $t\to\pm\infty$.
Moreover,
\begin{equation}\label{gap:uniform}
\sup_t\|\vec u(t)\|_{X^1}
+\|u\|_{L_t^6L_x^\infty(\mathbb R^2)}
+\|u\|_{S(\mathbb R)}\leq C_p,
\end{equation}
uniformly over this initial-data class. In particular,
\begin{equation}\label{gap:threshold}
E_c^{\mathrm{gen}}\geq c_*E(\vec Q)>\tfrac12E(\vec Q).
\end{equation}
\end{pro}
\begin{proof}
The proof combines a quantitative potential-well bound with a bounded
functional whose time derivative controls the spacetime norm. In diagonal coordinates the nonlinear error can be absorbed at half
the ground-state energy. The strictness of this estimate allows us to
increase the energy slightly. The auxiliary estimates are proved in
Appendix~\ref{app:gap}.

\medskip\noindent\textit{Step 1: diagonalization and the quantitative potential well.}
Use the unitary Fourier transform and the operators $D$ and $\Lambda$
defined before Lemma~\ref{le301}. Let $\mathcal H$ be the Hilbert
transform, with symbol $-i\operatorname{sgn}\xi$, and set
\[
z=u-i\Lambda^{-1}\mathcal H v,\qquad
(u,v)=(\operatorname{Re}z,\mathcal H\Lambda\operatorname{Im}z).
\]
The identities $\mathcal H^2=-1$, $\mathcal H\partial_x=D$, and
$D\mathcal H=-\partial_x$ hold on $L^2$ in the multiplier sense;
the value at the single frequency $\xi=0$ is irrelevant. In particular,
this change of variables requires no zero-mean condition. It is a
real-linear isometry from $X^1$ to complex $H^1$:
\[
\|z\|_{H^1}^2=\|u\|_{H^1}^2+\|v\|_{L^{2}}^2.
\]
Indeed, the real and imaginary parts contribute separately to the norm, and the Hilbert
transform is an $L^2$ isometry. Substituting into \eqref{eq02}, we find
\[
i\partial_t z=i\partial_xv+D\Lambda u-D\Lambda^{-1}f(u)
=D\Lambda z-D\Lambda^{-1}f(u),\qquad f(u)=|u|^pu.
\]
Thus
\begin{equation}\label{gap:diag}
i\partial_t z=\omega(D)z-r,\qquad
\omega(D)=D\Lambda,\qquad r=K(D)f(u).
\end{equation}
Write $S_0(t)=e^{-it\omega(D)}$ and introduce
\[
M=\|z\|_{H^1}^2,\quad m=\|z\|_{L^{2}}^2,\quad
k=\|\Lambda^{1/2}z\|_{L^{2}}^2,\quad a=\|\partial_x u\|_{L^{2}}^2,\quad
\mathscr D(z)=\|\partial_x|z|^2\|_{L^{2}}^2.
\]
These are time-dependent scalar quantities; no conservation of $M$,
$m$, or $k$ is assumed.

Let $S_Q=\|Q\|_{H^1}^2$. The sharp inequality from Lemma~\ref{le101} reads
\begin{equation}\label{gap:sharp}
\|h\|_{L^{p+2}}^{p+2}\leq S_Q^{-p/2}\|h\|_{H^1}^{p+2},\qquad
E(\vec Q)=\frac{p}{2(p+2)}S_Q.
\end{equation}

Temporarily suppose $E_0:=E(\vec u_0)<cE(\vec Q)$ for $0<c<1$.
If $K_1$ first vanished at some time, the variational characterization
in Lemma~\ref{le101} would give $E_0\geq E(\vec Q)$, unless the
solution were zero at that time. Uniqueness excludes a zero solution
because the initial data are nonzero. Thus $K_1>0$ throughout the lifespan.
Using
\[
E(\vec u)=\frac{p}{2(p+2)}\|\vec u\|_{X^1}^2
+\frac1{p+2}K_1(\vec u),
\]
we obtain
\[
0<\frac{p}{2(p+2)}M(t)<E_0\leq\frac12M(t),\qquad M(t)<cS_Q.
\]
Global existence follows from the uniform energy-space bound and the continuation criterion. Using \eqref{gap:sharp} and $p\geq6$ once more,
\[
\|u\|_{L^{p+2}}^{p+2}\leq M(M/S_Q)^{p/2}<c^{p/2}M\leq c^3M,
\]
so $E_0\geq[1/2-c^3/(p+2)]M$. This improves the bound to
\begin{equation}\label{gap:F}
M<\frac{cp S_Q}{p+2-2c^3}\leq F(c).
\end{equation}
Here $F$ is a continuous function independent of $p\geq6$, obtained
from the ground-state profile in Lemma~\ref{lem:gap-constants}.

\medskip\noindent\textit{Step 2: the free density bound and a bounded quartic functional.}
Let $W=\Lambda^{1/2}z$ and
$h_W(\xi,\eta)=(\mathcal FW)(\xi)\overline{(\mathcal FW)(\eta)}$.
Lemma~\ref{lem:gap-free} gives a linear operator $\mathcal B$ with
$\|\mathcal B\|^2\leq2/3$ such that
$\mathcal Bh_W=\mathcal F_x(\partial_x|S_0(s)z|^2)$.
Since $\|h_W\|_{L^2_{\xi,\eta}}=k$, Plancherel yields
$\int_{\mathbb R}\mathscr D(S_0(s)z)\,ds\leq 2k^2/3$.
Define
\[
\mathfrak M(z)=\frac12\int_{\mathbb R}\operatorname{sgn}(s)
\mathscr D(S_0(s)z)\,ds,\qquad
\mathfrak V(z)=\frac{\mathfrak M(z)}{k(z)^2}\quad(z\ne0).
\]
The integral is absolutely convergent, and
\begin{equation}\label{gap:bounded-functional}
|\mathfrak M(z)|\leq\tfrac13k^2,
\qquad |\mathfrak V(z)|\leq\tfrac13.
\end{equation}
If $\Sigma F(s,q)=\operatorname{sgn}(s)F(s,q)$, then
\[
\mathfrak M(z)=\tfrac12(\mathcal Bh_W,\Sigma\mathcal Bh_W)_{\mathbb C}.
\]
The map $W\mapsto h_W$ is a continuous quadratic map on the underlying
real space $L^2$, so this formula makes $\mathfrak M$ a continuous
real quartic polynomial on $H^{1/2}$. Multiplying $z$ by a constant
of modulus one leaves $h_W$, and hence $\mathfrak M$, unchanged.
Shifting the integration variable gives the free group identity
\begin{equation}\label{gap:free-increment}
\mathfrak M(S_0(h)z)-\mathfrak M(z)
=-\int_0^h\mathscr D(S_0(s)z)\,ds.
\end{equation}
For $z\in H^1$, the integrand is continuous near zero, since $H^1$ is
an algebra in one dimension. Hence
\[
\left.\frac{d}{ds}\mathfrak M(S_0(s)z)\right|_{s=0}=-\mathscr D(z).
\]

\medskip\noindent\textit{Step 3: the nonlinear increment and exact parallel cancellation.}
Along a nonzero solution, set
\[
b=\Lambda^{1/2}r,\qquad
\alpha=\frac{(b,W)_{\mathbb C}}{k},\qquad
b_\perp=b-\alpha W,\qquad (b_\perp,W)_{\mathbb C}=0,
\]
where $(f,g)_{\mathbb C}=\int f\overline g$ is linear in its first
argument. The projection is well defined: uniqueness prevents a nonzero solution from
vanishing at any time, and $k$ has a positive minimum on
each compact time interval.

To justify differentiation at energy regularity, note that $r(t)$ is
continuous in $H^1$. Duhamel's formula for \eqref{gap:diag} gives
\[
z(t+h)=S_0(h)z(t)
+i\int_0^h S_0(h-s)r(t+s)\,ds
=S_0(h)z(t)+ih\,r(t)+o_{H^1}(h).
\]
The last equality follows from strong continuity of $S_0$ and
continuity of $r$. Since $S_0$ preserves $k$,
\[
\partial_t k(t)=2\operatorname{Re}(ib,W)_{\mathbb C}
=-2\operatorname{Im}(b,W)_{\mathbb C}
=-2k\operatorname{Im}\alpha.
\]
Using the real Fr\'echet derivative of $\mathfrak M$ and
\eqref{gap:free-increment}, we similarly obtain
\[
\partial_t\mathfrak M=-\mathscr D+d\mathfrak M(z)[ir]
=-\mathscr D+d\mathfrak M(z)[i\alpha z]
+d\mathfrak M(z)[i\Lambda^{-1/2}b_\perp].
\]
Write $\alpha=\alpha_1+i\alpha_2$. By phase invariance,
$d\mathfrak M(z)[iz]=0$, while quartic homogeneity gives
$d\mathfrak M(z)[z]=4\mathfrak M(z)$. Therefore
\[
d\mathfrak M(z)[i\alpha z]=-4\alpha_2\mathfrak M(z),\qquad
\partial_t\mathfrak M=-\mathscr D-4(\operatorname{Im}\alpha)\mathfrak M
+\mathcal R_\perp,
\]
where $\mathcal R_\perp=d\mathfrak M(z)[i\Lambda^{-1/2}b_\perp]$.
The quotient rule now reveals the exact cancellation
\begin{equation}\label{gap:derivative}
\partial_t\mathfrak V
=\frac{-\mathscr D-4(\operatorname{Im}\alpha)\mathfrak M+\mathcal R_\perp}{k^2}
-\frac{2\mathfrak M(-2k\operatorname{Im}\alpha)}{k^3}
=\frac{-\mathscr D+\mathcal R_\perp}{k^2}.
\end{equation}
The mild formulation thus justifies \eqref{gap:derivative} at energy regularity.

Lemma~\ref{lem:gap-product} gives the remaining estimate
\begin{equation}\label{gap:projected-bound}
|\mathcal R_\perp|
\leq\frac{2\sqrt2}{3}k^{3/2}\|b_\perp\|_{L^2}
\leq\frac{2\sqrt2}{3}k^{3/2}\|D^{1/2}f(u)\|_{L^2}.
\end{equation}
The factor $\sqrt2$ comes from the complex orthogonality
$(b_\perp,W)_{\mathbb C}=0$, which makes the two terms in the
variation of $h_W$ orthogonal in $L^2(\mathbb R^2)$.

\medskip\noindent\textit{Step 4: absorption with explicit constants.}
Since $z\in H^1(\mathbb R)$, the density $\rho=|z|^2$ belongs to
$H^1\cap L^1$: indeed,
$|\partial_x \rho|\leq2|z||\partial_x z|$ and $H^1\hookrightarrow L^\infty$.
Its mass and derivative norm are $\|\rho\|_{L^{1}}=m$ and
$\|\partial_x \rho\|_{L^{2}}^2=\mathscr D$. Applying Lemma~\ref{lem:gap-density}, we obtain
\[
\|z\|_{L^{\infty}}^6\leq\tfrac9{16}m\mathscr D,
\qquad \int|z|^{26}\leq\tfrac1{64}m^5\mathscr D^4.
\]
The fractional chain estimate in Lemma~\ref{lem:gap-chain} is
\[
\|D^{1/2}f(u)\|_{L^{2}}
\leq\frac{p+1}{\sqrt{2p+1}}
\left(\int|u|^{4p+2}\right)^{1/4}a^{1/4}.
\]
Since $|u|\leq|z|$ and $p\geq6$,
\[
\int|u|^{4p+2}
\leq\|u\|_{L^{\infty}}^{4(p-6)}\int|z|^{26}
\leq\frac1{64}\|u\|_{L^{\infty}}^{4(p-6)}m^5\mathscr D^4.
\]
Combining these estimates with \eqref{gap:projected-bound}, including
the factor $64^{-1/4}=1/(2\sqrt2)$, yields
\[
|\mathcal R_\perp|
\leq\frac{p+1}{3\sqrt{2p+1}}
k^{3/2}m^{5/4}a^{1/4}\|u\|_{L^{\infty}}^{p-6}\mathscr D.
\]
By Cauchy--Schwarz in Fourier variables, $k^2\leq mM$.
Also $a\leq M-m$ and
$\|u\|_{L^{\infty}}^2\leq\|u\|_{L^{2}}\|\partial_x u\|_{L^{2}}\leq M/2$.
Consequently, putting $\sigma=m/M\in[0,1]$,
\[
k^{3/2}m^{5/4}a^{1/4}
\leq M^{3/4}m^2(M-m)^{1/4}
=M^3\sigma^2(1-\sigma)^{1/4}.
\]
The last factor is at most $1/2$. Indeed,
$4\sigma^4(1-\sigma)\leq1-\sigma^4$ and
$\sigma^4(1-\sigma^4)\leq1/4$, so $\sigma^8(1-\sigma)\leq1/16$.
It follows that
\begin{equation}\label{gap:absorb}
|\mathcal R_\perp|\leq a_p M^{p/2}\mathscr D,\qquad
a_p=\frac{p+1}{6\sqrt{2p+1}}\,2^{-(p-6)/2}.
\end{equation}

\medskip\noindent\textit{Step 5: a uniform open interval beyond half energy.}
By Lemma~\ref{lem:gap-constants}, there is an $L>0$ such that
\begin{equation}\label{gap:margin}
 F(1/2)<L,\qquad \theta_*:=\sup_{p\geq6}a_pL^{p/2}<1.
\end{equation}
By continuity of $F$ we may fix $c_*\in(1/2,1)$ with $F(c_*)<L$.
The choice is independent of $p$. If $E(\vec u_0)<c_*E(\vec Q)$,
then $M(t)<L$ throughout the trajectory in the positive well by
\eqref{gap:F}. Combining \eqref{gap:derivative} with \eqref{gap:absorb}, we obtain
\begin{equation}\label{gap:monotonicity}
\partial_t\mathfrak V(z(t))\leq-(1-\theta_*)\frac{\mathscr D(z(t))}{k(t)^2}.
\end{equation}

\medskip\noindent\textit{Step 6: spacetime control, scattering, and the uniform threshold.}
Integrate \eqref{gap:monotonicity} on $[T_1,T_2]$. Using
\eqref{gap:bounded-functional} and $k\leq M<L$, we obtain
\[
\int_{T_1}^{T_2}\mathscr D(z(t))\,dt
\leq\frac{L^2}{1-\theta_*}
\bigl(\mathfrak V(z(T_1))-\mathfrak V(z(T_2))\bigr)
\leq\frac{2L^2}{3(1-\theta_*)}.
\]
Letting $T_1\to-\infty$ and $T_2\to\infty$ and applying the density
bound from Step 4, we find
\[
\int_{\mathbb R}\|u(t)\|_{L^{\infty}}^6\,dt
\leq\frac{9L}{16}\int_{\mathbb R}\mathscr D(z(t))\,dt<\infty.
\]
Now $J\mathcal N(\vec u)=(0,-\partial_xf(u))^T$, so
\[
\begin{aligned}
\|J\mathcal N(\vec u)\|_{L_t^1X^1}
&\leq(p+1)\int\|u(t)\|_{L^{\infty}}^p\|\partial_x u(t)\|_{L^{2}}\,dt\\
&\leq(p+1)\sqrt L(L/2)^{(p-6)/2} \int\|u(t)\|_{L^{\infty}}^6\,dt\leq C_p.
\end{aligned}
\]
The scattering states
\[
\vec u_\pm=\vec u_0+\int_0^{\pm\infty}
\U(-s)J\mathcal N(\vec u(s))\,ds
\]
are therefore defined by absolutely convergent $X^1$ integrals.
For example, subtracting the free evolution of the forward scattering state from
Duhamel's formula yields
\[
\vec u(t)-\U(t)\vec u_+
=-\int_t^\infty\U(t-s)J\mathcal N(\vec u(s))\,ds,
\]
and hence
\[
\|\vec u(t)-\U(t)\vec u_+\|_{X^1}
\leq\int_t^\infty\|J\mathcal N(\vec u(s))\|_{X^1}\,ds\longrightarrow0.
\]
The integral over $(-\infty,t]$ proves the analogous backward limit.

It remains to obtain the uniform scattering-norm bound in the
definition of $A_{\mathrm{gen}}$. Put
$\gamma=\tilde p(1-2/\tilde q)=6+(p-6)/(3p)>6$.
Spatial interpolation gives
\[
\|u(t)\|_{L^{\tilde q}}^{\tilde p}
\leq\|u(t)\|_{L^{2}}^{2\tilde p/\tilde q}\|u(t)\|_{L^{\infty}}^\gamma
\leq L^{\tilde p/\tilde q}(L/2)^{(\gamma-6)/2}\|u(t)\|_{L^{\infty}}^6.
\]
The uniform spacetime bound therefore proves \eqref{gap:uniform}, hence
$A_{\mathrm{gen}}(c_*E(\vec Q))<\infty$ and \eqref{gap:threshold}.
\end{proof}

\begin{pro}[Scattering below the ground-state energy in the even-odd space]
\label{proo001}
Let $p>6$. If $\vec u_0\in X^1_{\mathrm{eo}}$ satisfies
$E(\vec u_0)<E(\vec Q)$ and $K_1(\vec u_0)>0$, then the corresponding
solution is global and scatters as $t\to\pm\infty$. More precisely,
$A_{\mathrm{eo}}(E)<\infty$ for every $0<E<E(\vec Q)$, whereas
$A_{\mathrm{eo}}(E(\vec Q))=\infty$. In particular,
$E_c^{\mathrm{eo}}=E(\vec Q)$.
\end{pro}
\begin{proof}
The potential-well barrier gives global existence and a uniform $X^1$
bound at every fixed energy below $E(\vec Q)$. To see the upper bound
for the threshold, take $\vec u_a(0)=(aQ,0)$ with $0<a<1$.
Then $K_1(\vec u_a(0))=a^2(1-a^p)\|Q\|_{H^1}^2>0$ and
$E(aQ,0)\uparrow E(\vec Q)$ as $a\uparrow1$. For every fixed $T>0$,
$u_a\to Q$ in $C([0,T],H^1)$ by continuous dependence at the standing wave. Hence
\[
 A_{\mathrm{eo}}(E(\vec Q))\geq
 \lim_{a\uparrow1}\|u_a\|_{S([0,T])}
 = T^{1/\tilde p}\|Q\|_{L^{\tilde q}} .
\]
Letting $T\to\infty$ proves endpoint divergence. It remains to show
$E_c^{\mathrm{eo}}\geq E(\vec Q)$. Suppose, to the contrary, that
$E_c^{\mathrm{eo}}<E(\vec Q)$. Proposition~\ref{prop:strict-gap}
and inclusion of the data classes imply
\begin{equation}\label{critical:gap}
0<\tfrac12E_c^{\mathrm{eo}}<\tfrac12E(\vec Q)<E_c^{\mathrm{gen}}\leq E_c^{\mathrm{eo}}.
\end{equation}
By the definition of the threshold,
$A_{\mathrm{eo}}(E)<\infty$ for every $0<E<E_c^{\mathrm{eo}}$ and
$A_{\mathrm{eo}}(E_c^{\mathrm{eo}}+\delta)=\infty$ for every $\delta>0$.
Choose $\delta_n\downarrow0$ with
$E_c^{\mathrm{eo}}+\delta_n<E(\vec Q)$, and even-odd solutions
$\vec v_n$ such that
\[
K_1(\vec v_n(0))>0,\qquad
E(\vec v_n)<E_c^{\mathrm{eo}}+\delta_n,\qquad
\|P\vec v_n\|_{S(\mathbb R)}>2^{1/\tilde p}n.
\]
These solutions are global and uniformly bounded in $X^1$.
Their energies converge to $E_c^{\mathrm{eo}}$: otherwise, a subsequence
would have energies below a fixed $E<E_c^{\mathrm{eo}}$, contradicting
the finite bound $A_{\mathrm{eo}}(E)$.

We choose the time origins so that the scattering norms diverge in
both directions. By local well-posedness and the energy bound, the
scattering integral is finite on every compact time interval. For each $n$, choose finite
$a_n<b_n$ with
\[
\int_{a_n}^{b_n}\|P\vec v_n(t)\|_{L^{\tilde q}}^{\tilde p}\,dt>2n^{\tilde p},
\]
and then $\tau_n\in(a_n,b_n)$ such that
\[
\int_{a_n}^{\tau_n}\|P\vec v_n(t)\|_{L^{\tilde q}}^{\tilde p}\,dt
=\int_{\tau_n}^{b_n}\|P\vec v_n(t)\|_{L^{\tilde q}}^{\tilde p}\,dt
>n^{\tilde p}.
\]
Set $\vec u_n(t)=\vec v_n(t+\tau_n)$. Time translation preserves the
energy, the even-odd symmetry and the potential well; with this choice,
\begin{equation}\label{critical:balanced}
\|P\vec u_n\|_{S(( -\infty,0))}\longrightarrow\infty,
\qquad
\|P\vec u_n\|_{S((0,\infty))}\longrightarrow\infty.
\end{equation}
Both limits in \eqref{critical:balanced} are needed to exclude divergent
profile time parameters.

\medskip\noindent\textit{Step 1: even-odd profile decomposition.}
Fix a common energy-space bound $A$ for these solutions.
The free scattering norm of $\vec u_n(0)$ cannot tend to zero: the
small-free-norm argument in Remark~\ref{lRe302}, with its smallness
threshold chosen in terms of $A$, would then bound the full nonlinear
scattering norms, contradicting \eqref{critical:balanced}. If the
decomposition had no nonzero profile, its remainder estimate would
give precisely this vanishing. The paired decomposition
\eqref{paired-profiles} therefore has at least one nonzero profile. Write
\[
\vec u_n(0)=\sum_{j=1}^L\vec\Phi_n^j+\vec r_n^L.
\]
Let $\mathcal J_c$ denote the centered packets and $\mathcal J_e$ the
escaping packets. With $m_j=1$ on $\mathcal J_c$ and $m_j=2$ on
$\mathcal J_e$, the quadratic and nonlinear decoupling identities are
\begin{equation}\label{critical:decouple}
\begin{aligned}
G(\vec u_n(0))&=\sum_{j=1}^Lm_jG(\vec\psi^j)+G(\vec r_n^L)+o_n(1),\\
E(\vec u_n(0))&=\sum_{j=1}^Lm_jE(\U(-t_n^j)\vec\psi^j)
+E(\vec r_n^L)+o_n(1),\\
K_1(\vec u_n(0))&=\sum_{j=1}^Lm_jK_1(\U(-t_n^j)\vec\psi^j)
+K_1(\vec r_n^L)+o_n(1).
\end{aligned}
\end{equation}
The centered profiles are themselves even-odd. The factor of two occurs
only for a spatially escaping reflected pair, even when its time
parameter also diverges.

Since
\[
G(\vec u_n(0))=E(\vec u_n)-\frac{K_1(\vec u_n(0))}{p+2}
<E_c^{\mathrm{eo}}+o_n(1)<E(\vec Q),
\]
the quadratic decoupling identity places each nonzero constituent and remainder
strictly below the variational barrier. Indeed, if
$S_Q=\|Q\|_{H^1}^2$, then
$G(\vec f)=p\|\vec f\|_{X^1}^2/[2(p+2)]<E(\vec Q)$ implies
$\|\vec f\|_{X^1}^2<S_Q$. The sharp inequality \eqref{gap:sharp}
then gives
\[
K_1(\vec f)\geq\|\vec f\|_{X^1}^2
\left[1-\left(\frac{\|\vec f\|_{X^1}^2}{S_Q}\right)^{p/2}\right]\geq0,
\]
with strict inequality for $\vec f\ne0$. The constituent and remainder energies are therefore nonnegative, with
$E(\vec f)\geq G(\vec f)$.
Define the limiting constituent energy by
\[
e_j=\begin{cases}
E(\vec\psi^j),&t_n^j=0,\\
\frac12\|\vec\psi^j\|_{X^1}^2,&|t_n^j|\to\infty.
\end{cases}
\]
The expression for a divergent time parameter follows from \eqref{free-decay}. Passing to the limit
in \eqref{critical:decouple} yields
\begin{equation}\label{critical:budget}
\sum_{j=1}^L m_je_j\leq E_c^{\mathrm{eo}}.
\end{equation}

\medskip\noindent\textit{Step 2: reduction to one nontrivial profile.}
We associate a nonlinear profile to each constituent of the linear
decomposition.

If $t_n^j=0$ and $j\in\mathcal J_c$, let $\vec z^j$ be the even-odd
solution with initial data $\vec\psi^j$. If $e_j<E_c^{\mathrm{eo}}$, this solution
scatters by the definition of $E_c^{\mathrm{eo}}$. If $t_n^j=0$ and
$j\in\mathcal J_e$, take the general solution with the same initial
data. For such a profile, \eqref{critical:budget} and \eqref{critical:gap} imply
\[
e_j\leq E_c^{\mathrm{eo}}/2<E(\vec Q)/2<E_c^{\mathrm{gen}},
\]
so it scatters by Proposition~\ref{prop:strict-gap}; its reflected
copy scatters as well.

If $|t_n^j|\to\infty$, use Proposition~\ref{le305} to construct
$\vec z^j$ such that
\[
\|\vec z^j(-t_n^j)-\U(-t_n^j)\vec\psi^j\|_{X^1}\to0.
\]
Its conserved energy is $e_j$, and its $K_1$ is positive near the
scattering end, hence throughout its evolution by the potential-well
barrier. For a centered profile, the wave-operator construction preserves even--odd symmetry,
so it scatters in both directions if $e_j<E_c^{\mathrm{eo}}$. If the
profile is escaping, $e_j\leq E_c^{\mathrm{eo}}/2$, so the strict general threshold
again ensures two-sided scattering. Thus all nonlinear profiles scatter
unless one centered profile has energy exactly $E_c^{\mathrm{eo}}$.

Suppose that every nonlinear profile scatters, and define
\[
\vec Z_n^L(t)=
\sum_{\substack{j\leq L\\j\in\mathcal J_c}}\vec z^j(t-t_n^j)
+\sum_{\substack{j\leq L\\j\in\mathcal J_e}}
\left[T_{x_n^j}\vec z^j(t-t_n^j)
+\mathscr R T_{x_n^j}\vec z^j(t-t_n^j)\right].
\]
The wave-operator approximation and smallness of the linear remainder
imply
\[
\lim_{L\to\infty}\limsup_{n\to\infty}
\|P\U(t)(\vec u_n(0)-\vec Z_n^L(0))\|_{S(\mathbb R)}=0.
\]
For fixed $L$, the sum satisfies the first equation exactly.
In the second equation the nonlinear error satisfies
\[
\partial_t\vec Z_n^L-J\mathcal A\vec Z_n^L
-J\mathcal N(\vec Z_n^L)=J\vec e_n^L,
\qquad \|P\vec e_n^L\|_{N(\mathbb R)}\longrightarrow0,
\]
where $\vec e_n^L=(e_n^L,0)^T$. To see this, let $v_1,\dots,v_J$ denote
the first components of the individual summands, counting reflected copies
separately. For $f(s)=|s|^ps$, the mean value theorem gives
\begin{equation}\label{profile:nonlinear-error}
\left|f\!\left(\sum_{j=1}^Jv_j\right)-\sum_{j=1}^J f(v_j)\right|
\le C_{p,J}\sum_{j\ne k}|v_j|^p|v_k|.
\end{equation}
Each first component belongs to $L_t^{(p+1)\tilde r'}L_x^{(p+1)\tilde q'}$
by interpolation with the energy bound. Approximate these components in
this space by compactly supported functions before shifting them.
Distinct shifted supports are then disjoint for large $n$, and
H\"older's inequality controls the approximation error in $N(\mathbb R)$.
Here $J\le 2L$ is fixed as $n\to\infty$.

To apply perturbation theory, we need bounds for $\vec Z_n^L$ that are
uniform in $L$.
Every nonlinear profile remains in the positive potential well, so
\[
\sup_t\|\vec z^j(t)\|_{X^1}^2
\leq\frac{2(p+2)}p e_j.
\]
For sufficiently small $e_j$, Lemma~\ref{le304} consequently gives
$\|P\vec z^j\|_{S(\mathbb R)}^2\lesssim e_j$.
There are only finitely many remaining profiles, since
$\sum_jm_je_j\leq E_c^{\mathrm{eo}}$, and all have finite scattering norms
by the assumption that every profile scatters. Hence
\begin{equation}\label{critical:profile-summability}
\sum_jm_j\left(\sup_t\|\vec z^j(t)\|_{X^1}^2
+\|P\vec z^j\|_{S(\mathbb R)}^2\right)<\infty.
\end{equation}

Index the individual terms in $\vec Z_n^L$ by $a$, writing each as
$T_{y_n^a}\vec v^a(t-\tau_n^a)$ with first component $v_n^a$.
For a centered packet the triple $(\tau_n^a,y_n^a,\vec v^a)$ is
$(t_n^j,0,\vec z^j)$. For an escaping pair the two triples are
$(t_n^j,x_n^j,\vec z^j)$ and $(t_n^j,-x_n^j,\mathscr R\vec z^j)$.
For $a\ne b$, parameter orthogonality and approximation by compactly
supported functions in $S(\mathbb R)$ give
\[
\|v_n^a v_n^b\|_{L_t^{\tilde p/2}L_x^{\tilde q/2}}\to0.
\]
Since $\tilde p,\tilde q>2$, expansion of the square followed by
Minkowski's inequality yields, for every fixed $L$,
\[
\limsup_n\|P\vec Z_n^L\|_{S(\mathbb R)}^2
\leq\sum_{j=1}^Lm_j\|P\vec z^j\|_{S(\mathbb R)}^2.
\]
The corresponding energy inner products vanish uniformly in time.
Indeed, put $\vec h^a(s)=\U(-s)\vec v^a(s)$. Two-sided scattering implies that
$\{\vec h^a(s):s\in\mathbb R\}$ is precompact in $X^1$:
the two tails converge to the scattering states, and the image of
each bounded time interval is compact. For distinct constituents,
their energy inner product can be written
\[
\bigl\langle\vec h^a(t-\tau_n^a),
\U(\tau_n^a-\tau_n^b)T_{y_n^b-y_n^a}\vec h^b(t-\tau_n^b)\bigr\rangle_{X^1}.
\]
By \eqref{profile:dislocation}, this inner product tends to zero uniformly over the
two compact sets, and hence uniformly in $t$. Consequently
\[
\limsup_n\sup_t\|\vec Z_n^L(t)\|_{X^1}^2
\leq\sum_{j=1}^Lm_j\sup_t\|\vec z^j(t)\|_{X^1}^2.
\]
Together with \eqref{critical:profile-summability}, these inequalities
provide constants $A_0,B_0$ independent of $L$ for the perturbation
argument. Choose $L$ so that the free initial discrepancy is smaller
than the required tolerance, and then choose $n$ so large that the
wave-operator discrepancies and the nonlinear error are also small.
Proposition~\ref{pr001} then bounds
$\|P\vec u_n\|_{S(\mathbb R)}$ uniformly, contradicting
\eqref{critical:balanced}. This proves the required nonlinear profile
approximation; compare \cite{ref-Ho2008,ref-Du2008}.

Thus one centered profile carries energy $E_c^{\mathrm{eo}}$. All other
profiles vanish because the terms in \eqref{critical:budget} are nonnegative.
The remainder energy tends to zero, and hence
\[
\frac{p}{2(p+2)}\|\vec r_n^1\|_{X^1}^2
=G(\vec r_n^1)\leq E(\vec r_n^1)\longrightarrow0.
\]
Consequently, for some $\vec\psi\in X^1_{\mathrm{eo}}$,
\begin{equation}\label{critical:single}
\vec u_n(0)=\U(-t_n)\vec\psi+o_{X^1}(1).
\end{equation}
In particular, even a single escaping pair is excluded: its two
constituents would still be below the strict general threshold.

\medskip\noindent\textit{Step 3: the critical element and its compact orbit.}
We show that $t_n$ in \eqref{critical:single} is bounded. If
$t_n\to+\infty$, its nonlinear profile $\vec z$ scatters as
$t\to-\infty$, and
$\vec z(t-t_n)$ approximates $\vec u_n(t)$ on the backward half-line.
More precisely, the initial difference tends to zero in $X^1$ and
$\|P\vec z\|_{S((-\infty,-t_n))}\to0$ by Remark~\ref{lRe302}.
One-sided perturbation therefore bounds $\|P\vec u_n\|_{S(( -\infty,0))}$, a contradiction.
If $t_n\to-\infty$, the forward wave operator and the same argument
bound the forward norm, again contradicting \eqref{critical:balanced}.
Absorbing a finite limit of $t_n$ into $\vec\psi$, we obtain
$\vec u_n(0)\to\vec u_c(0)$ strongly in $X^1$ and
\begin{equation}\label{critical:element}
E(\vec u_c)=E_c^{\mathrm{eo}},\qquad K_1(\vec u_c(t))>0,\qquad
\|P\vec u_c\|_{S(( -\infty,0))}
=\|P\vec u_c\|_{S((0,\infty))}=\infty.
\end{equation}
The solution is global by the potential-well bound. If either norm in
\eqref{critical:element} were finite, continuous dependence on a
compact time interval followed by one-sided stability would uniformly bound the
corresponding norms in \eqref{critical:balanced}.

To prove compactness, apply Steps 1 and 2 to $\vec u_c(s_n)$ for any
sequence $s_n\in\mathbb R$. The shifted solutions still have infinite
scattering norm in both time directions. Escaping packets are excluded
by \eqref{critical:gap}, and any nontrivial energy remainder would
leave a centered profile strictly below $E_c^{\mathrm{eo}}$. Thus a single centered
profile carries all the energy. Divergence of both scattering norms excludes
unbounded profile time parameters by the argument at the start of Step 3, and the remainder
tends to zero in $X^1$. Hence every sequence in the orbit has a strongly
convergent subsequence:
\[
\mathcal C=\{\vec u_c(t):t\in\mathbb R\}
\quad\text{is precompact in }X^1.
\]
No spatial translation is needed. In particular,
\begin{equation}\label{critical:tails}
\lim_{L\to\infty}\sup_t\int_{|x|>L}
\bigl(|\partial_xu_c|^2+|u_c|^2+|v_c|^2+|u_c|^{p+2}\bigr)\,dx=0.
\end{equation}

\medskip\noindent\textit{Step 4: the localized virial contradiction.}
Every point in the compact closure of $\mathcal C$ has energy
$E_c^{\mathrm{eo}}\in(0,E(\vec Q))$. Hence the closure contains
neither zero nor a standing wave, and neither $K_1$ nor $K_2$
vanishes there by Lemma~\ref{le101}. Choose $R$ smaller than the minimum distance to the standing waves.
Since $K_1>0$ along the orbit, Proposition~\ref{pro106} implies $K_2>0$ there.
By compactness and continuity, $\inf_t K_2(\vec u_c(t))>0$, contradicting
the fixed-center virial estimate \eqref{virial:tail}.
Thus $E_c^{\mathrm{eo}}\geq E(\vec Q)$; by the
definition of the threshold and Lemma~\ref{le303}, the proposition
follows.
\end{proof}

To complete the proof of Theorem~\ref{thm003}, we next show that any solution $\vec{u}(t)$ satisfying
\[
E\big(\vec{u}(t)\big)<E(\vec{Q}),\quad K_{1}\big(\vec{u}(t)\big)<0
\]
blows up in finite time in both time directions. We use the classical concavity criterion of Levine \cite{ref-Le1974,ref-Lee1974}, which we recall without proof.

\begin{lem}[Concavity criterion, \cite{ref-Le1974,ref-Lee1974}]
\label{le805}
Let $f$ be a nonnegative $C^2$ function on a time interval beginning at
$t_0$, with $f(t_0)>0$ and $\partial_t f(t_0)>0$. Suppose that, for some $\alpha>1$,
\[
f(t)\partial_t^2f(t)-\alpha(\partial_t f(t))^2\geq0.
\]
Then, for as long as $f$ exists and $t<T_0$,
\[
f(t)\geq f(t_0)
\left(1-(\alpha-1)\frac{\partial_t f(t_0)}{f(t_0)}(t-t_0)\right)^{-1/(\alpha-1)},
\qquad
T_0=t_0+\frac{f(t_0)}{(\alpha-1)\partial_t f(t_0)}.
\]
In particular, if the interval contains $[t_0,T_0)$, then $f(t)\to\infty$ as $t\uparrow T_0$.
\end{lem}

Recall the notation $u_{0,r}=\Pi_r u_0$ from
\eqref{equu0015}--\eqref{equu00015}.

\begin{lem}[Virial identity]
\label{le806}
For initial data $\vec{u}_0=(u_0,v_0)^T\in X^1$, let $\vec{u}(t)=(u,v)^T(t)$ be the corresponding solution of (\ref{eq02}) on its maximal interval of existence $(T^-_{\max},T^+_{\max})$. Define
\[
h_r(t):=\frac{1}{2}\|\partial_x^{-1}(u-u_{0,r})\|_{L^2}^2,\quad t\in(T^-_{\max},T^+_{\max}).
\]
Then
\[
\partial_t h_r(t)=\bigl\langle \partial_x^{-1}(u-u_{0,r}),\, v\bigr\rangle,
\]
and
\[
\partial_t^2 h_r(t)=-\|u\|_{H^1}^2+\|u\|_{L^{p+2}}^{p+2}+\|v\|_{L^2}^2
+\int_{\mathbb R} u_{0,r}\bigl(-\partial_x^2 u+u-|u|^p u\bigr)\,dx.
\]
\end{lem}

\begin{proof}
Equation~\eqref{eqqq02} shows that $h_r$ is well defined
on the lifespan. Differentiation gives
\[
\partial_t h_r(t)=\bigl\langle \partial_x^{-1}(u-u_{0,r}),\, \partial_x^{-1}\partial_t u\bigr\rangle
=\bigl\langle \partial_x^{-1}(u-u_{0,r}),\, v\bigr\rangle,
\]
and
\[
\begin{aligned}
\partial_t^2 h_r(t)
&=\bigl\langle \partial_x^{-1}\partial_t u,\, v\bigr\rangle
+\bigl\langle \partial_x^{-1}(u-u_{0,r}),\, \partial_t v\bigr\rangle \\
&=\|v\|_{L^2}^2-\bigl\langle u-u_{0,r},\, -\partial_x^2 u+u-|u|^p u\bigr\rangle \\
&=-\|u\|_{H^1}^2+\|u\|_{L^{p+2}}^{p+2}+\|v\|_{L^2}^2
+\int_{\mathbb R} u_{0,r}\bigl(-\partial_x^2 u+u-|u|^p u\bigr)\,dx.
\end{aligned}
\]
We first carry out the computation for smooth solutions. Approximation in
$C(I,X^1)$ on compact time intervals then extends the identity to energy
regularity, with the second-order spatial term interpreted in
$H^{-1}$--$H^1$ duality.
\end{proof}

\begin{lem}[Blow-up under a uniform negative sign]\label{lee802}
Suppose that an energy solution satisfies $K_1(\vec u(t))\leq-\sigma$
for some $\sigma>0$ and all $t\in[t_0,T^+_{\max})$. Then
$T^+_{\max}<\infty$ and $\|\vec u(t)\|_{X^1}\to\infty$ as
$t\uparrow T^+_{\max}$.
\end{lem}
\begin{proof}
Translate $t_0$ to zero and write $M(t)=\|\vec u(t)\|_{X^1}^2$
and $N(t)=\|u(t)\|_{L^{p+2}}^{p+2}$.
The energy identity gives
\[
 N=M-K_1(\vec u),\qquad \tfrac p2 M=(p+2)E(\vec u_0)-K_1(\vec u).
\]
Since $K_1(\vec u)\leq-\sigma$, we have $M\leq C_0|K_1(\vec u)|$
and $N\leq(1+C_0)|K_1(\vec u)|$, with $C_0$ depending only on $p$, $\sigma$ and
the conserved energy. In particular,
$M^{1/2}+N^{(p+1)/(p+2)}\leq C_1|K_1(\vec u)|$, since both powers are
less than one and $|K_1(\vec u)|\geq\sigma$.

For $w\in H^1(\mathbb R)$ define
\begin{equation}\label{eq:cutoff-pairing}
 \mathcal R(w,u)=\int_{\mathbb R}
 \bigl(\partial_x w\,\partial_x u+wu-w|u|^pu\bigr)\,dx .
\end{equation}
With $w=u_{0,r}=\Pi_r u_0$, Lemma~\ref{le806} gives the exact identity
\[
 \partial_t^2 h_r(t)=2\|v(t)\|_{L^2}^2-K_1(\vec u(t))+\mathcal R(u_{0,r},u(t)).
\]
By H\"older's inequality and $H^1(\mathbb R)\hookrightarrow L^{p+2}(\mathbb R)$,
\begin{equation}\label{eq:cutoff-bound}
\begin{aligned}
 |\mathcal R(w,u)|
 &\leq \|w\|_{H^1}\|u\|_{H^1}
       +\|w\|_{L^{p+2}}\|u\|_{L^{p+2}}^{p+1}\\
 &\leq C_p\|w\|_{H^1}\bigl(M^{1/2}+N^{(p+1)/(p+2)}\bigr).
\end{aligned}
\end{equation}
Since $\|\Pi_r u_0\|_{H^1}\to0$ as $r\downarrow0$, we may fix $r$
such that $C_p C_1\|\Pi_r u_0\|_{H^1}\leq1/2$. Then, for every
$t\in[0,T^+_{\max})$,
\[
 |\mathcal R(u_{0,r},u(t))|\leq\tfrac12|K_1(\vec u(t))|,\qquad
 \partial_t^2 h_r(t)\geq2\|v(t)\|_{L^2}^2-\tfrac12 K_1(\vec u(t))\geq\sigma/2.
\]
Also $\|v\|_{L^2}^2\leq C_0|K_1(\vec u)|$ and, by Cauchy--Schwarz,
$(\partial_t h_r)^2\leq 2h_r\|v\|_{L^2}^2$. Consequently,
\[
 h_r\,\partial_t^2h_r
 \geq \bigl(1+(4C_0)^{-1}\bigr)(\partial_t h_r)^2 .
\]
If the solution were global forward, both $h_r$ and $\partial_t h_r$
would eventually be positive by the lower bound on $\partial_t^2h_r$. Lemma~\ref{le805} would then force $h_r$ to become
unbounded in finite time, whereas \eqref{eqqq02} keeps it finite on
every compact time interval of the lifespan. The energy-space
continuation criterion gives the asserted blow-up.
\end{proof}
\begin{pro}
\label{proo002}
For any initial datum $\vec{u}_{0}\in X^{1}$ satisfying
\begin{equation*}
E(\vec{u}_{0})<E(\vec{Q}),\quad K_{1}(\vec{u}_{0})<0,
\end{equation*}
the corresponding solution $\vec{u}(t)$ blows up in finite time in both time directions.
\end{pro}
\begin{proof}
Let $(T^{-}_{\max},T^{+}_{\max})$ be the maximal existence interval of $\vec{u}(t)$. By Lemma~\ref{le101} and energy conservation, $K_1$ cannot vanish
on this interval. The variational bound in that lemma also implies
$G(\vec u(t))\geq E(\vec Q)$, so
\[
K_1(\vec u(t))=(p+2)\big(E(\vec u_0)-G(\vec u(t))\big)
\leq-\sigma,\qquad
\sigma:=(p+2)\big(E(\vec Q)-E(\vec u_0)\big)>0.
\]
By Lemma~\ref{lee802}, the solution blows up in finite positive time.
Since time reversal preserves both $E$ and $K_1$, the backward
conclusion follows by applying the lemma to $\mathcal T\vec u(-t)$.
\end{proof}
Propositions~\ref{proo001} and \ref{proo002} complete the proof of Theorem~\ref{thm003}.
\section{Local dynamics near the standing wave}
\subsection{Evolution of the local coordinates}
\indent
\par
We use the coordinates and distance defined in
\eqref{distance:coordinates}--\eqref{distance:equivalence}.
On a connected time interval where $d_0(\vec u)$ is sufficiently
small, the closest sign of $\vec Q$ is fixed. Replacing $\vec u$
by $-\vec u$ if necessary, we work near $\vec Q$. With $\vec\eta=(\eta,\zeta)^T$, write
\[
\vec u=\vec Q+\vec\eta,\qquad
\vec\eta=\lambda_+\vec\phi^++\lambda_-\vec\phi^-+\vec\gamma.
\]
The two sets of coordinates are related by
\begin{equation}\label{local:coordinate-relations}
\lambda_1=\frac{\lambda_++\lambda_-}{2},\qquad
\lambda_2=\frac{\lambda_+-\lambda_-}{2},\qquad
\lambda_+=\lambda_1+\lambda_2,\quad
\lambda_-=\lambda_1-\lambda_2.
\end{equation}
Thus $\lambda_\pm$ are the coefficients of the unstable and stable
eigenfunctions, respectively. We write $\vec\lambda=(\lambda_1,\lambda_2)$
for the coordinates used in the local energy and distance formulas, with
$|\vec\lambda|=(\lambda_1^2+\lambda_2^2)^{1/2}$.
The perturbation satisfies
\[
\partial_t\vec\eta=J\mathcal L\vec\eta
+J\widetilde{\mathcal N}(\vec\eta),
\qquad
\widetilde{\mathcal N}(\vec\eta)=
\begin{pmatrix}
Q^{p+1}+(p+1)Q^p\eta-|Q+\eta|^p(Q+\eta)\\[0.3ex]
0
\end{pmatrix}.
\]
The corresponding local energy expansion is \eqref{distance:energy-expansion}.
We choose the ejection radius $\delta^*>0$ sufficiently small that
$d_Q\leq\delta^*$ implies that the cutoff in
\eqref{distance:definition} equals one.

Put $n(x,\eta)=P\widetilde{\mathcal N}(\vec\eta)$.
With $\vec\phi^+=(f,g)^T$ and $\vec\phi^-=(-f,g)^T$ as in
Proposition~\ref{le13}, the projection formulas give
\[
 \lambda_1=\int_{\mathbb R}g\zeta\,dx,\qquad
 \lambda_2=-\langle Lf,\eta\rangle .
\]
The eigenfunction equations are $\partial_xg=\mu f$ and
$\partial_x(Lf)=\mu g$. Since $f\in H^1$, the first identity implies $g\in H^2$.
By the second identity and the smooth bounded potential
in $L$, we also have $f\in H^3$ and $Lf\in H^1$. Thus we can test
$\partial_t\eta=\partial_x\zeta$ and
$\partial_t\zeta=\partial_x(L\eta+n)$, whose right-hand sides are
continuous in $H^{-1}$ and $H^{-2}$, against $Lf$ and $g$ respectively.
Integration by parts gives
\begin{equation}\label{equ002}
 \partial_t\lambda_1=\mu\lambda_2+r(t),\qquad
 \partial_t\lambda_2=\mu\lambda_1,\qquad
 r(t)=-\int_{\mathbb R}(\partial_x g)n(x,\eta(t,x))\,dx .
\end{equation}
Indeed $-\langle L\eta,\partial_xg\rangle=\mu\lambda_2$, and
$\langle\zeta,\partial_x(Lf)\rangle=\mu\lambda_1$. Equivalently,
\begin{equation}\label{equ001}
 \partial_t\lambda_+=\mu\lambda_++r(t),\qquad
 \partial_t\lambda_-=-\mu\lambda_-+r(t).
\end{equation}
By Taylor's formula, $|n(x,\eta)|\leq C|\eta|^2$ in the fixed
neighborhood, and
\[
 \|n(\cdot,\eta)-n(\cdot,\widetilde\eta)\|_{L^1}
 \leq C(\|\eta\|_{L^2}+\|\widetilde\eta\|_{L^2})
          \|\eta-\widetilde\eta\|_{L^2}.
\]
Since $\partial_xg=\mu f\in L^\infty$, $r$ is continuous and
$|r(t)|\lesssim\|\vec\eta(t)\|_{X^1}^2$. Thus the two coordinates are $C^1$. Differentiating the second
equation in \eqref{equ002}, we find that
$\partial_t^2\lambda_2=\mu^2\lambda_2+\mu r(t)$ at energy regularity.
\subsection{Local analysis and the one-pass theorem}
\indent
\par
Fix the ejection radius $\delta^*>0$ from the preceding subsection and choose smaller parameters such that
\[
0<\varepsilon_{\mathrm{op}}\ll R\ll\delta^*\ll1.
\]
For a solution starting within $\varepsilon_{\mathrm{op}}$ of either solitary wave,
we establish the following properties at its first exit from the
$R$-neighborhood of $\{\pm\vec Q\}$.
\begin{enumerate}
\item[(1)] Its distance from the solitary waves grows exponentially until it reaches $\delta^*$.

\item[(2)] It stays outside the $R$-neighborhood for the rest of its forward lifespan.

\item[(3)] The signs of $K_1$ and $K_2$ then remain fixed and agree, determining the subsequent dynamics.
\end{enumerate}
\begin{pro}[Ejection from a small neighborhood]\label{pro105}
Let $p>2$. There exist $\delta^*>0$ and $c_{\mathrm{ej}}>0$ such that the
following holds for every $0<R<\delta^*$. Let $\vec u$ be an even-odd
solution satisfying
\[
E(\vec u)-E(\vec Q)\leq c_{\mathrm{ej}}R^2.
\]
Suppose that, at some time $t_0$ in its lifespan,
\[
d_Q(\vec u(t_0))=R,\qquad
\left.\sigma\frac{d}{dt}d_Q(\vec u(t))^2\right|_{t=t_0}\geq0,
\qquad \sigma\in\{1,-1\}.
\]
Then there is $\tau_*>0$ such that
$d_Q(\vec u(t_0+\sigma\tau_*))=\delta^*$, and the distance is
strictly increasing as a function of $\tau\in(0,\tau_*)$.
Uniformly for $0\leq\tau\leq\tau_*$,
\begin{equation}\label{equ00012}
d_Q(\vec u(t_0+\sigma\tau))
\sim|\vec\lambda(t_0+\sigma\tau)|
\sim|\lambda_2(t_0+\sigma\tau)|
\sim R e^{\mu\tau},
\end{equation}
and
\[
\min_{j=1,2}|K_j(\vec u(t_0+\sigma\tau))|
\gtrsim R e^{\mu\tau}.
\]
All comparison constants are independent of $R$, $t_0$ and $\sigma$.
The energy hypothesis holds for solutions starting in
$U_{\varepsilon_{\mathrm{op}}}(\pm\vec Q)\cap X^1_{\mathrm{eo}}$ when
$\varepsilon_{\mathrm{op}}/R$ is sufficiently small. The outgoing condition is
satisfied if $d_Q(\vec u(t_0+\sigma\tau))>R$ for all sufficiently
small $\tau>0$.
\end{pro}

\begin{proof}
We first fix a sufficiently small neighborhood in which the local
energy formula and Lemma~\ref{le103} hold, and then choose
$\delta^*$ within that neighborhood. On a connected interval where
$d_Q\leq\delta^*$, the nearest sign of $\vec Q$ is fixed. The
symmetry $\vec u\mapsto-\vec u$ therefore allows us to work near
$\vec Q$ and to write $\vec\eta=(\eta,\zeta)^T=\vec u-\vec Q$.
Put
\[
\rho=E(\vec u)-E(\vec Q),\qquad
F(t)=d_Q(\vec u(t))^2=\rho+2\lambda_2(t)^2.
\]
The quantity $\rho$ is constant by energy conservation. On the
annulus $R\leq d_Q\leq\delta^*$, choose
$c_{\mathrm{ej}}\leq1/2$ to obtain
\[
2\lambda_2^2=F-\rho
\geq F-c_{\mathrm{ej}}R^2\geq\tfrac12F.
\]
Conversely, the projections in \eqref{distance:coordinates} are
bounded linear functionals of $\vec\eta$, and
$d_Q\sim\|\vec\eta\|_{X^1}$. Thus
\begin{equation}\label{equa001}
d_Q(\vec u)\sim|\vec\lambda|\sim|\lambda_2|,
\qquad \|\vec\eta\|_{X^1}\sim d_Q(\vec u).
\end{equation}
The comparison constants depend only on the fixed local coordinates
and are independent of the inner radius. Only the upper bound on
$\rho$ is needed.

By \eqref{equ002}, $\lambda_2\in C^2$ and
$\partial_t^2\lambda_2=\mu^2\lambda_2+\mu r(t)$, with
$|r(t)|\lesssim d_Q^2$. Differentiating $F=\rho+2\lambda_2^2$ gives
\begin{equation}\label{ejection:convex}
\partial_t^2F=4(\partial_t\lambda_2)^2+4\lambda_2\partial_t^2\lambda_2 =4\mu^2(\lambda_1^2+\lambda_2^2)+4\mu\lambda_2r(t) \geq\mu^2F-Cd_Q^3 \geq\tfrac12\mu^2F>0.
\end{equation}
Here the error is controlled by \eqref{equa001}, and the last inequality
holds after decreasing the fixed $\delta^*$.

Now use the oriented time $t=t_0+\sigma\tau$ and set
$F_\sigma(\tau)=F(t_0+\sigma\tau)$. We have
\[
F_\sigma(0)=R^2,\qquad \partial_\tau F_\sigma(0)\geq0,
\qquad \partial_\tau^2 F_\sigma(\tau)=(\partial_t^2F)(t_0+\sigma\tau).
\]
If the initial derivative is positive, the solution enters the
annulus immediately. If it vanishes, \eqref{ejection:convex} at
$t_0$ implies $\partial_\tau^2 F_\sigma(0)>0$, so the solution again moves outward
immediately. On the maximal subsequent interval where
$R\leq d_Q\leq\delta^*$, the same inequality ensures that
$\partial_\tau F_\sigma(\tau)>0$ for every $\tau>0$. Thus the solution cannot
leave through the inner boundary.

For completeness, put $\omega_0=\mu/\sqrt2$ and
$q(\tau)=\partial_\tau^2 F_\sigma(\tau)-\omega_0^2F_\sigma(\tau)\geq0$.
Solving the equation for $F_\sigma$ by variation of constants, we find
\[
\begin{aligned}
F_\sigma(\tau)
&=R^2\cosh(\omega_0\tau)
+\frac{\partial_\tau F_\sigma(0)}{\omega_0}\sinh(\omega_0\tau)\\
&\quad+\int_0^\tau
\frac{\sinh(\omega_0(\tau-s))}{\omega_0}q(s)\,ds
\geq R^2\cosh(\omega_0\tau).
\end{aligned}
\]
Thus the solution cannot remain in the annulus for all positive $\tau$.
A finite endpoint of the lifespan there is also impossible, since
$\|\vec u\|_{X^1}\leq\|\vec Q\|_{X^1}+C\delta^*$ and the
continuation criterion applies in both time directions. By continuity,
there is therefore a finite first time $\tau_*>0$ at which
$d_Q(\vec u(t_0+\sigma\tau_*))=\delta^*$, with strictly increasing
distance on $(0,\tau_*)$.

We next refine this growth bound to obtain the precise rate $\mu$. In the same oriented time, let
\[
a(\tau)=\lambda_2(t_0+\sigma\tau),\qquad
b(\tau)=\sigma\lambda_1(t_0+\sigma\tau),\qquad
w(\tau)=a(\tau)+b(\tau).
\]
The exact equations \eqref{equ002} become
\[
\partial_\tau a=\mu b,\qquad
\partial_\tau b=\mu a+r(t_0+\sigma\tau),\qquad
\partial_\tau w=\mu w+r(t_0+\sigma\tau).
\]
In particular, $\partial_\tau F_\sigma=4\mu ab\geq0$, so $a$ and $b$ have
the same sign whenever both are nonzero. Since
$|a|\sim d_Q\geq R$ by \eqref{equa001},
\[
|w|=|a|+|b|\sim d_Q.
\]
Thus $w$ has a fixed nonzero sign on $[0,\tau_*]$.
Set $h=|w|$. Its equation and the bound on $r$ imply
\[
\left|\frac{\partial_\tau h}h-\mu\right|
\leq\frac{|r(t_0+\sigma\tau)|}{h}
\leq C d_Q,
\qquad \partial_\tau h\geq\frac\mu2h,
\]
where the last inequality follows by taking
$C\delta^*\leq\mu/2$. This lower growth bound controls the time integral of the error
uniformly in the ejection time:
\[
\int_0^\tau d_Q(\vec u(t_0+\sigma s))\,ds
\leq C\int_0^\tau h(s)\,ds
\leq\frac{2C}{\mu}\bigl(h(\tau)-h(0)\bigr)
\leq C\delta^*.
\]
Consequently,
\[
\left|\log\frac{h(\tau)}{h(0)}-\mu\tau\right|
\leq C\delta^*,\qquad 0\leq\tau\leq\tau_*,
\]
or equivalently,
\[
e^{-C\delta^*}h(0)e^{\mu\tau}
\leq h(\tau)\leq e^{C\delta^*}h(0)e^{\mu\tau}.
\]
Since $h(0)\sim R$ and $h\sim d_Q$, this proves
\eqref{equ00012}. All comparison constants depend only on the
fixed neighborhood and the eigenfunctions. Moreover,
$w=\lambda_+$ for $\sigma=1$, whereas $w=-\lambda_-$ for
$\sigma=-1$. Hence the forward unstable mode and the backward
stable mode have the corresponding lower bounds. These estimates hold for either sign of the hyperbolic coordinates.

Finally, by Lemma~\ref{le103}, there are fixed constants $c_2,C_2>0$ such
that
\begin{equation}\label{equ0012}
c_2d_Q(\vec u)^2\leq\rho+C_2K_j(\vec u)^2,
\qquad j=1,2.
\end{equation}
Taking $c_{\mathrm{ej}}\leq c_2/2$ and using $d_Q\geq R$, we
absorb the energy term:
\[
C_2K_j(\vec u)^2
\geq c_2d_Q^2-c_{\mathrm{ej}}R^2
\geq\tfrac12c_2d_Q^2.
\]
The lower bounds for $|K_j|$ follow from \eqref{equ00012}.
For initial data in $U_{\varepsilon_{\mathrm{op}}}(\pm\vec Q)$, the energy
expansion gives $|\rho|\leq C\varepsilon_{\mathrm{op}}^2$, so the energy hypothesis
holds when $\varepsilon_{\mathrm{op}}/R$ is sufficiently small. The strict distance
inequality in the last assertion of the proposition implies $\partial_\tau F_\sigma(0)\geq0$ by
differentiability.

The choices of $\delta^*$ and $c_{\mathrm{ej}}$ are independent of
the inner radius. The argument therefore applies with any
$m\in(0,\delta^*)$ in place of $R$ whenever
$\rho\leq c_{\mathrm{ej}}m^2$. In particular, the original energy
bound suffices for every $m\in[R,\delta^*)$. At a local minimum
with $d_Q=m$, the derivative of $d_Q^2$ vanishes. We may therefore
apply the argument with both $\sigma=1$ and $\sigma=-1$. In each
time direction the distance increases until it reaches $\delta^*$. The
constants remain the same even when $m$ is close to $\delta^*$.
\end{proof}

To prove that an outgoing solution cannot return, we use the
regularized virial functional
\begin{equation*}
\mathcal{T}_{r}(u,v)(t)=\langle\partial_{x}^{-1}[u(t)-u_{r}(t_{0})],\, v(t)-v(t_{1})\rangle,\qquad t_{0}, t_{1}\in (T^{-}_{\max},T^{+}_{\max}),
\end{equation*}
where $(T^{-}_{\max},T^{+}_{\max})$ is the maximal lifespan.
The identity \eqref{eqqq02}, based now at $t_0$, makes the
antiderivative well defined in $L^2$. Differentiating first for smooth solutions as in Lemma~\ref{le806}
and then passing to the limit in the energy space, we find
\begin{equation}\label{eqq004}
 \partial_t\mathcal T_r
 =-K_1(\vec u)+2\|v\|_{L^2}^2-\langle v(t_1),v\rangle
   +\mathcal R(\Pi_r u(t_0),u).
\end{equation}
The pairing $\mathcal R$ is defined in \eqref{eq:cutoff-pairing}.

\begin{pro}[One-pass]\label{pro001}
Let $p>6$. There exist $0<\varepsilon_{\mathrm{op}}\ll R\ll1$ such that every
even--odd solution with initial data in $U_{\varepsilon_{\mathrm{op}}}(\pm\vec Q)$
satisfies exactly one of the following alternatives.

\textup{(1)} It is global forward and $d_Q(\vec u(t))\lesssim\varepsilon_{\mathrm{op}}$
for every $t\geq0$.

\textup{(2)} There is a first exit time $t_0\in(0,T^+_{\max})$ such that
\[
d_Q(\vec u(t))<R\quad(0\leq t<t_0),\qquad
d_Q(\vec u(t))>R\quad(t_0<t<T^+_{\max}).
\]
In the second case $K_1$ and $K_2$ have the same strict sign on
$[t_0,T^+_{\max})$. Appendix~\ref{app:powers} gives the same
conclusion for $6-\varrho<p\leq6$.
\end{pro}

\begin{proof}
By the symmetry $\vec u\mapsto-\vec u$ we may assume that the data
lie in $U_{\varepsilon_{\mathrm{op}}}(\vec Q)$. Fix $\delta^*>0$ as in
Proposition~\ref{pro105}, decreasing it so that
$\delta^*<R_0/2$, where $R_0$ is from Proposition~\ref{pro106}. We will choose
$R$ sufficiently small, and then $\varepsilon_{\mathrm{op}}$ sufficiently small
relative to $R$. In particular, we require
\[
\sqrt R\ll\delta^*,\qquad
\varepsilon_{\mathrm{op}}\ll R.
\]
The constants in the estimates below are independent of the exit and
return times. The trapping bound in Section~\ref{sec:spectral-local} gives the first alternative for
initial data in $\mathcal M$. If the initial data lie outside
$\mathcal M$, the solution must leave the set $\{d_Q<R\}$: otherwise
the energy norm stays bounded, the solution extends globally and
remains in a sufficiently small neighborhood of $\vec Q$, contrary
to the converse characterization of $\mathcal M$ in the same section.

Let $t_0$ be the first exit time. By Proposition~\ref{pro105}, the
distance increases immediately after $t_0$. Suppose that $t_1>t_0$
is its first return to $d_Q=R$. Then
\[
d_Q(\vec u(t_0))=d_Q(\vec u(t_1))=R,\qquad
d_Q(\vec u(t))>R\quad(t_0<t<t_1).
\]
Because $t_0$ is the first exit and the two small neighborhoods of
$\pm\vec Q$ are disjoint, $\|\vec u(t_0)-\vec Q\|_{X^1}\lesssim R$.
Note that
\[|E(\vec u(t))-E(\vec Q)|=|E(\vec u_0)-E(\vec Q)|
\leq C\varepsilon_{\mathrm{op}}^2\ll R^2.\]
Proposition~\ref{pro106} implies, throughout $(t_0,t_1)$,
\begin{equation*}
 K_{j}(\vec{u})\gtrsim \min\{\|\vec{u}\|^{2}_{X^{1}},d_{Q}(\vec{u})\}\quad\text{or}\quad K_j(\vec u)\leq-c\,d_Q(\vec u),\quad j=1,2.
\end{equation*}
Since $E(\vec u)>0$ for small $\varepsilon_{\mathrm{op}}$, the solution never
equals zero. Thus $K_1$ and $K_2$ have the same, constant sign on
this interval. Let $T_0$ be the first time after $t_0$ at which
$d_Q=\delta^*$, and $T_1$ the last such time before $t_1$.
Applying Proposition~\ref{pro105} forward from the exit time and backward
from the return time establishes the existence of these two times.
Moreover, $T_0<T_1$: otherwise $d_Q^2$ would be strictly convex
throughout the return interval, contradicting its equal endpoint
values and larger interior values. The ejection estimates take the form
\begin{equation*}
 d_{Q}\big(\vec{u}(T_{0})\big)=\delta^{*},\quad d_{Q}\big(\vec{u}(t)\big)\sim Re^{\mu(t-t_{0})}\quad\text{for all}\quad t\in(t_{0},T_{0}),
\end{equation*}
and
\begin{equation*}
 d_{Q}\big(\vec{u}(T_{1})\big)=\delta^{*},\quad d_{Q}\big(\vec{u}(t)\big)\sim Re^{\mu(t_{1}-t)}\quad\text{for all}\quad t\in(T_{1},t_{1}).
\end{equation*}
The durations of the two ejection segments therefore satisfy
\begin{equation*}
 e^{\mu(T_{0}-t_{0})}\sim e^{\mu(t_{1}-T_{1})}\sim \delta^{*}/R.
\end{equation*}
Write $d(t)=d_Q(\vec u(t))$ and divide $[T_0,T_1]$ into
\[
I_{\mathrm{in}}=\{t:d(t)\leq\delta^*\},\qquad
I_{\mathrm{out}}=\{t:d(t)>\delta^*\}.
\]
\medskip\noindent\textbf{Case 1.} Suppose $K_1(\vec u)<0$.
Choose $r=R$. By the variational bound above,
$-K_1(\vec u(t))\geq c\,d(t)\geq cR$ on $[t_0,t_1]$, including
the endpoints by continuity. Put $w=\Pi_R u(t_0)$. Since $Q\in L^1$
and $\|u(t_0)-Q\|_{H^1}\lesssim R$, Plancherel gives
\[
 \|w\|_{H^1}\leq CR+\|\Pi_R Q\|_{H^1}\lesssim R+\sqrt R .
\]
Also $\|v(t_1)\|_{L^2}\lesssim R$. We estimate the last term
of \eqref{eqq004} on the two parts of the return interval.

When $d(t)\leq\delta^*$, let $q\in\{Q,-Q\}$ be the nearer profile.
By the stationary equation, $\mathcal R(w,q)=0$. Since $u$ and
$q$ are uniformly bounded in $H^1$ here, Sobolev embedding and
the mean value theorem imply $\||u|^pu-|q|^pq\|_{L^2}\lesssim\|u-q\|_{L^2}$.
Subtracting $\mathcal R(w,q)$ in \eqref{eq:cutoff-pairing} therefore yields
\[
 |\mathcal R(w,u(t))|\lesssim\|w\|_{H^1}\|u(t)-q\|_{H^1}
 \lesssim\sqrt R\,d(t)\lesssim\sqrt R\,[-K_1(\vec u(t))].
\]
When $d(t)>\delta^*$, we have $-K_1(\vec u(t))\geq c\delta^*$. With
$M,N$ as in the proof of Lemma~\ref{lee802}, the energy identities read
\[
 \tfrac p2 M(t)=(p+2)E(\vec u_0)-K_1(\vec u(t)),\qquad
 N(t)=M(t)-K_1(\vec u(t)).
\]
Since the conserved energy is bounded, $M(t)+N(t)\lesssim 1-K_1(\vec u(t))$.
Both powers in \eqref{eq:cutoff-bound} are less than one, so
\[
 |\mathcal R(w,u(t))|\lesssim\sqrt R\, [1-K_1(\vec u(t))]
 \lesssim\sqrt R\,(1+1/\delta^*)[-K_1(\vec u(t))].
\]
Finally, completing the square gives
$2\|v\|_{L^2}^2-\langle v(t_1),v\rangle\geq-\|v(t_1)\|_{L^2}^2/8
\geq-CR^2$. Taking $R$ small relative to the fixed $\delta^*$,
we conclude from \eqref{eqq004} that $\partial_t\mathcal T_R(t)\geq-\tfrac12 K_1(\vec u(t))$
throughout $[t_0,t_1]$. Since $\mathcal T_R(t_1)=0$, we integrate this inequality and retain
only the contribution from the initial ejection segment to obtain
\begin{equation}\label{eqq030}
 -\mathcal T_R(t_0)\geq-\tfrac12\int_{t_0}^{T_0} K_1(\vec u(t))\,dt
 \gtrsim R\int_{t_0}^{T_0}e^{\mu(t-t_0)}\,dt
 \geq c\delta^*-CR .
\end{equation}
However, for $r=R\ll\delta^{*}\ll1$,
\begin{equation}
\label{eqq031}
\begin{aligned}
|\mathcal{T}_{r}(u,v)(t_{0})|
&=|\langle\partial_{x}^{-1}[u(t_{0})-u_{r}(t_{0})],\, v(t_{0})-v(t_{1})\rangle|\\
&\lesssim R\|\partial_{x}^{-1}[u(t_{0})-u_{r}(t_{0})]\|_{L^{2}}\\
&\lesssim R\|\partial_x^{-1}\Pi_r^\perp[u(t_0)-Q]\|_{L^{2}}
+R\|\partial_x^{-1}\Pi_r^\perp Q\|_{L^{2}}\\
&\lesssim R^{2}/r+R/\sqrt{r}\ll \delta^{*}.
\end{aligned}
\end{equation}
For the ground-state term, we used
$\int_{|\xi|>r}|(\mathcal{F}Q)(\xi)|^2\xi^{-2}\,d\xi
\leq2r^{-1}\|\mathcal{F}Q\|_{L^{\infty}}^2$; the perturbation term is bounded by
$r^{-1}\|u(t_0)-Q\|_{L^{2}}$. With $r=R$, the right-hand side of
\eqref{eqq031} is $O(R+\sqrt R)=o(\delta^*)$, contradicting
\eqref{eqq030}. This rules out a return after a negative exit.

\medskip
\noindent\textbf{Case 2.} Suppose $K_1(\vec u)>0$.
\par\nobreak\noindent
On the portion of the return interval where $R\leq d(t)\leq\delta^*$,
the positive variational bound gives $K_2(\vec u(t))\gtrsim d(t)$,
since $\|\vec u(t)\|_{X^1}$ is bounded below there. Choose the nearer
profile $q\in\{Q,-Q\}$. We have
$\|u-q\|_{H^1}+\|v\|_{L^2}\lesssim d(t)$, and $Q,\partial_xQ$
decay exponentially. Thus the quadratic tail in \eqref{virial:tail}
is at most $C d(t)^2+Ce^{-\beta/R}$ for some $\beta>0$.
After decreasing $\delta^*$ and then $R$, we obtain
\begin{equation}\label{eqq009}
 \partial_t\mathcal I_{\varphi_R}(\vec u(t))
 \leq-c\,d(t)+Cd(t)^2+Ce^{-\beta/R}\leq-c'\,d(t).
\end{equation}

We claim that for every sufficiently small $R>0$ one can
choose $\varepsilon_R>0$ with $\varepsilon_R\ll R$ such that, for
every solution starting in $U_{\varepsilon_R}(\vec Q)$ with a
positive first exit followed by a return, one has
\[
\partial_t\mathcal I_{\varphi_R}(\vec u(t))\leq0\qquad(t\in I_{\mathrm{out}}).
\]
If this were false, there would be $R_n\to0$ for which no such
$\varepsilon_R$ exists. Taking $0<\varepsilon_n\leq R_n/n$, we
could choose initial data in $U_{\varepsilon_n}(\vec Q)$ with a
first return excursion $[t_0(n),t_1(n)]$ violating the claim.
The sign of $K_1$ is positive throughout this excursion. There are
times $t_n\in I_{\mathrm{out}}\subset[T_0(n),T_1(n)]$ such that
\[
\left.\partial_t\mathcal I_{\varphi_n}(\vec u_n(t))\right|_{t=t_n}>0,
\qquad \varphi_n:=\varphi_{R_n}.
\]
By energy conservation and $K_1>0$, the following bound holds uniformly in $n$
throughout the return interval:
\[
\sup_{t_0(n)\leq t\leq t_1(n)}\|\vec u_n(t)\|_{X^1}^2
\leq \frac{2(p+2)}p E(\vec u_n),
\qquad E(\vec u_n)\longrightarrow E(\vec Q).
\]
Moreover, $\|\vec u_n(t)\|_{X^1}^2\geq2E(\vec u_n)$.
Since $d_Q(\vec u_n(t_n))>\delta^*$, we apply
Proposition~\ref{pro106} with the fixed radius $\delta^*/2$.
The energy condition holds for large $n$, since
$|E(\vec u_n)-E(\vec Q)|\leq C\varepsilon_n^2$.
The variational estimates, combined with the preceding lower norm bound, give a constant
$c_{\mathrm{out}}>0$ independent of $n$ such that
\[
K_j(\vec u_n(t_n))\geq c_{\mathrm{out}}\delta^*,\quad j=1,2,
\qquad
G(\vec u_n(t_n))\leq E(\vec Q)-\kappa
\]
for all sufficiently large $n$, with a fixed $\kappa>0$.

Translate each solution so that the selected time is zero:
\[
\vec w_n(t)=\vec u_n(t_n+t),\qquad
I_n=[0,t_1(n)-t_n].
\]
The final ejection segment $[T_1(n),t_1(n)]$ lies in $t_n+I_n$.
On that segment $d_Q(\vec u_n(t))\leq\delta^*$. Taking
$\delta^*$ sufficiently small, we obtain by Sobolev embedding
\[
\|P\vec u_n(t)\|_{L^{\tilde q}}
\geq\|Q\|_{L^{\tilde q}}-C\delta^*\geq c_Q>0.
\]
By Proposition~\ref{pro105}, the length of this segment satisfies
$e^{\mu(t_1(n)-T_1(n))}\sim\delta^*/R_n$. Hence
\[
\|P\vec w_n\|_{S(I_n)}^{\tilde p}
\geq c_Q^{\tilde p}\big(t_1(n)-T_1(n)\big)
\longrightarrow\infty.
\]
The final segment therefore forces divergence. In view of the
energy-space bound on $I_n$ and the strict bound on $G(\vec w_n(0))$,
we may apply Lemma~\ref{lem:threshold-compactness} on $I_n$ without
any information beyond $t_1(n)$. We conclude that
\[
\vec u_n(t_n)=\U(-s_n)\vec\psi+o_{X^1}(1),
\qquad \vec\psi\in X^1_{\mathrm{eo}}.
\]
If $|s_n|\to\infty$, the free decay \eqref{free-decay} and the
one-dimensional interpolation inequality
\[\|f\|_{L^{\infty}}\lesssim\|f\|_{L^{4}}^{2/3}\|\partial_x f\|_{L^{2}}^{1/3}\]
 imply
$\|P\vec u_n(t_n)\|_{L^{\infty}}\to0$. Using
$\varphi'_n>0$ and $|\varphi'''_n|\leq C R_n^2\varphi'_n$ in the
virial identity \eqref{eqq005}, we obtain, for large $n$,
\[
\left.\partial_t\mathcal I_{\varphi_n}(\vec u_n(t))\right|_{t=t_n}
\leq-\frac14\int_{\mathbb R}\varphi'_n
\big(|\partial_xP\vec u_n(t_n)|^2
+|P\vec u_n(t_n)|^2+|P^c\vec u_n(t_n)|^2\big)\,dx\leq0.
\]
The $\varphi'''_nu_n^2$ term and the nonlinear term are absorbed
into $\varphi'_nu_n^2$, contradicting the choice of $t_n$. Thus $s_n$ is bounded, and a subsequence satisfies
$\vec u_n(t_n)\to\vec u_*$ strongly in $X^1$.
Strong convergence makes the energy-space tails of $\vec u_n(t_n)$
uniformly small. Since $R_n\to0$, \eqref{virial:tail} and the coercive
bound on $I_{\mathrm{out}}$ give
\[
\left.\partial_t\mathcal I_{\varphi_n}(\vec u_n(t))\right|_{t=t_n}
\le -K_2(\vec u_n(t_n))+o_n(1)
\le -c_{\mathrm{out}}\delta^*+o_n(1)<0
\]
for large $n$, again contradicting the choice of $t_n$.
This proves the claim. Fix such an $R$ and take the initial radius
$\varepsilon_{\mathrm{op}}\leq\varepsilon_R$ small enough to meet the preceding conditions.

\medskip\noindent We can now integrate over the return interval.
The derivative is nonpositive on $I_{\mathrm{out}}$, and
\eqref{eqq009} applies on $I_{\mathrm{in}}$ and both ejection
segments. Integrating the exponential lower bound for $d(t)$ on
the first segment gives
\begin{equation}\label{eqqe007}
 \mathcal I_{\varphi_R}(\vec u(t_1))-\mathcal I_{\varphi_R}(\vec u(t_0))
 \leq-c'\int_{t_0}^{T_0}d(t)\,dt\leq-c\delta^*+CR .
\end{equation}
On the other hand, choose $\sigma_j\in\{\pm1\}$ such that
$\|\vec u(t_j)-\sigma_j\vec Q\|_{X^1}\lesssim R$, $j=0,1$.
Since $\|\varphi_R\|_{L^{\infty}}\lesssim R^{-1}$ and
$|\varphi_R(x)|\leq |x|$, at each endpoint we have
\[
|\mathcal I_{\varphi_R}(\vec u(t_j))|
\leq\|\varphi_R\|_{L^{\infty}}\|u(t_j)-\sigma_jQ\|_{L^{2}}\|v(t_j)\|_{L^{2}}
+\|xQ\|_{L^{2}}\|v(t_j)\|_{L^{2}}\lesssim R.
\]
Consequently,
$|\mathcal I_{\varphi_R}(\vec u(t_1))-\mathcal I_{\varphi_R}(\vec u(t_0))|
\lesssim R$.
For $0<R\ll\delta^{*}$, this contradicts \eqref{eqqe007} and
rules out a return after a positive exit as well. In either case
$d_Q\geq R$ from the exit time onward. Since the conserved energy
differs from $E(\vec Q)$ by $O(\varepsilon_{\mathrm{op}}^2)$, Proposition~\ref{pro106}
at radius $R/2$ and continuity keep $K_1$ and $K_2$ strictly of the
same sign throughout $[t_0,T^+_{\max})$.
\end{proof}
\section{Dynamics in the ground-state energy region}
In this section we take $p>6$ and write $E_Q=E(\vec Q)$. We first prove a uniform scattering estimate away from
$\pm\vec Q$, then combine it with the local one-pass theorem. Appendix~\ref{app:powers} treats the remaining powers in Theorem~\ref{thm002}.

\subsection{Uniform scattering away from the standing waves}
\begin{pro}[Exterior scattering estimate]\label{prop:exterior-scattering}
Let $p>6$, and fix $0<r<R_0$, where $R_0$ is as in
Proposition~\ref{pro106}. There exist $\eta_r,M_r>0$ such that any
even--odd solution $\vec u$ on a time interval $I$ satisfying
\[
E(\vec u)<E_Q+\eta_r^2,\qquad
d_Q(\vec u(t))\geq r,\qquad K_1(\vec u(t))>0
\quad(t\in I)
\]
obeys $\|P\vec u\|_{S(I)}\leq M_r$. In particular, a forward-global
solution satisfying these conditions on $[0,\infty)$ scatters forward.
\end{pro}

\begin{proof}
Suppose the uniform estimate were false. There would be even--odd
solutions $\vec v_n$ on intervals $I_n$, with
\[
E_n:=E(\vec v_n)<E_Q+n^{-2},\qquad
d_Q(\vec v_n(t))\geq r,\qquad K_1(\vec v_n(t))>0,
\]
and $\|P\vec v_n\|_{S(I_n)}>2^{1/\tilde p}n$.
The identity $E=G+K_1/(p+2)$ gives
\begin{equation}\label{exterior:energy-bound}
0<E_n,\qquad
\sup_{t\in I_n}\|\vec v_n(t)\|_{X^1}^2
\leq\frac{2(p+2)}p E_n\leq A^2
\end{equation}
for a fixed $A$. Since $\limsup_n E_n\leq E_Q$, if $E_n$ did not tend to $E_Q$,
we could pass to a subsequence with $E_n<E'$ for some fixed $E'<E_Q$.
Along that subsequence, the solution through any point of $I_n$
lies in the positive potential well below $E'$ and extends globally. By Proposition~\ref{proo001},
$\|P\vec v_n\|_{S(I_n)}\leq A_{\mathrm{eo}}(E')$, a contradiction.
Consequently $E_n\to E_Q$.

Apply Proposition~\ref{pro106} with radius $r/2$ and
$\alpha_n=2/n$. For large $n$ its hypotheses hold throughout $I_n$.
Since $\|\vec v_n(t)\|_{X^1}^2\geq2E_n\geq E_Q$, there is a fixed
$\kappa>0$ such that
\begin{equation}\label{exterior:coercivity}
K_j(\vec v_n(t))\geq\kappa\quad(j=1,2),\qquad
G(\vec v_n(t))\leq E_Q-\frac{\kappa}{2(p+2)}
\quad(t\in I_n)
\end{equation}
after discarding finitely many indices.

Choose a compact subinterval of $I_n$ on which the scattering
integral exceeds $2n^{\tilde p}$, and choose its time origin to split
this integral into two equal parts. Restricting and translating in
time, we may assume that $I_n=[-a_n,b_n]$, with
\begin{equation}\label{exterior:balanced}
\|P\vec v_n\|_{S([-a_n,0])}\geq n,\qquad
\|P\vec v_n\|_{S([0,b_n])}\geq n.
\end{equation}
The bound \eqref{exterior:energy-bound} and
$H^1(\mathbb R)\hookrightarrow L^{\tilde q}(\mathbb R)$ imply
$\|P\vec v_n\|_{S(J)}\leq C A|J|^{1/\tilde p}$.
Thus $a_n,b_n\to\infty$.

\medskip\noindent\textit{Step 1: a full-time threshold solution.}
The bounds \eqref{exterior:energy-bound}--\eqref{exterior:balanced}
allow us to apply Lemma~\ref{lem:threshold-compactness} and write
\[
\vec v_n(0)=\U(-s_n)\vec\psi+o_{X^1}(1),
\qquad \vec\psi\in X^1_{\mathrm{eo}}.
\]
This lemma depends only on the scattering results in Section~3;
in particular, no conclusion of the present section is used.
We claim that $s_n$ is bounded. If $s_n\to-\infty$, the linear
Strichartz estimate gives
\[
\|P\U(t)\vec v_n(0)\|_{S([0,b_n])}
\leq\|P\U(t)\vec\psi\|_{S([-s_n,\infty))}+o(1)\longrightarrow0.
\]
If $s_n\to+\infty$, the corresponding bound is
\[
\|P\U(t)\vec v_n(0)\|_{S([-a_n,0])}
\leq\|P\U(t)\vec\psi\|_{S(( -\infty,-s_n])}+o(1)\longrightarrow0.
\]
In either case, the small-free-norm bootstrap of
Remark~\ref{lRe302}, on the indicated finite interval and with the
fixed bound $A$, contradicts \eqref{exterior:balanced}. That
bootstrap has the same constants on finite intervals, since it
uses only Duhamel's formula and the a priori energy-space bound.
Absorbing a finite limit of $s_n$ into the profile, we obtain
$\vec v_n(0)\to\vec u_c(0)$ in $X^1$.

By continuous dependence, $\vec v_n$ converges to $\vec u_c$ on every
compact subinterval of the lifespan of $\vec u_c$. The uniform
bound $A$, together with $a_n,b_n\to\infty$, excludes a finite endpoint
by the continuation criterion. Thus $\vec u_c$ exists on all of
$\mathbb R$ and satisfies
\begin{equation}\label{exterior:critical}
E(\vec u_c)=E_Q,\qquad
\sup_t\|\vec u_c(t)\|_{X^1}\leq A,\qquad
d_Q(\vec u_c(t))\geq r,\qquad K_j(\vec u_c(t))\geq\kappa.
\end{equation}
Its scattering norm is infinite in both time directions. Indeed,
if either half-line norm were finite, Proposition~\ref{pr001}, with
$\vec u_c$ as the approximate solution, would give a uniform bound
for the corresponding norm of $\vec v_n$ on $[-a_n,0]$ or $[0,b_n]$.
The initial discrepancy tends to zero in $X^1$, and both solutions
have the required energy-space bounds. This contradicts
\eqref{exterior:balanced}.

\medskip\noindent\textit{Step 2: compactness and rigidity.}
For any sequence $\tau_n\in\mathbb R$, apply
Lemma~\ref{lem:threshold-compactness} to
$\vec u_c(\tau_n+t)$ on the whole time axis.
The bound on $G$ in \eqref{exterior:coercivity} passes to this
solution, and both half-line scattering norms remain infinite after
every finite time translation. The lemma gives
\[
\vec u_c(\tau_n)=\U(-\theta_n)\vec\psi_*+o_{X^1}(1).
\]
An unbounded positive or negative sequence $\theta_n$ would make
the free norm small on the negative or positive half-line,
respectively. Remark~\ref{lRe302} would then make the corresponding
nonlinear norm finite, a contradiction. Hence $\theta_n$ is bounded
and $\vec u_c(\tau_n)$ has a strongly convergent subsequence.
The full orbit is therefore precompact in $X^1$, with no spatial
translation.

The forward orbit is precompact and $K_2(\vec u_c(t))\geq\kappa$
for every $t$. Applying \eqref{virial:tail} and integrating as above
rules out this global solution.
This proves the uniform bound. For the final assertion, the identity
$E=G+K_1/(p+2)$ also bounds the energy-space norm uniformly in time.
Scattering now follows from Lemma~\ref{le303}.
\end{proof}

\subsection{Classification for initially small perturbations}
We record explicitly the local consequence of the one-pass theorem;
it will be applied at an entry time into a standing-wave neighborhood.
\begin{lem}[Local classification]\label{lem:local-classification}
There exists $\delta_{\mathrm{in}}>0$ such that every even--odd
solution with initial data in
$U_{\delta_{\mathrm{in}}}(\pm\vec Q)$ either scatters forward,
blows up in finite positive time, or belongs to the corresponding
local center-stable manifold and remains near that standing wave.
In the last case the first component converges to the corresponding
standing wave profile in $L^2(I)\cap L^\infty(I)$ on every compact interval
$I\subset\mathbb R$, by \cite[Theorem~1]{ref-Ma2023}.
\end{lem}
\begin{proof}
By the symmetry $\vec u\mapsto-\vec u$, it suffices to work near
$\vec Q$. Fix the radii $\varepsilon_{\mathrm{op}}\ll R$ of
Proposition~\ref{pro001}, with its trapped neighborhood contained
in the ball of radius $\delta_0/2$ from Section~\ref{sec:spectral-local}. Apply
Proposition~\ref{prop:exterior-scattering} with $r=R/2$.
Choose $\delta_{\mathrm{in}}\leq\varepsilon_{\mathrm{op}}$ so small
that the quadratic energy expansion implies
\[
E_Q/2<E(\vec u_0)<E_Q+\min\{\eta_{R/2}^2,(c_0R/2)^2\},
\qquad \vec u_0\in U_{\delta_{\mathrm{in}}}(\vec Q).
\]
By the trapping bound in Section~\ref{sec:spectral-local}, we may decrease $\delta_{\mathrm{in}}$ so that solutions starting on $\mathcal M$ stay within both $\delta_0/2$ and the smallness radius in \cite[Theorem~1]{ref-Ma2023}.

For data outside $\mathcal M$, the trapped alternative in
Proposition~\ref{pro001} is excluded by the converse characterization of $\mathcal M$.
The solution therefore reaches $d_Q=R$ at a first exit time $t_0$
and stays outside, with the common sign of $K_1,K_2$
fixed at exit.
If the sign is positive, energy conservation bounds the solution in $X^1$,
so it exists for all forward times. It scatters by the exterior estimate
on $[t_0,\infty)$.

If the sign is negative, Proposition~\ref{pro106} at radius $R/2$
gives $K_1(\vec u(t))\leq-cR$ throughout the remaining lifespan.
Taking $t_0$ as the initial time in Lemma~\ref{lee802}, we conclude
that the solution blows up in finite time. This proves the local classification.
\end{proof}

\subsection{Proof of Theorem~\ref{thm002}}
Choose $r>0$ so small that $r<R_0/2$ and
\[
d_Q(\vec f)<r\quad\Longrightarrow\quad
\vec f\in U_{\delta_{\mathrm{in}}}(\pm\vec Q),
\]
using \eqref{distance:equivalence}. Take the energy width
$\varepsilon_E>0$ smaller than $\eta_r$ and $c_0r/2$.
It suffices by time reversal to prove the forward statement.

The zero solution scatters trivially. If $\vec u_0\ne0$ and
$E(\vec u_0)<E_Q$, Theorem~\ref{thm003} gives scattering or
finite-time blow-up. We may therefore assume
\[
E_Q\leq E(\vec u_0)<E_Q+\varepsilon_E^2.
\]
If the solution has a finite forward lifespan, the blow-up
alternative already holds. Otherwise it is global forward.
If $d_Q(\vec u(t_0))<r$ at some $t_0\geq0$, apply
Lemma~\ref{lem:local-classification} at that time. Its three
alternatives apply to the original solution. In the trapped case, the
sign of the nearby standing wave is fixed and the data at $t_0$ belong
to the corresponding copy of $\mathcal M$. The bound in Section~\ref{sec:spectral-local} controls the distance to that
wave, and \cite[Theorem~1]{ref-Ma2023} gives local convergence.

It remains to consider a solution with $d_Q(\vec u(t))\geq r$
for every $t\geq0$. Proposition~\ref{pro106}, applied with radius
$r/2$ and $\alpha=\varepsilon_E$, shows that $K_1,K_2$ have a
common, constant strict sign. If it is positive, the $X^1$ norm stays
bounded by energy conservation, and the solution scatters by
Proposition~\ref{prop:exterior-scattering}.
If it is negative, $K_1\leq-cr$ throughout $[0,\infty)$, contradicting Lemma~\ref{lee802}.
This exhausts the forward possibilities. The constant $\delta$ in
the theorem can be taken to be $C\delta_{\mathrm{in}}$, with $C$ from
the trapping bound in Section~\ref{sec:spectral-local}.
In the scattering case $\vec u_+=\lim_{t\to\infty}\U(-t)\vec u(t)$
belongs to the closed invariant subspace $X^1_{\mathrm{eo}}$.
The alternatives are disjoint: finite-time
blow-up excludes the global alternatives, and local convergence to
the nonzero profile $\pm Q$ is incompatible with scattering.
Indeed, for every $\vec f\in X^1$ and compact interval $K\subset\mathbb R$,
\eqref{free-decay} and H\"older's inequality give
$\|P\U(t)\vec f\|_{L^2(K)}\leq |K|^{1/4}\|P\U(t)\vec f\|_{L^4}\to0$
as $t\to\infty$.

Finally, $\mathcal T(f,g)=(f,-g)$ preserves $X^1_{\mathrm{eo}}$ and
the energy, and satisfies
\[
\mathcal T\U(t)=\U(-t)\mathcal T,\qquad
\vec w(t):=\mathcal T\vec u(-t)
\quad\hbox{solves \eqref{eq02}}.
\]
The backward trichotomy for $\vec u$ follows from the forward trichotomy
for $\vec w$. In the trapped case the center-stable manifold for $\vec w$
corresponds to $\mathcal T\mathcal M$ for $\vec u$. If $\vec w$ scatters
to $\vec w_+$ forward in time, then $\vec u_-=\mathcal T\vec w_+$.
This proves Theorem~\ref{thm002} for $p>6$. The extension in the
power is proved in Appendix~\ref{app:powers}.

\subsection{Realization of the nine two-sided behaviors}\label{sec:nine}

We prove Remark~\ref{rem:nine} near \(\vec Q\). Consider the two-dimensional
family
\[
 U(a,b)=\vec Q+a\vec Y_++b\vec Y_-,\qquad |a|,|b|<\rho .
\]
Take \(\rho>0\) so small that this square lies in the neighborhood of
Lemma~\ref{lem:local-classification} and in the energy region of
Theorem~\ref{thm002}. Write \(H(b)=h(b\vec Y_-)\). By
the graph formula in Section~\ref{sec:spectral-local} and \(\mathcal T U(a,b)=U(b,a)\), the forward
and backward trapping curves in the square are respectively
\[
 a=H(b),\qquad b=H(a),\qquad |H(s)|\le C|s|^{3/2}.
\]
By the local classification, solutions starting off the forward graph
leave the small neighborhood and then either scatter or blow up.
The backward characterization follows by time reversal.

We identify the forward outcomes on the two sides of the first graph.
Every datum off the graph first reaches \(d_Q=R\) at a finite time.
This crossing is transverse: a zero derivative at that time,
together with the strict convexity of \(d_Q^2\) in
\eqref{ejection:convex}, would force \(d_Q>R\) immediately before the
first crossing. At the crossing, \(K_1\ne0\) by Proposition~\ref{pro106}
with radius \(R/2\), provided \(\rho\) is small relative to \(R\). Continuous dependence on
the initial data up to a little past that time now makes the sign of
\(K_1\) at first exit locally constant. Indeed the reference orbit
stays strictly inside the boundary on each compact interval preceding
the crossing, and \(K_1\) has a fixed strict sign in a neighborhood
of the crossing. By Proposition~\ref{pro001} this first exit is final.
The proof of Lemma~\ref{lem:local-classification} shows that a positive
exit leads to scattering and a negative exit to finite-time blow-up.

The sets \(a>H(b)\) and \(a<H(b)\) in the square are both connected:
for small \(\rho\), \(\sup_{|b|<\rho}|H(b)|<\rho/4\), so each point
can first be joined horizontally to the line \(a=3\rho/4\) or
\(a=-3\rho/4\) on its own side. It remains to distinguish the two
outcomes. On the diagonal,
\(U(s,s)=(Q+2s\phi_0,0)\). The eigenfunction equation and the
normalization $\|\partial_x^{-1}\phi_0\|_{L^2}=1$ give
\(\langle L\phi_0,\phi_0\rangle=-\nu_0^2\). Thus
\[
 E(U(s,s))=E(\vec Q)-2\nu_0^2s^2+O(|s|^3),\qquad
 K_1(U(s,s))=2s\langle\widetilde K_1'(Q),\phi_0\rangle+O(s^2).
\]
The linear coefficient is nonzero by Lemma~\ref{le102}: its vanishing
would imply \(\langle L\phi_0,\phi_0\rangle>0\).
By Theorem~\ref{thm003}, the two signs of \(s\) correspond to opposite
sub-threshold outcomes. Thus solutions on one side of the graph scatter,
while those on the other side blow up. By time reversal, the same
assignment applies to the sign of \(b-H(a)\) in backward time.

Finally take \(a,b\) independently equal to \(s\) or \(-s\) for
\(0<s<\rho/2\) small. Since \(H(\pm s)=o(s)\), these four choices
realize all pairs with neither direction trapped. For \(b=\pm s,\ a=H(b)\), the solution is trapped forward in time.
Both backward outcomes occur since \(b-H(H(b))\) has the sign of \(b\).
Interchanging \(a,b\) reverses the two time directions.

The standing wave is trapped in both time directions. To construct a
nonstationary solution with the same behavior, fix a nonzero even test function
\(\psi\) with \(\langle(\psi,0),\vec Z_+\rangle=0\) and put
\(e=(\psi,0)\). The data
\(\vec Q+\tau e+s(\vec Y_++\vec Y_-)\) are fixed by \(\mathcal T\).
For small \(\tau\ne0\), the continuous map
\(s\mapsto h(\tau e+s\vec Y_-)\) sends \([-|\tau|,|\tau|]\)
into itself, by the \(3/2\)-power bound. At a fixed point the data
belong to \(\mathcal M\) and hence are trapped in both directions.
They differ from \(\vec Q\) since \(e\) and \(\vec Y_++\vec Y_-\)
are independent. The only even stationary state sufficiently close
to \(\vec Q\) is \(\vec Q\): the first integral of \eqref{eq06} at
\(c=0\) fixes the value at the origin for a positive central value,
and evenness fixes the first derivative there. All the data constructed
above can be chosen arbitrarily close to \(\vec Q\).

\appendix
\section{Auxiliary estimates for the strict scattering threshold}\label{app:gap}
\indent
\par
We collect here the estimates used in
Proposition~\ref{prop:strict-gap}. Throughout,
$\Lambda=(1+D^2)^{1/2}$, $\omega(D)=D\Lambda$,
$S_0(t)=e^{-it\omega(D)}$, $k=\|\Lambda^{1/2}z\|_{L^{2}}^2$, and
$\mathscr D(z)=\|\partial_x|z|^2\|_{L^{2}}^2$.
All Fourier transforms use the unitary convention stated in Section~1.

\subsection{The free density estimate}\label{app:free}
\begin{lem}[Free density estimate]\label{lem:gap-free}
For every complex $z\in H^1(\mathbb R)$,
\begin{equation}\label{gap:free}
\int_{\mathbb R}\|\partial_x|S_0(s)z|^2\|_{L^{2}}^2\,ds
\leq\frac23\|\Lambda^{1/2}z\|_{L^{2}}^4.
\end{equation}
More precisely, there is a bounded linear operator
$\mathcal B:L^2_{\xi,\eta}\to L^2_{s,q}$ with
$\|\mathcal B\|^2\leq2/3$ such that, for
$W=\Lambda^{1/2}z$ and
$h_W(\xi,\eta)=(\mathcal{F}W)(\xi)\overline{(\mathcal{F}W)(\eta)}$,
\[
\mathcal Bh_W=\mathcal{F}_x\bigl(\partial_x|S_0(s)z|^2\bigr).
\]
\end{lem}
\begin{proof}
Put $\lambda(\xi)=\sqrt{1+\xi^2}$. We begin with the symbol bound
\begin{equation}\label{gap:symbol}
\frac{(\xi-\eta)^2}{|\partial_\xi\omega(\xi)-\partial_\eta\omega(\eta)|}
\leq\frac23\lambda(\xi)\lambda(\eta)
\quad\text{for almost every }(\xi,\eta).
\end{equation}
For $x\geq0$, set $g(x)=(1+2x^2)/\sqrt{1+x^2}$. Then
$\partial_\xi\omega(\xi)=\operatorname{sgn}(\xi)g(|\xi|)$ away from zero and
\[
\partial_xg(x)=\frac{x(3+2x^2)}{(1+x^2)^{3/2}},\qquad
\partial_x^2g(x)=\frac3{(1+x^2)^{5/2}}>0.
\]
For $\xi=\eta+q$ with $\eta\geq0$, $q>0$, convexity implies
$g(\eta+q)-g(\eta)\geq g(q)-g(0)$, and
$\lambda(\eta+q)\lambda(\eta)\geq\lambda(q)$. Hence
\[
\bigl[g(\eta+q)-g(\eta)\bigr]\lambda(\eta+q)\lambda(\eta) \geq[g(q)-1]\lambda(q) =1+2q^2-\sqrt{1+q^2}\geq\tfrac32q^2,
\]
where the last step uses $\sqrt{1+q^2}\leq1+q^2/2$.
This proves \eqref{gap:symbol} for nonnegative frequencies; reflection
settles the nonpositive case. For opposite signs write $\xi=x\geq0$,
$\eta=-y\leq0$. The elementary inequalities
\[
g(x)\geq2x,\qquad
\lambda(x)\lambda(y)\geq x+y
\]
follow by squaring, since the respective squared differences are
$1/(1+x^2)$ and $(xy-1)^2$. Thus
\[
\frac{(x+y)^2}{g(x)+g(y)}\leq\frac{x+y}{2}
\leq\frac12\lambda(x)\lambda(y),
\]
which is stronger than the required bound.

For smooth compactly supported $h$, define
\[
(\mathcal Bh)(s,q)=\frac{iq}{\sqrt{2\pi}}
\int_{\mathbb R}e^{-is[\omega(\eta+q)-\omega(\eta)]}
\frac{h(\eta+q,\eta)}{\lambda(\eta+q)^{1/2}\lambda(\eta)^{1/2}}\,d\eta.
\]
Fix $q\ne0$ and let $\Phi_q(\eta)=\omega(\eta+q)-\omega(\eta)$.
The function $\partial_\xi\omega$ is strictly increasing on each half-line and
has an upward jump at zero. Hence $\Phi_q$ is strictly monotone,
piecewise smooth, and
$|\partial_\eta\Phi_q(\eta)|=|\partial_\eta\omega(\eta+q)-\partial_\eta\omega(\eta)|>0$ almost everywhere.
Changing variables by $\nu=\Phi_q(\eta)$ on each smooth piece and
applying Plancherel in $s$, we find
\[
\int_{\mathbb R}|(\mathcal Bh)(s,q)|^2\,ds
=q^2\int_{\mathbb R}
\frac{|h(\eta+q,\eta)|^2}
{|\partial_\eta\Phi_q(\eta)|\lambda(\eta+q)\lambda(\eta)}\,d\eta.
\]
Integrating in $q$ and setting $\xi=\eta+q$, we obtain
\[
\|\mathcal Bh\|_{L^{2}}^2
=\iint\frac{(\xi-\eta)^2|h(\xi,\eta)|^2}
{|\partial_\xi\omega(\xi)-\partial_\eta\omega(\eta)|\lambda(\xi)\lambda(\eta)}\,d\xi\,d\eta
\leq\frac23\|h\|_{L^{2}}^2.
\]
The values at $q=0$ and at the finitely many nonsmooth points do not
affect the integrals. Density extends $\mathcal B$ to all of
$L^2(\mathbb R^2)$ with the asserted norm bound.

For smooth $z$, substitute
$(\mathcal Fz)(\xi)=\lambda(\xi)^{-1/2}(\mathcal FW)(\xi)$
in the Fourier product formula to obtain
\[
\mathcal F_x\bigl(\partial_x|S_0(s)z|^2\bigr)(q)
=(\mathcal Bh_W)(s,q).
\]
Since $\|h_W\|_{L^{2}}^2=\|W\|_{L^{2}}^4=k^2$, the operator bound proves
\eqref{gap:free}. Approximation of $z$ in $H^1$ preserves the
density expression on compact time intervals and the expression $\mathcal Bh_W$
in $L^2_{s,q}$, establishing the identity and estimate for all $H^1$
data. The same formula also defines the quartic functional
$\mathfrak M$ continuously on $H^{1/2}$.
\end{proof}

\subsection{Density and fractional chain estimates}\label{app:density}
\begin{lem}[Two density inequalities with explicit constants]\label{lem:gap-density}
Let $\rho\geq0$ belong to $H^1(\mathbb R)\cap L^1(\mathbb R)$ and put
$n=\|\rho\|_{L^{1}}$, $d=\|\partial_x\rho\|_{L^{2}}^2$, $H=\|\rho\|_{L^{\infty}}$. Then
\begin{equation}\label{gap:density}
H^3\leq\frac9{16}nd,\qquad
\int_{\mathbb R}\rho^{13}\,dx\leq\frac1{64}n^5d^4.
\end{equation}
\end{lem}
\begin{proof}
The inequalities are immediate for the zero function. Otherwise $n,d,H>0$.
The continuous representative of $\rho$ tends to zero at both
infinities and attains its maximum $H$. Integrating toward both infinities from a maximum
point, we obtain
\[
\frac43H^{3/2}
\leq\int\sqrt\rho\,|\partial_x\rho|\,dx
\leq\left(\int\rho\right)^{1/2}\left(\int|\partial_x\rho|^2\right)^{1/2}.
\]
Squaring proves the first inequality.

To prove the integral inequality, write $P_{13}=\int\rho^{13}$ and
\[
I_0=\int_0^1\sqrt{s-s^{13}}\,ds,\qquad
F_H(r)=\int_0^r\sqrt{t-t^{13}/H^{12}}\,dt\quad(0\leq r\leq H).
\]
Since $F_H(H)=H^{3/2}I_0$, the same two-sided variation argument, followed by
Cauchy--Schwarz, yields
\[
2H^{3/2}I_0
\leq\int\sqrt{\rho-\rho^{13}/H^{12}}\,|\partial_x\rho|\,dx
\leq\bigl(n-P_{13}/H^{12}\bigr)^{1/2}d^{1/2}.
\]
It follows that
\[
P_{13}\leq nH^{12}-\frac{4I_0^2}{d}H^{15}.
\]
To remove $H$, put $y=H^3$ and maximize
$ny^4-4I_0^2y^5/d$ over $y\geq0$. Its maximum occurs at
$y=nd/(5I_0^2)$ and equals $n^5d^4/(5^5I_0^8)$.

It remains to bound $I_0$ from below. The substitution $x=s^{3/2}$
followed by integration in the opposite order on
$\{(x,y)\in[0,1]^2:x^8+y^2\leq1\}$ gives
\[
 I_0=\tfrac23\int_0^1(1-y^2)^{1/8}\,dy
 \geq\tfrac23\exp\!\left(\tfrac18\int_0^1\log(1-y^2)\,dy\right)
 =\tfrac23(2/e)^{1/4}.
\]
Here we used Jensen's inequality and
$\int_0^1\log(1-y^2)\,dy=2\log2-2$.
Thus $(5^5 I_0^8)^{-1}\leq 3^8e^2/(2^{10}5^5)<1/64$, using
$e<11/4$. This proves the second inequality in \eqref{gap:density}.
The argument applies directly to the absolutely continuous
representative of $\rho$; alternatively, the estimate follows by nonnegative
smooth approximation in $H^1\cap L^1$.
\end{proof}

\begin{lem}[An explicit fractional chain estimate]\label{lem:gap-chain}
For real $u\in H^1(\mathbb R)$ and $p>0$, let $f(u)=|u|^pu$. Then
\begin{equation}\label{gap:chain}
\|D^{1/2}f(u)\|_{L^{2}}
\leq\frac{p+1}{\sqrt{2p+1}}
\left(\int|u|^{4p+2}\right)^{1/4}\|\partial_xu\|_{L^{2}}^{1/2}.
\end{equation}
\end{lem}
\begin{proof}
Set $g(s)=|s|^{2p}s$. For $a>b$, the fundamental theorem of calculus
and Cauchy--Schwarz imply
\[
\begin{aligned}
|f(a)-f(b)|^2
&=(p+1)^2\left(\int_b^a|s|^p\,ds\right)^2\\
&\leq(p+1)^2(a-b)\int_b^a|s|^{2p}\,ds
=\frac{(p+1)^2}{2p+1}(a-b)(g(a)-g(b)).
\end{aligned}
\]
The inequality also holds for $a\leq b$ by symmetry.
For real smooth functions, the bilinear representation of $D$ is
\[
(h,D\ell)_{L^2}
=\frac1{2\pi}\iint
\frac{(h(x)-h(y))(\ell(x)-\ell(y))}{(x-y)^2}\,dx\,dy.
\]
Apply the scalar inequality with $a=u(x)$, $b=u(y)$ and integrate
against this kernel. We obtain
\[
\begin{aligned}
\|D^{1/2}f(u)\|_{L^{2}}^2
&\leq\frac{(p+1)^2}{2p+1}(g(u),Du)_{L^2}\\
&\leq\frac{(p+1)^2}{2p+1}\|g(u)\|_{L^{2}}\|Du\|_{L^{2}}
=\frac{(p+1)^2}{2p+1}
\left(\int|u|^{4p+2}\right)^{1/2}\|\partial_xu\|_{L^{2}}.
\end{aligned}
\]
Taking square roots proves \eqref{gap:chain}. For general $u\in H^1$,
smooth approximation converges also in $L^\infty$ and in every finite
$L^q$, $q\geq2$; the nonlinear terms and both sides therefore pass to
the limit. 
\end{proof}

\subsection{The error estimate from complex orthogonality}\label{app:product}
\begin{lem}[Complex-orthogonal variation]\label{lem:gap-product}
Let $z\in H^{1/2}(\mathbb R)\setminus\{0\}$, $W=\Lambda^{1/2}z$, $k=\|W\|_{L^{2}}^2$, and let
$b_\perp\in L^2$ satisfy $(b_\perp,W)_{\mathbb C}=0$.
For the quartic functional in Proposition~\ref{prop:strict-gap},
\[
|d\mathfrak M(z)[i\Lambda^{-1/2}b_\perp]|
\leq\frac{2\sqrt2}{3}k^{3/2}\|b_\perp\|_{L^{2}}.
\]
For the nonlinear projection in that proposition,
$\|b_\perp\|_{L^2}\le\|D^{1/2}f(u)\|_{L^2}$.
\end{lem}
\begin{proof}
The real derivative of
$h_W(\xi,\eta)=(\mathcal FW)(\xi)\overline{(\mathcal FW)(\eta)}$
in the direction $ib_\perp$ is
\[
\delta h(\xi,\eta)
=i(\mathcal Fb_\perp)(\xi)\overline{(\mathcal FW)(\eta)}
-i(\mathcal FW)(\xi)\overline{(\mathcal Fb_\perp)(\eta)}.
\]
By Fubini's theorem, each product on the right has squared
$L^2_{\xi,\eta}$ norm $k\|b_\perp\|_{L^{2}}^2$. The cross term factors as
\[
\begin{aligned}
&\iint (\mathcal Fb_\perp)(\xi)\overline{(\mathcal FW)(\xi)}
(\mathcal Fb_\perp)(\eta)\overline{(\mathcal FW)(\eta)}\,d\xi\,d\eta\\
&\quad=\left(\int(\mathcal Fb_\perp)(\xi)
\overline{(\mathcal FW)(\xi)}\,d\xi\right)^2
=(b_\perp,W)_{\mathbb C}^2=0.
\end{aligned}
\]
The two products are orthogonal in $L^2(\mathbb R^2)$. Thus
$\|\delta h\|_{L^{2}}^2=2k\|b_\perp\|_{L^{2}}^2$.
Since $\Sigma$ is self-adjoint and has norm one, we may differentiate
$\mathfrak M=\tfrac12(\mathcal Bh_W,\Sigma\mathcal Bh_W)_{\mathbb C}$
and estimate
\[
|d\mathfrak M(z)[i\Lambda^{-1/2}b_\perp]| =|\operatorname{Re}(\mathcal B\delta h,\Sigma\mathcal Bh_W)_{\mathbb C}| \leq\|\mathcal B\|^2\|\delta h\|_{L^{2}}\|h_W\|_{L^{2}} \leq\frac{2\sqrt2}{3}k^{3/2}\|b_\perp\|_{L^{2}}.
\]
For $b=\Lambda^{1/2}r=D\Lambda^{-1/2}f(u)$,
we use the contractivity of orthogonal projection and Plancherel to obtain
\[
\|b_\perp\|_{L^{2}}^2\leq\|b\|_{L^{2}}^2
=\int\frac{\xi^2}{\sqrt{1+\xi^2}}|\mathcal{F}[f(u)](\xi)|^2\,d\xi
\leq\int|\xi|\,|\mathcal{F}[f(u)](\xi)|^2\,d\xi.
\]
Together these estimates give \eqref{gap:projected-bound}. A merely real-orthogonal projection would
not in general eliminate the cross term between the two products in the variation; this explains the
complex projection in the normalization argument.
\end{proof}

\subsection{Verification of the strict margins}\label{app:constants}
\begin{lem}[Uniform scalar margins]\label{lem:gap-constants}
There is a continuous function $F:(0,1)\to(0,\infty)$ such that
\[
 \frac{cpS_Q}{p+2-2c^3}\leq F(c)\qquad (p\geq6,\ 0<c<1).
\]
Moreover, there is an $L>F(1/2)$ for which
$\sup_{p\geq6}a_p L^{p/2}<1$, with $a_p$ as in \eqref{gap:absorb}.
\end{lem}
\begin{proof}
Testing the stationary equation against $Q$ gives $S_Q=\int Q^{p+2}$.
In the explicit profile for $Q=Q_0$ in Section~1.1, put
$y=\tanh(px/2)$. By evenness and Jensen's inequality for the concave
function $s^{2/p}$,
\[
 S_Q=\frac{2(p+2)}p\left(\frac{p+2}{2}\right)^{2/p}
       \int_0^1(1-y^2)^{2/p}\,dy
 \leq \frac{2(p+2)}p\left(\frac{p+2}{3}\right)^{2/p}.
\]
It follows that $cpS_Q/(p+2-2c^3)$ is bounded above by
\[
 2c\,\frac{p+2}{p+2-2c^3}\left(\frac{p+2}{3}\right)^{2/p}.
\]
Both factors depending on $p$ decrease for $p\geq6$. For the second,
the logarithmic derivative has the sign of
$p/(p+2)-\log((p+2)/3)$, which is negative since
\[
 \log\frac{p+2}{3}\geq\frac{2(p-1)}{p+5}>\frac{p}{p+2}
\]
by $\log x\geq 2(x-1)/(x+1)$ for $x\geq1$. We may therefore take
$F(c)=\frac{16c}{8-2c^3}(8/3)^{1/3}$.

Let $L=3^{1/3}$. At half energy, $F(1/2)^3=(32/31)^3(8/3)<3$,
since $(1-1/32)^3\geq1-3/32>8/9$. To bound the error uniformly, note
that $L<3/2$ and
\[
 \frac{d}{dp}\log(a_p L^{p/2})
 =\frac{p}{(p+1)(2p+1)}-\tfrac12\log(2/L)<0 .
\]
Indeed the first term is less than $1/(2p)$, whereas
$\tfrac12\log(2/L)>\tfrac12\log(4/3)>1/8$. Hence
$\sup_{p\geq6}a_pL^{p/2}=3a_6=7/(2\sqrt{13})<1$.
\end{proof}

\section{A compactness consequence at the ground-state energy}\label{app:threshold-compactness}
This lemma uses only the profile decomposition and scattering
results of Section~3. It does not depend on the one-pass theorem or
the energy-region classification in Section~5.
\begin{lem}[Compactness at the ground-state energy]
\label{lem:threshold-compactness}
Let $p>6$. Let $I_n$ be intervals containing $0$, and let $\vec v_n$ be even-odd
solutions on $I_n$. Suppose that, for some $A,\kappa>0$,
\[
\sup_n\sup_{t\in I_n}\|\vec v_n(t)\|_{X^1}\leq A,\qquad
E(\vec v_n)\longrightarrow E(\vec Q),\qquad
G(\vec v_n(0))\leq E(\vec Q)-\kappa,
\]
and
\[
\|P\vec v_n\|_{S(I_n)}\longrightarrow\infty.
\]
After passing to a subsequence, there are
$\vec\psi\in X^1_{\mathrm{eo}}$ and $s_n\in\mathbb R$ such that
\[
\vec v_n(0)=\U(-s_n)\vec\psi+o_{X^1}(1),\qquad
E(\U(-s_n)\vec\psi)\longrightarrow E(\vec Q).
\]
In particular, no spatial translation is needed.
\end{lem}
\begin{proof}
Apply the paired profile decomposition to $\vec v_n(0)$. The strict
bound on $G$ places every nonzero constituent and remainder below the
quadratic ground-state barrier. As in Step~1 of
Proposition~\ref{proo001}, each has nonnegative energy and $K_1\geq0$, with strict inequality for nonzero data. Their limiting energies satisfy
\[
\sum_{j=1}^Lm_je_j\leq E(\vec Q).
\]
Every escaping constituent has $e_j\leq E(\vec Q)/2$ and therefore
scatters by Proposition~\ref{prop:strict-gap}. Every centered
constituent with $e_j<E(\vec Q)$ scatters by
Proposition~\ref{proo001}. For a constituent whose time parameter diverges, we construct the nonlinear profile
using Proposition~\ref{le305}. The construction is needed only for
$e_j<E(\vec Q)$, within the energy range of that proposition.

If all centered constituents had energy strictly below $E(\vec Q)$,
the nonlinear profile approximation in Step~2 of
Proposition~\ref{proo001} would produce approximate solutions with
uniform $X^1$ and $S(\mathbb R)$ bounds. The square-summability and
error estimates proved there are unchanged when the energy budget is
$E(\vec Q)$. On $I_n$, the assumed energy-space bound allows us to apply
Proposition~\ref{pr001} and obtain a uniform bound for
$\|P\vec v_n\|_{S(I_n)}$. This contradicts the assumed divergence of the scattering norms.
Thus one centered constituent carries energy $E(\vec Q)$.
All other constituent energies vanish, and the remainder satisfies
\[
\frac{p}{2(p+2)}\|\vec r_n\|_{X^1}^2
=G(\vec r_n)\leq E(\vec r_n)\longrightarrow0.
\]
This proves the lemma.
\end{proof}
\section{Continuation in the power at \texorpdfstring{\(p=6\)}{p=6}}\label{app:powers}

Write \(E_p\) and \(K_{1,p}\) for the energy and the Nehari functional at power \(p\).
Let \(Q^{(p)}\) denote the corresponding positive standing-wave profile
(\(Q_0\) in the notation of Section~1.1), and put \(e_p=E_p(Q^{(p)},0)\).
We show that the strict general-data threshold persists for all \(p\)
sufficiently close to \(6\). We then give the exponent choices needed
to apply the compactness and one-pass arguments at these powers.

\paragraph{Estimates in fixed spaces.}
On any time interval \(I\) use
\[
 S_*(I)=L_t^7L_x^{20}(I\times\mathbb R),\qquad
 N_*(I)=L_t^{140/97}L_x^{20/19}(I\times\mathbb R).
\]
In this appendix all bounds on the power are uniform for
\(|p-6|\le h\), where \(h=1/24\).
For the pair \((7,20)\), the Sobolev index in Lemma~\ref{le302} is
\(s=23/140\). The free multiplier has low-frequency order
\(s-1/7=3/140>0\) and is bounded from \(H^1\) to \(L^2\).
The source multiplier in \eqref{e203}, after diagonalization
\eqref{gap:diag}, is \(K(D)^{11/20}\), bounded on \(L^{20/19}\).
Thus the free solution is bounded in \(S_*\) by its energy norm and
the Duhamel term with scalar source \(e\) is bounded there by
\(C\|e\|_{N_*}\), for full, retarded or advanced integration.

Let \(\alpha=97/20\), \(f_p(v)=|v|^pv\), and set
\[
 m_r=\frac{400(r+1-\alpha)}{283},\qquad
 \frac{97}{140}=\frac{\alpha}{7},\qquad
 \frac{19}{20}=\frac{\alpha}{20}+\frac{r+1-\alpha}{m_r}.
\]
For \(|r-6|\le h\) we have \(2<m_r<\infty\). Hence
\begin{equation}\label{pow:holder}
 \big\||v|^{r+1}\big\|_{N_*(I)}
 \le \|v\|_{S_*(I)}^\alpha
       \|v\|_{L_t^\infty L_x^{m_r}(I)}^{r+1-\alpha}.
\end{equation}
For \(\|v\|_{L_t^\infty H^1}+\|d\|_{L_t^\infty H^1}\le A\), the same
allocation of factors to \(|v|^p|d|+|d|^{p+1}\) yields
\[
 \|f_p(v+d)-f_p(v)\|_{N_*}
 \le C_A\big(\|v\|_{S_*}^{\alpha-1}+\|d\|_{S_*}^{\alpha-1}\big)
             \|d\|_{S_*}.
\]
In particular Proposition~\ref{pr001} holds with \(S_*,N_*\).
Indeed, the subdivision into intervals on which the comparison
solution has \(S_*\)-norm at most \(\eta\) uses at most
\(1+(B/\eta)^7\) intervals per half-line if its full norm is at
most \(B\). The preceding difference estimate gives the same
bootstrap and propagation of the free discrepancy as
\eqref{stab:interval}--\eqref{stab:local}, uniformly in \(p\).

To pass from finite \(S_*\)-norm to energy scattering, take
\(\beta=35/6\), \(n_p=24(p-\beta)>2\). Then
\[
 \frac56=\frac{\beta}{7},\qquad
 \frac13=\frac{\beta}{20}+\frac{p-\beta}{n_p}.
\]
By \eqref{e211}, on any subinterval \(I_0\) where the energy norm is
bounded by \(A\),
\begin{equation}\label{pow:tail}
 \left\|\int_{I_0}\U(-t)J\mathcal N(\vec u(t))\,dt\right\|_{X^1}
 \le C_A\|u\|_{S_*(I_0)}^\beta .
\end{equation}
This proves the Cauchy criterion on both tails. The same estimates
give the small-data bound and the wave operators of
Lemma~\ref{le304} and Proposition~\ref{le305}. For clarity, in the
tail contraction the difference in the differentiated nonlinearity
is bounded pointwise by a constant times
\[
 (|u|+|v|)^p|\partial_x(u-v)|
 +(|u|+|v|)^{p-1}|u-v|(|\partial_xu|+|\partial_xv|).
\]
In the first term place \(\beta\) undifferentiated factors in
\(S_*\), and in the second place \(u-v\) and \(\beta-1\) other
factors there. The remaining \(p-\beta\) factors lie in
\(L_t^\infty L_x^{n_p}\), and the derivative in \(L_t^\infty L_x^2\).
The zeroth order term obeys the same estimate. Together with
\eqref{pow:holder} this makes the advanced Duhamel map a contraction
in \(L_t^\infty X^1\) and \(S_*\) when the free tail is sufficiently
small in \(S_*\) for its fixed energy bound.

\paragraph{Transferring the strict threshold.}
At \(p=6\), Steps 1--5 of Proposition~\ref{prop:strict-gap} and the
integration of \eqref{gap:monotonicity} in Step 6 give numbers
\(c_6\in(1/2,1)\) and \(B_6<\infty\) such that
\[
 \sup_t\|\vec w(t)\|_{X^1}+\int_{\mathbb R}\|w(t)\|_{L^\infty}^6\,dt
 \le B_6
 \quad\text{if } E_6(\vec w(0))<c_6e_6,\quad K_{1,6}(\vec w(0))>0 .
\]
Indeed the density and absorption estimates in those steps include
\(p=6\); the auxiliary factor with exponent \(p-6\) then equals one.
The explicit value of the original scattering pair at the endpoint
is unnecessary. By spatial interpolation,
\[
 \int_{\mathbb R}\|w\|_{L^{20}}^7\,dt
 \le \sup_t\|w\|_{L^2}^{7/10}\sup_t\|w\|_{L^\infty}^{3/10}
       \int_{\mathbb R}\|w\|_{L^\infty}^6\,dt .
\]

Fix \(c\in(1/2,c_6)\) and take arbitrary real data \(\vec g\) in
the positive well with \(E_p(\vec g)<ce_p\). Write
\(M(\vec g)=\|\vec g\|_{X^1}^2\), \(H_p=\|Q^{(p)}\|_{H^1}^2\).
The identities
\[
 E_p=\frac{p}{2(p+2)}M+\frac{1}{p+2}K_{1,p},\qquad
 e_p=\frac{p}{2(p+2)}H_p
\]
imply \(M(\vec g)<cH_p\). Invariance of the positive well preserves
this bound along the \(p\)-flow, so the solution exists globally.
By the explicit formula for $Q^{(p)}$ in Section~1.1, \(H_p,e_p\)
depend continuously on the power.
Thus \(M(\vec g)<H_6\) for \(p\) close to \(6\), and the sharp Sobolev
inequality at \(p=6\) gives
\(K_{1,6}(\vec g)\ge M(\vec g)[1-(M(\vec g)/H_6)^3]>0\).
Moreover \(|E_p(\vec g)-E_6(\vec g)|\le C|p-6|\) uniformly on
this ball. To see this, differentiate \(|g|^{r+2}/(r+2)\) in \(r\)
and bound its derivative by \(C|g|^2\), using
\(\|g\|_{L^\infty}\lesssim\|\vec g\|_{X^1}\), where \(g=P\vec g\).
By shrinking the interval about \(6\), we obtain
\(E_6(\vec g)<c_6e_6\) for the entire data class.

Let \(\vec u\) and \(\vec w\) be the solutions with initial data \(\vec g\)
at powers \(p\) and \(6\), respectively, and write \(u=P\vec u\), \(w=P\vec w\).
For \(|p-6|\le h/2\) the mean value theorem gives
\[
 |f_p(w)-f_6(w)|
 \le C_h|p-6|\big(|w|^{7-h}+|w|^{7+h}\big).
\]
Applying \eqref{pow:holder} with \(r=6\pm h\) shows that the scalar
error is at most \(C|p-6|\) in \(N_*(\mathbb R)\), uniformly in
\(\vec g\). The two solutions have the same initial data and uniform
energy bounds. By the fixed-space perturbation estimate, there are constants
\(0<\varrho<h/2\) and \(C<\infty\) such that
\begin{equation}\label{pow:gap}
 \|u\|_{S_*(\mathbb R)}\le C
 \quad\text{for } |p-6|<\varrho,\
 E_p(\vec g)<ce_p,\ K_{1,p}(\vec g)>0 .
\end{equation}
In particular the general-data threshold defined with \(S_*\) is at
least \(ce_p>e_p/2\). Energy scattering follows from
\eqref{pow:tail}.

\paragraph{Application to the even--odd dynamics.}
We indicate explicitly the changes of exponents in the profile
argument. The first component of the free remainder is small in \(L_t^\infty L_x^3\) by
the same extraction as in Section~3. Interpolate by
\[
 \|v\|_{S_*}\le
 \|v\|_{L_t^{63/10}L_x^{54}}^{9/10}
 \|v\|_{L_t^\infty L_x^3}^{1/10}.
\]
The first factor for a free remainder is bounded by its energy norm,
since \(3/(63/10)+1/54=187/378<1/2\) in \eqref{e202}.
By spacetime separation, products of the first components of distinct profiles
tend to zero in \(L_t^{7/2}L_x^{10}\). For the nonlinear cross terms in \(N_*\),
\eqref{pow:holder} places the first component of each scattering profile in
\(L_t^{(p+1)140/97}L_x^{(p+1)20/19}\). Approximating in this space by
compactly supported functions and applying \eqref{profile:nonlinear-error},
we conclude that the nonlinear error tends to zero in \(N_*\).
For a scattering solution the \(S_*\)-norm tends to zero on the corresponding
time tail, by the argument of Remark~\ref{lRe302} with the homogeneous
estimate for \(S_*\) and \eqref{pow:holder}.
The endpoint divergence in Proposition~\ref{proo001} also follows
with $S_*$, since on $[0,T]$ the limiting standing wave has norm
$T^{1/7}\|Q^{(p)}\|_{L^{20}}$.
Thus all the norm estimates in the profile argument and
in Appendix~\ref{app:threshold-compactness} hold with these fixed spaces.
In particular, at the threshold \(e_p\) each constituent of an
escaping reflected pair has energy at most \(e_p/2\) and lies
strictly below \eqref{pow:gap}. The centered-profile rigidity,
variational estimates and regularized concavity argument then apply
as in Sections~3--5. In the positive case of Proposition~\ref{pro001},
let $J_n$ be the final segment of a hypothetical return interval.
By the ejection estimate, $|J_n|\to\infty$. Since $u_n$ stays close
to $Q^{(p)}$ or $-Q^{(p)}$ in $H^1$, we have $\|u_n(t)\|_{L^{20}}\geq c_p>0$
throughout $J_n$.
Hence $\|u_n\|_{S_*(J_n)}\geq c_p|J_n|^{1/7}\to\infty$, as
required for the same compactness argument on the return interval.
This proves Theorems~\ref{thm002} and~\ref{thm003} for
\(6-\varrho<p\le6\), including the exit alternatives used
in Remark~\ref{rem:nine}. The size of \(\varrho\) depends on
the uniform perturbation bound in \eqref{pow:gap}.

\paragraph{A limitation of the direct absorption estimate.}
For \(p>5\), the perturbative scattering criterion of Chen and Shao
\cite[Proposition~3]{ref-CS2026} assumes that the free evolution is
sufficiently small in their spacetime norm \(Y_{1,p}\). To treat
escaping profiles throughout this range, our argument requires a
uniform bound on the nonlinear scattering norm for all data in the
general positive well below \(ce_p\) for some \(c>1/2\).
A family of broad wave packets illustrates a
limitation of the direct proof of that estimate. Fix \(k_0>0\) and
a nonzero, nonnegative Schwartz function \(\phi\) of compact Fourier
support. Set \(z_\ell=\ell^{1/2}e^{ik_0x}\phi(\ell x)\) for
\(0<\ell\ll1\). In the notation of \eqref{gap:diag},
\[
 \|z_\ell\|_{H^1}\sim1,\qquad k\sim1,\qquad
 \mathscr D(z_\ell)=\ell^3\|\partial_x(\phi^2)\|_{L^2}^2,\qquad
 \|b_\perp\|_{L^2}\gtrsim\ell^{p/2}.
\]
For the last bound test against
\(\ell^{1/2}e^{-ik_0x}\phi(\ell x)\), which is orthogonal to
\(W=\Lambda^{1/2}z_\ell\) by disjoint Fourier support.
The multiplier \(D\Lambda^{-1/2}\) on the test equals its nonzero
value at \(-k_0\) up to an \(O(\ell)\) error in \(L^2\).
Pairing with \(f_p(\operatorname{Re}z_\ell)\), whose \(L^2\)-norm
is \(O(\ell^{p/2})\), and changing variables to \(y=\ell x\)
leaves the periodic average of
\(e^{it}|\cos t|^p\cos t\) against \(\phi(y)^{p+2}\).
Its real mean is strictly positive. Thus for \(p<6\) no uniform
bound of \(k^{3/2}\|b_\perp\|_{L^2}\) by \(C\mathscr D(z)\) holds,
even on a small energy ball (multiply all packets by a fixed small
constant). The signed error \(\mathcal R_\perp\) in
\eqref{gap:derivative} may be smaller. The packet calculation leaves
open the scattering problem for the remaining powers.

\end{document}